\documentclass[11pt,a4paper]{article}
\usepackage[utf8]{inputenc}
\usepackage[T1]{fontenc}
\usepackage[english]{babel}
\usepackage{amsmath}
\usepackage{amsfonts}
\usepackage{amssymb}
\usepackage{graphicx}
\usepackage{makeidx}
\usepackage{dsfont}
\usepackage{xcolor}
\usepackage{array,multirow,tabularx}
\author{Victor \textsc{Marx} \thanks{
Université Côte d’Azur, CNRS, Laboratoire J.A. Dieudonné UMR 7351, France.       \newline
E-mail: victor.marx@unice.fr }}

\title{A new approach for the construction of a Wasserstein diffusion}
\date{December 2018}
\usepackage[left=2.5cm,right=2.5cm, top=3cm,bottom=3cm]{geometry}

\usepackage{amsthm}
\newtheorem{théo}{Theorem}[section]
\theoremstyle{plain}
\newtheorem{prop}[théo]{Proposition}
\newtheorem{coro}[théo]{Corollary}
\newtheorem{lemme}[théo]{Lemma}
\newtheorem*{lemme*}{Lemma}
\theoremstyle{remark}
\newtheorem{rem}[théo]{Remark}

\theoremstyle{definition}
\newtheorem{déf}[théo]{Definition}

\newcommand{\eps}{\varepsilon}
\renewcommand{\P}[1]{\ensuremath{\mathbb P \left[ #1 \right]}}

\newcommand{\E}[1]{\ensuremath{\mathbb E \left[ #1 \right]}}

\renewcommand{\leq}{\leqslant}
\renewcommand{\geq}{\geqslant}
\newcommand{\Leb}{\operatorname{Leb}}
\newcommand{\Q}{\mathds Q}
\newcommand{\N}{\mathds N}
\newcommand{\R}{\mathds R}
\newcommand{\A}{\mathcal A}
\newcommand{\B}{\mathcal B}
\newcommand{\umax}{u_{\max}}
\newcommand{\umin}{u_{\min}}
\newcommand{\umed}{u_{\operatorname{med}}}
\newcommand{\vmax}{V_{\max}}
\newcommand{\vmin}{V_{\min}}
\newcommand{\vmed}{V_{\operatorname{med}}}

\usepackage{fancyhdr}
\begin{document}
\maketitle
\begin{abstract}
%We propose in this paper a new construction of the diffusion on the Wasserstein space of probability measures introduced by Konarovskyi~\cite{konarovskyi2017system}. This diffusion has interesting properties analagous to those of a standard Euclidean Brownian motion and it is a close relative to the Wasserstein diffusion on $\mathcal P_2(\R)$ constructed by von Renesse and Sturm in~\cite{renesse_entropic_2009}.
%Whereas Konarovskyi's process is a system of mass-carrying coalescing particles, we introduce a family of smooth processes describing the dynamics of short-range interacting particles. Letting the range of interactions tend to zero, our main result is the convergence of this mollified sequence to a limit process having the features of a Wasserstein diffusion. 

We propose in this paper a construction of a diffusion process on the space $\mathcal P_2(\R)$ of probability measures with a second-order moment. This process was introduced in several papers by Konarovskyi (see e.g.~\cite{konarovskyi17system}) and consists of the limit as $N$ tends to $+\infty$ of a system of $N$ coalescing and mass-carrying particles. It has properties analogous to those of a standard Euclidean Brownian motion, in a sense that we will precise in this paper. We also compare it to the Wasserstein diffusion on $\mathcal P_2(\R)$ constructed by von Renesse and Sturm in~\cite{vrenessesturm09}. We obtain that process by the construction of a system of particles having short-range interactions and by letting the range of interactions tend to zero. This construction can be seen as an approximation of the singular process of Konarovskyi by a sequence of smoother processes. 
\medskip

 
\end{abstract}

\paragraph*{Acknowledgment.} The author is grateful to an anonymous referee for the important remarks and the very detailed report.

\section{Introduction}

This paper introduces a new approach to construct the stochastic diffusion process studied by Konarovskyi (\cite{konarovskii11}, \cite{konarovskyi17behavior}, \cite{konarovskyi17system}, \cite{konarovskyivonrenesse18}). It is a close relative to the Wasserstein diffusion, introduced by von Renesse and Sturm~\cite{vrenessesturm09}. 
Our interest is to construct an analogous process to the Euclidean Brownian motion taking values on the Wasserstein space $\mathcal P_2(\R)$, defined as the set of probability measures on $\R$ having a second-order moment. 
%Our aim, among others, is to use this process to add a noise to deterministic dynamics on $\mathcal P_2(\R)$, as for example Fokker-Planck equations. 

In~\cite{vrenessesturm09}, von Renesse and Sturm construct a strong Markov process called \emph{Wasserstein diffusion} on $\mathcal P_2(M)$, for $M$ equal either to the interval $[0,1]$ or to the circle $\mathds S^1$. 
Two major features of that process illustrate the analogy with the standard Brownian motion on a Euclidean space. 
First, the energy of the martingale part of the Wasserstein diffusion has the same form as that of a $k$-dimensional standard Brownian motion, up to replacing the Euclidean norm on $\R^k$ by the $L_2$-Wasserstein distance: 
\[
d_W(\mu,\nu)=\inf \E{|X-Y|^2}^{1/2},
\]
where the infimum is taken over all couplings of two random variables $X$ and $Y$ such that $X$ (resp. $Y$) has law $\mu$ (resp. $\nu$). 
It should be noticed that the geometry of $\mathcal P_2(M)$, equipped with the Wasserstein distance, for $M$ a Euclidean space, was the subject of fundamental studies conducted by Ambrosio, Gigli, Savare, Villani, Lions and many others~(\cite{ambrosiogiglisavare08}, \cite{cardaliaguet13}, \cite{lions_college}, \cite{villani03}, \cite{villani09}), which led to important improvements in optimal transport theory.
Second, the transition costs of the Wasserstein diffusion are given by a Varadhan formula (see~\cite{vrenessesturm09}, Corollary~7.19).
The formula is identical to the Euclidean case, up to the replacement of the Euclidean norm by $d_W$. 

Although the existence of a Wasserstein diffusion was initially proven by von Renesse and Sturm using Dirichlet processes and the theory of Dirichlet forms (see~\cite{fukushimaetal11}), it can also be obtained as a limit of finite-dimensional systems of interacting particles, see~\cite{andresvonrenesse10} and~\cite{sturm14}. 
Nevertheless, we will focus in this paper on a construction of a system of particles which seems more natural and simpler and which is due to Konarovskyi in~\cite{konarovskii11} and~\cite{konarovskyi17system}.

\subsection{Konarovskyi's model}

In~\cite{konarovskyi17system}, Konarovskyi studies a simple system of $N$ interacting and coalescing particles and proves its convergence to an infinite-dimensional process which has the features of a diffusion on the $L_2$-Wasserstein space of probability measures (see also~\cite{konarovskii11}, \cite{konarovskyi17behavior}, \cite{konarovskyivonrenesse18}). 
However, even if it has common properties with the diffusion of von Renesse and Sturm, there are also important differences between the two processes. An outstanding property of Konarovskyi's process is the fact that, for a large family of initial measures, it takes values in the set of measures with finite support for each time $t>0$ (see~\cite{konarovskyi17behavior}), whereas the values of the Wasserstein diffusion of von Renesse and Sturm are probability measures on $[0,1]$ with no absolutely continuous part and no discrete part.

The model introduced by Konarovskyi is a modification of the Arratia flow, also called Coalescing Brownian flow, introduced by Arratia~\cite{arratia79} and subject of many interest, among others in~\cite{dorogovtsev04}, \cite{lejanraimond04}, \cite{norristurner15}, \cite{piterbarg98}. It consists of Brownian particles starting at discrete points of the real line and moving independently until they meet another particle: when they meet, they stick together to form a single Brownian particle. 

In his model (see~\cite{konarovskyi17system}), Konarovskyi adds a mass to every particle: at time $t=0$, $N$~particles, denoted by $(x_k(t))_{k \in \{1,\dots,N\}}$, start from $N$ points regularly distributed on the unit interval $[0,1]$, and each particle has a mass equal to $\frac{1}{N}$. When two particles stick together, they form as in the standard Arratia flow a unique particle, but with a mass equal to the sum of the two incident particles. Furthermore, the quadratic variation process of each particle is assumed to be inversely proportional to its mass. In other words, the heavier a particle is, the smaller its fluctuations are. 

Konarovskyi constructs an associated process $(y^N(u,t))_{u \in [0,1],t \in [0,T]}$ in the set $\mathcal D([0,1], \mathcal C[0,T])$ of càdlàg functions on $[0,1]$ taking values in $\mathcal C[0,T]$ by setting: 
\[
y^N(u,t):= \sum_{k =1}^N x_k(t) \mathds 1_{\{ u\in [\frac{k-1}{N},\frac{k}{N})\}}+x_N(t)\mathds 1_{\{u=1\}}.
\]
In other words, $y^N(\cdot,t)$ is the quantile function associated to the empirical measure $\frac{1}{N}\sum_{k=1}^N \delta_{x_k(t)}$. 
Konarovskyi showed in~\cite{konarovskyi17system} that the sequence $(y^N)_{N\geq 1}$ is tight in $\mathcal D([0,1], \mathcal C[0,T])$. Hence, by passing to the limit upon a subsequence, there exists a process $(y(u,t))_{u \in [0,1],t \in [0,T]}$ belonging to $\mathcal D([0,1], \mathcal C[0,T])$ and satisfying the following four properties: 
\begin{itemize}
\label{propriétés (i)-(iv)}
\item[$(i_0)$] for all $u \in [0,1]$, $y(u,0)=u$; 
\item[$(ii)$] for all $u \leq v$, for all $t \in [0,T]$, $y(u,t)\leq y(v,t)$;
\item[$(iii)$] for all $u \in [0,1]$, $y(u,\cdot)$ is a square integrable continuous martingale relatively to the filtration $(\mathcal F_t)_{t \in [0,T]}:=(\sigma (y(v,s),v\in [0,1],s\leq t))_{t \in [0,T]}$;
\item[$(iv)$] for all $u,u' \in [0,1]$, 
\[
\langle y(u,\cdot),y(u',\cdot)\rangle_t = \int_0^t \frac{\mathds 1_{\{ \tau_{u,u'} \leq s \}}}{m(u,s)}\mathrm ds,
\]
where $m(u,t)=\int_0^1 \mathds 1_{\{\exists s \leq t: \, y(u,s)=y(v,s)\}} \mathrm dv$ and $\tau_{u,u'}=\inf\{t\geq 0: y(u,t)=y(u',t)\}\wedge T$. 
\end{itemize}

By transporting the Lebesgue measure on $[0,1]$ by the map $y(\cdot,t)$, we obtain a measure-valued process $(\mu_t)_{t \in [0,T]}$ defined by: $\mu_t:= \Leb|_{[0,1]} \circ y(\cdot,t)^{-1}$. In other words, $u\mapsto y(u,t)$ is the quantile function associated to $\mu_t$. An important feature of this process is that for each positive $t$, $\mu_t$ is an atomic measure with a finite number of atoms, or in other words that $y(\cdot,t)$ is a step function. 

More generally, Konarovskyi proves in~\cite{konarovskyi17behavior} that this construction also holds for a greater family of initial measures $\mu_0$. He constructs a process $y^g$ in $\mathcal D([0,1],\mathcal C[0,T])$ satisfying $(ii)-(iv)$ and:
\begin{itemize}
\label{propriété (i)}
\item[$(i)$] for all $u \in [0,1]$, $y^g(u,0)=g(u)$,
\end{itemize}
for every non-decreasing càdlàg function $g$ from $[0,1]$ into $\R$ such that there exists $p>2$ satisfying $\int_0^1 |g(u)|^p \mathrm du<\infty$. 
In other words, he generalizes the construction of a diffusion starting from any probability measure $\mu_0$ satisfying $\int_\R |x|^p \mathrm d\mu_0(x)<\infty$ for a certain $p>2$, where $\mu_0=\Leb|_{[0,1]}\circ g^{-1}$, which means that $g$ is the quantile function of the initial measure. The property that $y^g(\cdot,t)$ is a step function for each $t>0$ remains true for this larger class of functions $g$.

The process $y^g$ is said to be \emph{coalescent}:  almost surely, for every $u,v \in [0,1]$ and for every $t \in (\tau_{u,v},T]$, we have $y^g(u,t)=y^g(v,t)$ (recall that $\tau_{u,v}=\inf\{t\geq 0: y^g(u,t)=y^g(v,t)\}\wedge T$). This property is a consequence of~$(ii)$, $(iii)$  and of the fact that for each $t>0$, $y^g(\cdot,t)$ is a step function (see~\cite[p.11]{konarovskyivonrenesse18}). Therefore, we can rewrite the formula for the mass as follows: 
\begin{align*}
m^g(u,t)= \int_0^1 \mathds 1_{\{\exists s \leq t: \, y^g(u,s)=y^g(v,s)\}} \mathrm dv
=\int_0^1 \mathds 1_{\{y^g(u,t)=y^g(v,t) \}} \mathrm dv.
\end{align*}

Moreover, we can compare the diffusive properties of the process~$(\mu_t)_{t \in [0,T]}$ in the Wasserstein space $\mathcal P_2(\R)$ with the Wasserstein diffusion of von Renesse and Sturm. To that extent and thanks to Lions' differential calculus on $\mathcal P_2(\R)$~(\cite{lions_college},\cite{cardaliaguet13}), we give in Appendix~\ref{appendix} an Itô formula on $\mathcal P_2(\R)$ for the process $(\mu_t)_{t \in [0,T]}$ in order to describe the energy of the martingale part of this diffusion. 
Appendix~\ref{appendix} also contains a small introduction to the differentiability on $\mathcal P_2(\R)$ in the sense of Lions. 

\subsection{Approximation of a Wasserstein diffusion}

In this paper, we propose a new method to construct a process $y$ satisfying properties $(i)$-$(iv)$, 
by approaching $y$ by a sequence of smooth processes. Finding smooth approximations of processes having singularities has already led to interesting results, typically in the case of the Arratia flow. Piterbarg~\cite{piterbarg98} shows that the Coalescing Brownian flow is the weak limit of isotropic homeomorphic flows in some space of discontinuous functions, and deduces from the properties of the limit process a careful description of contraction and expansion regions of homeomorphic flows. 
Dorogovtsev's approximation~\cite{dorogovtsev04} is based on a representation of the Arratia flow with a Brownian sheet. 

We propose an adaptation of Dorogovtsev's idea in the case of Wasserstein diffusions. First, we show that a process $y$ satisfying~$(i)$-$(iv)$ admits a representation in terms of a Brownian sheet; we refer to the lectures of Walsh~\cite{walsh86} for a complete introduction to Brownian sheet and to Section~\ref{sec sing repr} for the characterization of Brownian sheet which we use in this paper. 
\begin{théo}
\label{théo 1}
Let $g:[0,1] \to \R$ be a non-decreasing and càdlàg function such that there exists $p>2$ satisfying $\int_0^1 |g(u)|^p \mathrm du <+\infty$. 
Let $y$ be a process in $L_2([0,1],\mathcal C[0,T])$ that satisfies conditions $(i)$, $(ii)$, $(iii)$ and $(iv)$. 
There exists a Brownian sheet $w$ on $[0,1]\times [0,T]$ such that for all $u \in [0,1]$ and $t \in [0,T]$:
\begin{equation}
\label{eq brownian sheet}
y(u,t)=g(u)+\int_0^t \!\!\int_0^1 \frac{\mathds 1_{\{y(u,s)=y(u',s)\}}}{m(u,s)}  \mathrm dw(u',s),
\end{equation}
where $m(u,s)=\displaystyle\int_0^1 \mathds 1_{\{y(u,s)=y(v,s)\}} \mathrm dv$.
\end{théo}

\begin{rem}
We refer to Appendix~\ref{appendix} to justify the use of the term "Wasserstein diffusion" for a process satisfying equation~\eqref{eq brownian sheet}. Indeed, we can  write an Itô formula for this process for a smooth function $u: \mathcal P_2(\R)\to \R$. As in the case of the standard Euclidean Brownian motion, the quadratic variation of the martingale term is proportional to the square of the gradient of $u$, in the sense of Lions' differential calculus on $\mathcal P_2(\R)$, which is the same as the differential calculus on the Wasserstein space (see~\cite[Section 5.4]{carmonadelarue18}). 
\end{rem}

The aim of this paper is to construct a sequence of smooth processes approaching $y$ in the space $L_2([0,1],\mathcal C[0,T])$. Therefore, we use the representation~\eqref{eq brownian sheet} in terms of a Brownian sheet of $y$ and, given a positive parameter $\sigma$, we replace in the latter representation the indicator functions by a smooth function $\varphi_\sigma$ equal to 1 in the neighbourhood of 0 and whose support is included in the interval $\left[-\frac{\sigma}{2}, \frac{\sigma}{2}\right]$ of small diameter $\sigma$.  Fix $\sigma>0$ and $\eps>0$.
Given a Brownian sheet $w$ on $[0,1]\times[0,T]$, we prove the existence of a process $y_{\sigma,\eps}$ satisfying:
\begin{equation}
\label{def y sigma}
y_{\sigma,\eps}(u,t)=g(u)+\int_0^t\!\! \int_0^1 \frac{\varphi_\sigma(y_{\sigma,\eps}(u,s)-y_{\sigma,\eps}(u',s))}{\eps+m_{\sigma,\eps}(u,s)}  \mathrm dw(u',s),
\end{equation}
where $m_{\sigma,\eps}(u,s):= \int_0^1 \varphi_\sigma^2(y_{\sigma,\eps}(u,s)-y_{\sigma,\eps}(v,s))\mathrm dv$ can be seen as a kind of mass of particle $y_{\sigma,\eps}(u)$ at time~$s$. Remark that, due to the fact that the support of $\varphi_\sigma$ is small, only the particles located at a distance lower than $\frac{\sigma}{2}$ of particle $u$ at time $s$ are taken into account in the computation of the mass $m_{\sigma,\eps}(u,s)$.

The smooth process $(y_{\sigma,\eps}(u,t))_{u\in [0,1],t \in [0,T]}$ offers several advantages. First, we are able to construct a strong solution $(y_{\sigma,\eps},w)$ to equation~\eqref{def y sigma}, whereas in equation~\eqref{eq brownian sheet}, we do not know if, given a Brownian sheet $w$, there exists an adapted solution~$y$. 
Second, in Konarovskyi's process, the question of uniqueness of a solution to~$(1)$, even in the weak sense, or equivalently the question of uniqueness of a process in $L_2([0,1],\mathcal C[0,T])$ satisfying conditions $(i)$-$(iv)$, remains open. Here, pathwise uniqueness holds for equation~\eqref{def y sigma}. {Moreover, the measure-valued process $(\mu_t^{\sigma,\eps})_{t \in [0,T]}$ associated to the process of quantile functions $(y_{\sigma,\eps}(\cdot,t))_{t \in [0,T]}$ does generally no longer consist of atomic measures. For example, if $g(u)=u$, $(\mu_t^{\sigma,\eps})_{t \in [0,T]}$ is a process of absolutely continuous measures with respect to the Lebesgue measure.
%This smoothness property of $(\mu_t^{\sigma,\eps})_{t \in [0,T]}$ makes the study of this process easier and makes it conceivable to observe regularizing effects. 

Let $L_2[0,1]$ be the usual space of square integrable functions from $[0,1]$ to $\R$, and $(\cdot,\cdot)_{L_2}$ the usual scalar product. We denote by $L^\uparrow_2[0,1]$ the set of functions $f\in L_2[0,1]$ such that there exists a non-decreasing and therefore \emph{càdlàg} (\emph{i.e.} right-continuous with left limits everywhere) element in the equivalence class of $f$. 
Let $\mathcal D((0,1),\mathcal C[0,T])$ be the space of right-continuous $\mathcal C[0,T]$-valued functions with left limits, equipped with the Skorohod metric.

We follow the definition given in~\cite[p.21]{gawareckimandrekar11}: 
\begin{déf}
\label{def martingale}
An $(\mathcal F_t)_{t \in [0,T]}$-adapted process $M$ is an \emph{$L^\uparrow_2[0,1]$-valued $(\mathcal F_t)_{t \in [0,T]}$-martingale} if $M_t$ belongs to $L^\uparrow_2[0,1]$ for each $t \in [0,T]$, if $\E{\|M_t\|_{L_2}}< \infty$ and if for each $h \in L_2[0,1]$, $(M_t,h)_{L_2}$ is a real-valued $(\mathcal F_t)_{t \in [0,T]}$-martingale. 
The martingale is said to be \emph{square integrable} if for each $t \in [0,T]$, $\E{\|M_t\|_{L_2}^2}<+\infty$, and \emph{continuous} if the process $t \mapsto M_t$ is a continuous function from $[0,T]$ to $L_2[0,1]$.  
\end{déf}

Let us denote by $\mathcal  L^\uparrow_{2+}[0,1]$ the set of all non-decreasing and càdlàg functions $g:[0,1]\rightarrow\overline{\R}$, where $\overline{\R}:= \R\cup \{-\infty,+\infty\}$, such that there exists $p>2$ for which $\int_0^1 |g(u)|^p \mathrm du <+\infty$. Let $\Q_+=\Q \cap [0,1]$. 
The following Theorem states the convergence of the mollified sequence $(y_{\sigma,\eps})_{\sigma>0,\eps>0}$ to a limit process satisfying properties~$(i)-(iv)$.
It uses the framework introduced by Konarovskyi in~\cite{konarovskyi17behavior}:

\begin{théo}
\label{théo 2}
Let $g \in \mathcal  L^\uparrow_{2+}[0,1]$. For each positive $\sigma$ and $\eps$, there exists a solution $y_{\sigma,\eps}$ to equation~\eqref{def y sigma} such that  $(y_{\sigma,\eps}(u,t))_{u \in [0,1], t\in [0,T]}$ belongs to $L_2([0,1], \mathcal C[0,T])$ and almost surely, for each $t \in [0,T]$, $y_{\sigma,\eps}(\cdot,t) \in L^\uparrow_2[0,1]$. 

Furthermore, up to extracting a subsequence, the sequence $(y_{\sigma,\eps})_{\eps>0}$ converges in distribution in $L_2([0,1],\mathcal C[0,T])$ for every $\sigma \in \Q_+$ as $\eps$ tends to $0$ to a limit $y_\sigma$ and the sequence $(y_\sigma)_{\sigma \in \Q_+}$ converges in distribution in $L_2([0,1],\mathcal C[0,T])$ as $\sigma$ tends to $0$ to a limit $y$. Let $Y(t):=y(\cdot,t)$. Then $(Y(t))_{t \in [0,T]}$ is a $L^\uparrow_2[0,1]$-valued process such that:
\begin{itemize}
\item[$(C1)$] $Y(0)=g$;

\item[$(C2)$] $(Y(t))_{t \in [0,T]}$ is a square integrable continuous $L^\uparrow_2[0,1]$-valued $(\mathcal F_t)_{t \in [0,T]}$-martingale, where $\mathcal F_t:=\sigma(Y(s),s\leq t)$; 

\item[$(C3)$] almost surely, for every $t>0$, $Y(t)$ is a step function, \emph{i.e.} there exist $n \geq 1$,  $0=a_1<a_2<\dots<a_n<a_{n+1}=1$ and $z_1<z_2<\dots<z_n$ such that for all $u\in [0,1]$
\begin{align*}
Y(t)(u)=y(u,t)=\sum_{k=1}^n z_k \mathds 1_{\{u \in [a_k,a_{k+1})\}}+z_n \mathds 1_{\{u=1\}};  
\end{align*}

\item[$(C4)$]
$y$ belongs to $\mathcal D((0,1),\mathcal C[0,T])$ and for every $u \in (0,1)$, $y(u, \cdot)$ is a square integrable and continuous $(\mathcal F_t)_{t \in [0,T]}$-martingale and 
\begin{align*}
\P{\forall u,v \in (0,1), \forall s \in [0,T],y(u,s)=y(v,s) \text{ implies } \forall t \geq s, y(u,t)=y(v,t)}=1;
\end{align*}

\item[$(C5)$]
for each $u$ and $u'$ in $(0,1)$, 
\begin{align*}
\langle y(u,\cdot),y(u',\cdot)\rangle_t = \int_0^t \frac{\mathds 1_{\{ \tau_{u,u'} \leq s \}}}{m(u,s)}\mathrm ds,
\end{align*}
where $m(u,s)=\displaystyle\int_0^1 \mathds 1_{\{ y(u,s)=y(v,s) \}} \mathrm dv$ and $\tau_{u,u'}=\inf\{t\geq 0: y(u,t)=y(u',t)\}\wedge T$.
\end{itemize}
\end{théo}

\begin{rem}
More precisely, the filtration $(\mathcal F_t)_{t \in [0,T]}$ is given by:
\begin{align*}
\mathcal F_t=\sigma((Y(s),h)_{L_2}, s \leq t, h \in L_2[0,1]).
\end{align*}
\end{rem}

\begin{rem}
By property~$(C4)$, the limit process $y$ is said to be coalescent: if for a certain time~$t_0$, two particles $y(u,t_0)$ and $y(v,t_0)$ coincide, then they move together forever, \textit{i.e.} $ y(u,t)=y(v,t)$ for every $t \geq t_0$. 
\end{rem}

\begin{figure}[hbt]\centering
\includegraphics[width=16cm,height=12cm]{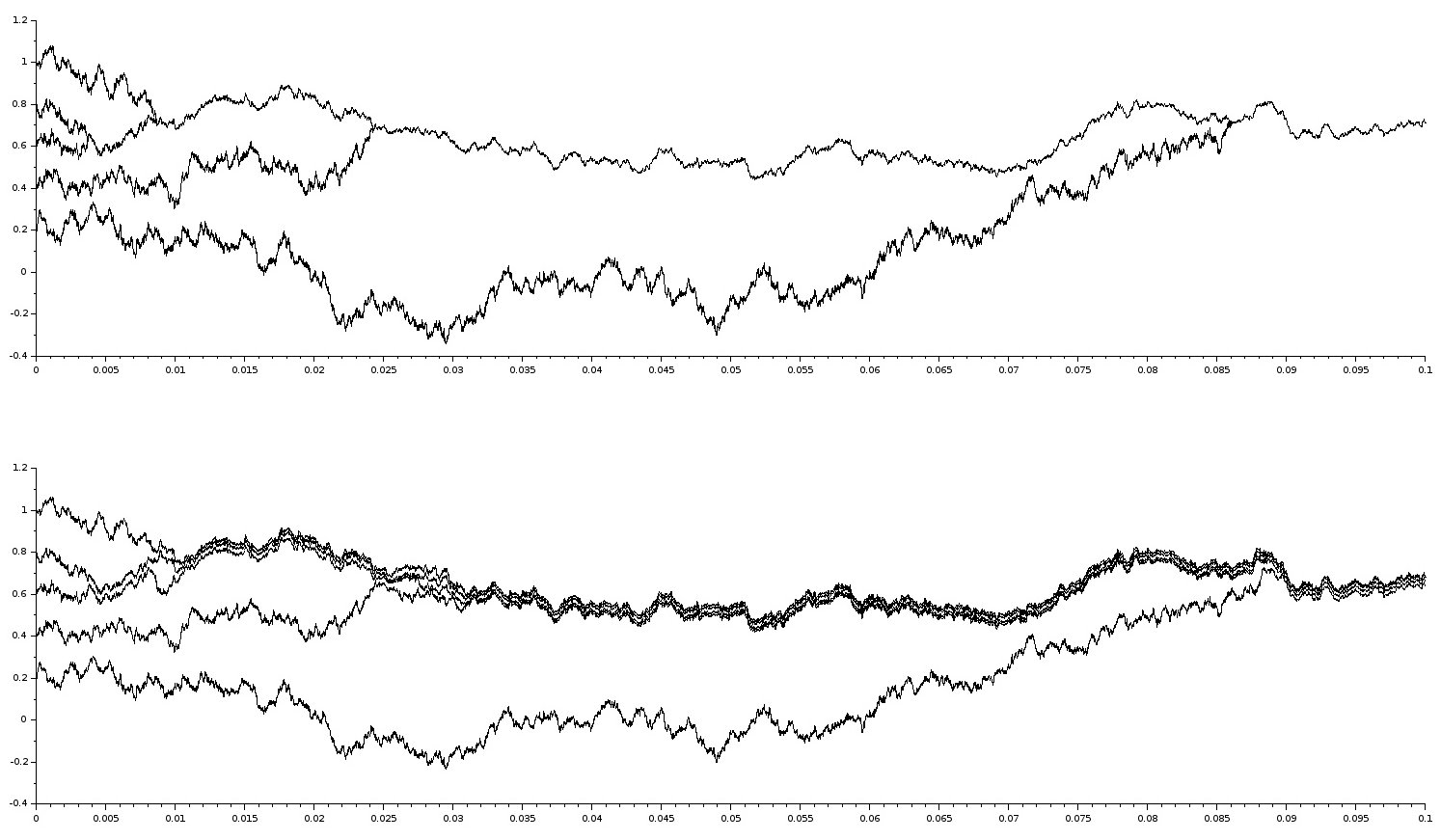}
\caption{\footnotesize{Two simulations, based on the same underlying Brownian sheet, for the limit process $(\mu_t)_{t \in [0,T]}$ \emph{(on top)} and for the process $(\mu_t^{\sigma,\eps})_{t \in [0,T]}$ with positive $\sigma$ and $\eps$ \emph{(on bottom)}. The horizontal axis represents time. On the vertical axis, we put the position of the particles (initially, we took five particles on $[0,1]$).} 
}
\label{figure}
\end{figure} 

It is interesting to wonder how the coalescence property of the process $y$ translates to its smooth approximation $y_{\sigma,\eps}$: two paths $(y_{\sigma,\eps}(u,t))_{t \in [0,T]}$ and 
$(y_{\sigma,\eps}(v,t))_{t \in [0,T]}$, starting from two distinct points $g(u)$ and $g(v)$, do not meet, which means that $y_{\sigma,\eps}(\cdot,t)$ is non-decreasing for each fixed~$t$. If $y_{\sigma,\eps}(u,\cdot)$ and $y_{\sigma,\eps}(v,\cdot)$ get close enough, at distance smaller than $\sigma$, they begin to interact and to move together, whereas as long as they remain at distance greater than $\sigma$, they move "independently": more precisely, the covariation $\langle y_{\sigma,\eps}(u,\cdot),y_{\sigma,\eps}(v,\cdot)\rangle_t$ is equal to zero for every time $t \leq \tau^\sigma_{u,v}:=\inf\{ s \geq 0: |y_{\sigma,\eps}(u,s)-y_{\sigma,\eps}(v,s)| \leq \sigma\}$ (see figure~\ref{figure}).

\subsection*{Organisation of the article}
We begin in Section~\ref{sec sing repr} by proving Theorem~\ref{théo 1}, which states that a process $y$ satisfying properties $(i)$-$(iv)$ admits a representation in terms of a Brownian sheet.
In Section~\ref{sec construction}, given a two-dimensional Brownian sheet, we prove the existence of a smooth process in the space $L_2([0,1],\mathcal C[0,T])$ intended to approach Konarovskyi's process of coalescing particles. This smooth process can be seen as a cloud of point-particles interacting with all the particles at a distance smaller than $\sigma$, and in which two particles have independent trajectories conditionally to the fact that the distance between them is greater than $ \sigma$. When the distance becomes smaller than $\sigma$, both trajectories are correlated, mimicking the coalescence property. 

Section~\ref{sec convergence} is devoted to the proof of convergence when the parameter $\eps$ and the range of interaction $\sigma$ tend to zero, using a tightness criterion in $L_2([0,1],\mathcal C[0,T])$. In Section~\ref{sec properties}, we study the stochastic properties of the limit process, including the convergence of the mass process. The aim of this final part is to prove that the limit process $y$ satisfies properties $(C1)$-$(C5)$ of Theorem~\ref{théo 2}, in other words that our sequence of short-range interaction processes converges in distribution to the process of coalescing particles. 

In Appendix~\ref{appendix}, we give an Itô formula in the Wasserstein space for the limit process $y$, after having recalled some basic definitions and properties of Lions' differential calculus on $\mathcal P_2(\R)$.

\section{Singular representation of the process $y$}
\label{sec sing repr}

Let $(\Omega, \mathcal F,  \mathbb P)$ be a probability space. 
Let us consider on $(\Omega, \mathcal F,  \mathbb P)$ a random process $y \in L_2((0,1),\mathcal C[0,T])$ satisfying properties $(i)$-$(iv)$.
We refer to~\cite{konarovskyi17behavior} for a comprehensive construction of $y$. We will give another one later in this paper. 

The aim of this Paragraph is to prove Theorem~\ref{théo 1}. Before that, we recall the definition of a Brownian sheet given by Walsh in~\cite[p.269]{walsh86}. 
Let $(E, \mathcal E,\nu)$ be a Euclidean space equipped with Lebesgue measure. A white noise based on $\nu$ is a random set function $W$ on the sets $A \in \mathcal E$ of finite $\nu$-measure such that
\begin{itemize}
\item  $W(A)$ is a $\mathcal N(0,\nu(A))$ random variable,
\item if $A \cap B = \emptyset$, then $W(A)$ and $W(B)$ are independent and $W(A \cap B)=W(A)+W(B)$.
\end{itemize}
Let $T>0$. Consider $E=[0,1]\times [0,T]$ and $\nu$ the associated Lebesgue measure. The \emph{Brownian sheet} $w$ on $[0,1]\times[0,T]$ associated to the white noise $W$ is the process $(w(u,t))_{u \in [0,1] \times [0,T]}$ defined by $w(u,t):=W((0,u] \times (0,t])$. 

Define the filtration $(\mathcal G_t)_{t\in [0,T]}$ by $\mathcal G_t:=\sigma(w(u,s), u\in [0,1], s\leq t)$. Then in particular,
\begin{itemize}
\item [(i)] for each $(\mathcal G_t)_{t\in[0,T]}$-progressively measurable function $f$ defined on $[0,1]\times[0,T]$ such that $\int_0^T \int_0^1 f^2(u,s)\mathrm du\mathrm ds<+\infty$ almost surely, the process $\left(\int_0^t \int_0^1 f(u,s)\mathrm dw(u,s)\right)_{t\in[0,T]}$ is a local martingale  (we often write $\mathrm dw(u,s)$ instead of $w(\mathrm du, \mathrm ds)$);
\item [(ii)] for each $f_1$ and $f_2$ satisfying the same conditions as $f$, 
\begin{align*}
\langle \int_0^\cdot \!\! \int_0^1 f_1(u,s)\mathrm dw(u,s),\int_0^\cdot\!\! \int_0^1 f_2(u,s)\mathrm dw(u,s) \rangle_t
&= \int_0^t\!\! \int_0^1 f_1(u,s)f_2(u,s) \mathrm du\mathrm ds.
\end{align*}
\end{itemize}
By Lévy's characterization of the Brownian motion, a process $w$ satisfying~$(i)$ and~$(ii)$ is a Brownian sheet. 
Let us now prove Theorem~\ref{théo 1}. 

\begin{proof}[Proof (Theorem~\ref{théo 1})]
We take a Brownian sheet $\eta$ on $[0,1]\times [0,T]$ independent of the process $y$, constructed by possibly extending the probability space $(\Omega, \mathcal F,  \mathbb P)$. 
Then, we define $(w(u,t))_{u\in [0,1],t\in[0,T]}$ by $w(0,\cdot)\equiv 0$, $w(\cdot,0)\equiv 0$ and:
\begin{equation*}
w(\mathrm du,\mathrm dt)=\eta(\mathrm du,\mathrm dt)+y(u,\mathrm dt)\mathrm du - \frac{1}{m(u,t)}\int_0^1 \mathds 1_{\{y(u,t)=y(u',t)  \}}\eta (\mathrm du',\mathrm dt)\mathrm du.
\end{equation*}
We denote by $\mathcal{H}_t$ the filtration $\sigma((y(u,s))_{u \in [0,1],s\leq t},(\eta(u,s))_{u \in [0,1],s\leq t})$.

In order to prove that $w$ is an $(\mathcal H_t)_{t\in[0,T]}$-Brownian sheet on $[0,1]\times[0,T]$, let us consider two $(\mathcal H_t)_{t\in[0,T]}$-progressively measurable functions $f_1$ and $f_2$ and compute, using independence of $\eta$ and $y$:
\begin{equation*}
\langle \int_0^\cdot \!\!\int_0^1 f_1(u,s) \mathrm dw(u,s),\int_0^\cdot \!\! \int_0^1 f_2(v,s) \mathrm dw(v,s) \rangle_t = V_1+V_2-V_3-V_4+V_5, 
\end{equation*}
where 
\begin{equation*}
V_1:=\langle \int_0^\cdot \!\!\int_0^1 f_1(u,s) \mathrm d\eta(u,s),\int_0^\cdot\!\! \int_0^1 f_2(v,s) \mathrm d\eta(v,s) \rangle_t =\int_0^t \!\! \int_0^1 f_1(u,s)f_2(u,s) \mathrm du \mathrm ds,
\end{equation*}
since $\eta$ is an $(\mathcal H_t)_{t\in[0,T]}$-Brownian sheet;
\begin{align*}
V_2&:=\langle \int_0^\cdot \!\!\int_0^1 f_1(u,s) \mathrm dy(u,s) \mathrm du,\int_0^\cdot\!\! \int_0^1 f_2(v,s) \mathrm dy(v,s)\mathrm dv \rangle_t \\
&=\int_0^t \!\!\int_0^1 \!\!\int_0^1 f_1(u,s)f_2(v,s) \frac{\mathds 1_{\{ y(u,s)=y(v,s) \}}}{m(u,s)}\mathrm du \mathrm dv \mathrm ds,
\end{align*}
using property $(iv)$ of process $y$;
\begin{align*}
V_3&:=\langle \int_0^\cdot\!\! \int_0^1 f_1(u,s) \mathrm d\eta(u,s) ,\int_0^\cdot \!\!\int_0^1 \frac{f_2(v,s)}{m(v,s)}\int_0^1 \mathds 1_{\{y(v,s)=y(v',s)  \}} \mathrm d\eta(v',s)\mathrm dv \rangle_t \\
& =\int_0^t\!\! \int_0^1 \!\!\int_0^1  \frac{f_1(u,s)f_2(v,s)}{m(v,s)} \mathds 1_{\{ y(v,s)=y(u,s) \}}\mathrm du  \mathrm dv \mathrm ds= V_2,
\end{align*}
since $m(u,s)=m(v,s)$ whenever $y(u,s)$ is equal to $y(v,s)$. 
By similar computations, 
\begin{equation*}
V_4:=\langle \int_0^\cdot \!\!\int_0^1 \frac{f_1(u,s)}{m(u,s)}\int_0^1 \mathds 1_{\{y(u,s)=y(u',s)  \}} \mathrm d\eta(u',s)\mathrm du,\int_0^\cdot\!\! \int_0^1 f_2(v,s) \mathrm d\eta(v,s) \rangle_t =V_2,
\end{equation*}
and
\begin{align*}
V_5:&=\langle \int\limits_0^\cdot \int\limits_0^1 \frac{f_1(u,s)}{m(u,s)}\int\limits_0^1 \mathds 1_{\{y(u,s)=y(u',s)  \}} \mathrm d\eta(u',s)\mathrm du,\int\limits_0^\cdot \int\limits_0^1 \frac{f_2(v,s)}{m(v,s)}\int\limits_0^1 \mathds 1_{\{y(v,s)=y(v',s)  \}} \mathrm d\eta(v',s)\mathrm dv\rangle_t \\
&= \int_0^t\!\! \int_0^1 \!\!\int_0^1 \!\!\int_0^1 \frac{f_1(u,s)f_2(v,s)}{m(u,s)m(v,s)}\mathds 1_{\{y(u,s)=y(u',s)\}}\mathds 1_{\{y(v,s)=y(u',s)\}}\mathrm du'  \mathrm du \mathrm dv \mathrm ds\\
&= \int_0^t \!\!\int_0^1 \!\!\int_0^1  \frac{f_1(u,s)f_2(v,s)}{m(u,s)^2}\left(\int_0^1\mathds 1_{\{y(u,s)=y(u',s)\}}\mathrm du'\right)\mathds 1_{\{y(u,s)=y(v,s)\}}  \mathrm du \mathrm dv \mathrm ds\\
&= \int_0^t \!\!\int_0^1 \!\!\int_0^1  \frac{f_1(u,s)f_2(v,s)}{m(u,s)}\mathds 1_{\{y(u,s)=y(v,s)\}}  \mathrm du \mathrm dv \mathrm ds=V_2. 
\end{align*}
To sum up, 
\begin{equation*}
\langle \int_0^\cdot \!\!\int_0^1 f_1(u,s) \mathrm dw(u,s),\int_0^\cdot \!\! \int_0^1 f_2(v,s) \mathrm dw(v,s) \rangle_t = V_1=\int_0^t \!\!\int_0^1 f_1(u,s)f_2(u,s) \mathrm du \mathrm ds, 
\end{equation*}
whence $w$ is an $(\mathcal H_t)_{t \in [0,T]}$-Brownian sheet. Finally, we show that $(y,w)$ satisfies equation~(\ref{eq brownian sheet}):
\begin{align}
\int_0^t \!\!\int_0^1 \frac{\mathds 1_{\{y(u,s)=y(u',s)\}}}{m(u,s)}  \mathrm dw(u',s)
&=\int_0^t \!\!\int_0^1 \frac{\mathds 1_{\{y(u,s)=y(u',s)\}}}{m(u,s)}\mathrm d\eta(u',s) \tag{$=: W_1$}\\
&+ \int_0^t \!\!\int_0^1 \frac{\mathds 1_{\{y(u,s)=y(u',s)\}}}{m(u,s)}\mathrm dy(u',s)\mathrm du' \tag{$=: W_2$}\\
&-\int_0^t \!\!\int_0^1 \frac{\mathds 1_{\{y(u,s)=y(u',s)\}}}{m(u,s)} \int_0^1 \frac{\mathds 1_{\{y(u',s)=y(v,s)  \}}}{m(u',s)} \mathrm d\eta (v,s)\mathrm du'\tag{$=: W_3$}.
\end{align}
The result follows from the two below equalities:
\begin{equation*}
W_2=\int_0^t \!\!\int_0^1 \frac{\mathds 1_{\{y(u,s)=y(u',s)\}}}{m(u,s)}\mathrm dy(u,s)\mathrm du'=\int_0^t \mathrm dy(u,s)=y(u,t)-y(u,0)=y(u,t)-g(u);
\end{equation*}
\begin{align*}
W_3&=\int_0^t\!\! \int_0^1 \frac{\mathds 1_{\{y(u,s)=y(u',s)\}}}{m(u,s)} \int_0^1 \frac{\mathds 1_{\{y(u',s)=y(v,s)  \}}}{m(v,s)} \mathrm d\eta (v,s)\mathrm du' \\
&=\int_0^t \!\!\int_0^1 \frac{\mathds 1_{\{y(u,s)=y(v,s)\}}}{m(u,s)m(v,s)} \left(\int_0^1 \mathds 1_{\{y(u',s)=y(v,s) \}}\mathrm du'\right) \mathrm d\eta (v,s) 
=\int_0^t \!\!\int_0^1 \frac{\mathds 1_{\{y(u,s)=y(v,s)\}}}{m(u,s)}  \mathrm d\eta (v,s),
\end{align*}
which implies that $W_3=W_1$ and consequently equation~(\ref{eq brownian sheet}). 
\end{proof}

Therefore, every solution of the martingale problem $(i)$-$(iv)$ has a representation in terms of a Brownian sheet. 
In the next Section, we will construct, given a Brownian sheet, an approximation of the process $y$.

\section{Construction of a process with short-range interactions}
\label{sec construction}

Let $(\Omega, \mathcal F,  \mathbb P)$ be a probability space, on which we define a Brownian sheet $w$ on $[0,1]\times[0,T]$. We associate to that process the filtration $\mathcal G_t:= \sigma(w(u,s),u\in [0,1],s\leq t)$. Up to completing the filtration, we assume that $\mathcal G_0$ contains all the $\mathbb P$-null sets of $\mathcal F$ and that the filtration $(\mathcal G_t )_{t \in [0,T]}$ is right-continuous.

Fix $\sigma>0$ and $\eps >0$. Let  $\varphi_\sigma$ denote a smooth and even function, bounded by 1, equal to~1 on $[0,\frac{\sigma}{3}]$ and equal to 0 on $[\frac{\sigma}{2},+\infty)$. 
Recall that $\mathcal L^\uparrow_{2+}[0,1]$ represents the set of non-decreasing  and càdlàg functions $g:[0,1]\rightarrow \R$ such that there exists $p>2$ satisfying $\int_0^1 |g(u)|^p \mathrm du <+\infty$. 
The aim of this Section is to construct, for each initial quantile function $g \in \mathcal L^\uparrow_{2+}[0,1]$, a square integrable random variable $y^g_{\sigma,\eps}$ taking values in $L_2([0,1],\mathcal C[0,T])$  such that almost surely, for every $t \in [0,T]$, the following equality holds in $L_2[0,1]$:
\begin{equation}
\label{equation avec phi}
y^g_{\sigma,\eps}(\cdot,t)
= g+ \int_0^t\!\!\int_0^1 \frac{\mathds \varphi_\sigma(y^g_{\sigma,\eps}(\cdot,s)-y^g_{\sigma,\eps}(u',s))}
{\eps + \int_0^1\varphi_\sigma^2(y^g_{\sigma,\eps}(\cdot,s)-y^g_{\sigma,\eps}(v,s)) \mathrm dv}    \mathrm dw(u',s).
\end{equation}

\begin{rem}
We add the parameter $\eps$ to the denominator in order to ensure that it is bounded by below. 
We also point out that relation~(\ref{equation avec phi}) has to be compared with equation~(\ref{eq brownian sheet}), 
where $x \mapsto \mathds 1_{\{x=0\}}$ is replaced by the function $\varphi_\sigma$.
\end{rem}

More precisely, we will prove the following Proposition. Recall that $L^\uparrow_2[0,1]$ represents the set of functions $f\in L_2[0,1]$ such that there is a non-decreasing and càdlàg element in the equivalence class of $f$.
\begin{prop}
\label{proprietes A1-A3}
Let $g \in \mathcal L^\uparrow_{2+}[0,1]$. There exists an $L^\uparrow_2[0,1]$-valued process $(Y^g_{\sigma,\eps}(t))_{t \in [0,T]}=(y^g_{\sigma,\eps}(\cdot,t))_{t \in [0,T]}$ such that:
\begin{itemize}
\item[$(A1)$] $Y^g_{\sigma,\eps}(0)=g$;

\item[$(A2)$] $Y^g_{\sigma,\eps}$ is a square integrable continuous $L^\uparrow_2[0,1]$-valued $(\mathcal F^{\sigma,\eps}_t)_{t \in [0,T]}$-martingale, where $\mathcal F^{\sigma,\eps}_t:=\sigma(Y^g_{\sigma,\eps}(s),s\leq t)$;

\item[$(A3)$]
for every $h,k \in L_2[0,1]$, 
\begin{align*}
\langle (Y^g_{\sigma,\eps},h)_{L_2},(Y^g_{\sigma,\eps},k)_{L_2}\rangle_t = \int_0^t \!\! \int_0^1 \!\! \int_0^1 h(u)k(u') \frac{m^g_{\sigma,\eps}(u,u',s)}{(\eps+m^g_{\sigma,\eps}(u,s))(\eps+m^g_{\sigma,\eps}(u',s))}\mathrm du\mathrm du' \mathrm ds,
\end{align*}
where $m^g_{\sigma,\eps}(u,u',s) =\int_0^1 \varphi_\sigma (y^g_{\sigma,\eps}(u,s)-y^g_{\sigma,\eps}(v,s)) \varphi_\sigma (y^g_{\sigma,\eps}(u',s)-y^g_{\sigma,\eps}(v,s)) \mathrm dv$ and \newline $m^g_{\sigma,\eps}(u,s)=\int_0^1 \varphi_\sigma^2(y^g_{\sigma,\eps}(u,s)-y^g_{\sigma,\eps}(v,s))\mathrm dv$. 
\end{itemize}
\end{prop}

\subsection{Existence of an approximate solution}

Denote by $\mathcal M$ the set of random variables  $z\in L_2(\Omega,\mathcal C([0,T], L_2(0,1)))$ such that $(z(\omega,\cdot,t))_{t \in [0,T]}$ is a $(\mathcal G_t)_{t \in [0,T]}$-progressively measurable process with values in $L_2(0,1)$. We consider the following norm on $\mathcal M$:
\[\| z \|_\mathcal M=\E{ \sup_{t \leq T}\int_0^1 | z(u,t) |^2 \mathrm du}^{1/2}.\]

Throughout this Section, $\sigma$ and $\eps$ are two \emph{fixed} positive numbers. To begin, we want to prove that the map $\psi: \mathcal M \rightarrow \mathcal M$, defined below, admits a unique fixed point. 
Fix $g \in \mathcal L^\uparrow_{2+}[0,1]$ an initial quantile function. 
For all $z \in \mathcal M$, define:
\begin{equation}
\psi(z)(\omega, u,t):=g(u) +\int_0^t \!\! \int_0^1 \frac{\varphi_\sigma(z(\omega,u,s)-z(\omega,u',s))}{\eps + m_\sigma(\omega,u,s) } \mathrm dw(\omega,u',s),
\label{eq definition psi}
\end{equation}
where $m_\sigma(\omega,u,s)=\int_0^1 \varphi_\sigma^2(z(\omega,u,s)-z(\omega,v,s)) \mathrm dv$. 
We start by making sure that $\psi$ is well-defined.}
\begin{prop}\label{prop belong to M}
For all $z\in \mathcal M$, $\psi(z)$ belongs to $\mathcal M$. Furthermore, $(\psi(z)(\cdot,t))_{t \in [0,T]}$ is an  $L_2(0,1)$-valued continuous $(\mathcal G_t)_{t \in [0,T]}$-martingale. 
\end{prop}

\begin{rem}
The definition of an $L^\uparrow_2[0,1]$-valued martingale was given in Definition~\ref{def martingale}. 
Up to replacing $L^\uparrow_2$ by $L_2$, the definition of an $L_2(0,1)$-valued martingale is exactly the same. 
\end{rem}

\begin{proof}
We want to prove that $(\psi(z)(\cdot,t))_{t \in [0,T]}$ is an $L_2(0,1)$-valued $(\mathcal G_t)_{t \in [0,T]}$-martingale.
Since $z$ belongs to $\mathcal M$, the process $(z(\cdot,t))_{t\in[0,T]}$ is $(\mathcal G_t)_{t \in [0,T]}$-progressively measurable. Therefore $(m_\sigma(\cdot,t))_{t\in[0,T]}$ is also $(\mathcal G_t)_{t \in [0,T]}$-progressively measurable and we deduce that $(\psi(z)(\cdot,t))_{t\in[0,T]}$ is $(\mathcal G_t)_{t \in [0,T]}$-progressively measurable.

Then, we check that for each $t \in [0,T]$, $\psi(z)(\cdot,t) \in L_2(0,1)$ and $\E{\|\psi(z)(\cdot,t)\|_{L_2}}<\infty$. We deduce this statement by recalling that $\|g\|_{L_2}<+\infty$, because $g \in \mathcal L^\uparrow_{2+}[0,1]$, and by computing:
\begin{multline}
\label{ineg L2 prop 3.3}
\E{\left\| \int_0^t\!\!\int_0^1 \frac{\varphi_\sigma(z(\cdot,s)-z(u',s))}{\eps + m_\sigma(\cdot,s) } \mathrm dw(u',s) \right\|_{L_2}}^2 \!\!
\leq \E{\left\| \int_0^t\!\!\int_0^1 \frac{\varphi_\sigma(z(\cdot,s)-z(u',s))}{\eps + m_\sigma(\cdot,s) } \mathrm dw(u',s) \right\|_{L_2}^2} \\
\begin{aligned}
&= \E{\int_0^1 \left| \int_0^t\!\!\int_0^1 \frac{\varphi_\sigma(z(u,s)-z(u',s))}{\eps + m_\sigma(u,s) } \mathrm dw(u',s) \right|^2 \mathrm du}\\
&= \int_0^1\E{ \left| \int_0^t\!\!\int_0^1 \frac{\varphi_\sigma(z(u,s)-z(u',s))}{\eps + m_\sigma(u,s) } \mathrm dw(u',s) \right|^2 }\mathrm du \\
&= \int_0^1\E{ \ \int_0^t\!\!\int_0^1  \left(\frac{\varphi_\sigma(z(u,s)-z(u',s))}{\eps + m_\sigma(u,s) } \right)^2 \mathrm du' \mathrm ds }\mathrm du\\
&\leq \frac{\| \varphi_\sigma\|_\infty^2 t}{\eps^2}= \frac{ t}{\eps^2}<+\infty.
\end{aligned}
\end{multline}
Furthermore, for each $h \in L_2[0,1]$,
\begin{align*}
(\psi(z)(\cdot,t),h)_{L_2}=(g,h)_{L_2} + \int_0^t \!\! \int_0^1 \!\!\int_0^1 h(u) \frac{\varphi_\sigma(z(u,s)-z(u',s))}{\eps + m_\sigma(u,s) }\mathrm du\mathrm dw(u',s)  
\end{align*}
is a $(\mathcal G_t)_{t \in [0,T]}$-local martingale. Then, we compute the quadratic variation:
\begin{multline*}
\E{\langle (\psi(z),h)_{L_2},  (\psi(z),h)_{L_2}\rangle_t}
\\=\!\!\int_0^t \!\! \int_0^1 \!\!\int_0^1 \!\!\int_0^1\!\! h(u_1) h(u_2)  \frac{\varphi_\sigma(z(u_1,s)-z(u',s))\varphi_\sigma(z(u_2,s)-z(u',s))}{(\eps + m_\sigma(u_1,s) )(\eps + m_\sigma(u_2,s))}      \mathrm du_1 \mathrm du_2 \mathrm du' \mathrm ds 
\leq \frac{t}{\eps^2} \|h\|_{L_2}^2. 
\end{multline*}
Since it is finite, the local martingale is actually a martingale.

Moreover, by Doob's inequality
(see Theorem 2.2 in~\cite[p.22]{gawareckimandrekar11})
\begin{align*}
\|\psi(z)\|_{\mathcal M}
&=\E{ \sup_{t \leq T}\int_0^1  |\psi(z)(u,t)|^2\mathrm du}^{1/2}\\
&\leq \|g\|_{L_2}+ \E{ \sup_{t \leq T}\int_0^1  \left| \int_0^t\!\!\int_0^1 \frac{\varphi_\sigma(z(u,s)-z(u',s))}{\eps + m_\sigma(u,s) } \mathrm dw(u',s)   \right|^2\mathrm du}^{1/2}\\
&\leq \|g\|_{L_2}+ 2 \E{ \int_0^1  \left| \int_0^T\!\!\int_0^1 \frac{\varphi_\sigma(z(u,s)-z(u',s))}{\eps + m_\sigma(u,s) } \mathrm dw(u',s)   \right|^2\mathrm du}^{1/2}.
\end{align*}
The last term is finite by~\eqref{ineg L2 prop 3.3}.
Thus $\|\psi(z)\|_{\mathcal M}$ is finite and $\psi(z)$ belongs to $\mathcal M$, which concludes the proof. 
\end{proof}

Let us  now prove that $\psi$ has a unique fixed point: 
\begin{prop}
Let $\sigma>0$ and $\eps >0$. Then the map $\psi: \mathcal M \rightarrow \mathcal M$ defined by~(\ref{eq definition psi}) has a unique fixed point in $\mathcal M$,  denoted by $y_{\sigma,\eps}^g$. 
\label{prop point fixe}
\end{prop}
\begin{proof}
For all $n \in \mathds N$, denote by $\psi^n$ the $n$-fold composition of $\psi$, where $\psi^0$ denotes the identity function of $\mathcal M$. 
We want to prove that $\psi^n$ is a contraction for $n$ large enough. 

Let $z_1$ and $z_2$ be two elements of $\mathcal M$. We define
\[h_n(t):=\E{\sup_{s \leq t} \int_0^1 | \psi^n(z_1)(u,s)-\psi^n(z_2)(u,s)  |^2  \mathrm du}.
\]
Let us remark that $h_n(T)=\|\psi^n(z_1)-\psi^n(z_2)  \|^2_{\mathcal M}$ and recall that, by Proposition~\ref{prop belong to M}, $(\psi(z_1)(\cdot,t)-\psi(z_2)(\cdot,t))_{t \in [0,T]}$ is a $(\mathcal G_t)_{t \in[0,T]}$-martingale.
We denote by $m_{\sigma,1}$ and $m_{\sigma,2}$ the masses associated respectively to $z_1$ and $z_2$. 
By Doob's inequality, we have:
\begin{align*}
h_1(t) &=   \E{\sup_{s \leq t} \int_0^1| \psi(z_1)(u,s)-\psi(z_2)(u,s)  |^2\mathrm du }  \\
&=   \E{\sup_{s \leq t} \int_0^1\left| \int_0^s\!\! \int_0^1 \left( \frac{\varphi_\sigma(z_1(u,r)-z_1(u',r))}{\eps + m_{\sigma,1}(u,r)} - \frac{\varphi_\sigma(z_2(u,r)-z_2(u',r))}{\eps + m_{\sigma,2}(u,r)}  \right) \mathrm dw(u',r)  \right|^2\mathrm du }  \\
&\leq  4  \E{ \int_0^1 \!\! \int_0^t \!\!\int_0^1 \left|\frac{\varphi_\sigma(z_1(u,s)-z_1(u',s))}{\eps + m_{\sigma,1}(u,s)} - \frac{\varphi_\sigma(z_2(u,s)-z_2(u',s))}{\eps + m_{\sigma,2}(u,s)}  \right|^2 \mathrm du' \mathrm ds  \mathrm du } .
\end{align*}
Furthermore, we compute: 
\begin{multline*}
\left|\frac{\varphi_\sigma(z_1(u,s)-z_1(u',s))}{\eps + m_{\sigma,1}(u,s)} - \frac{\varphi_\sigma(z_2(u,s)-z_2(u',s))}{\eps + m_{\sigma,2}(u,s)}  \right|^2 \\
 \leq 2 \bigg( \left | \frac{\varphi_\sigma(z_1(u,s)-z_1(u',s))-\varphi_\sigma(z_2(u,s)-z_2(u',s))}{\eps + m_{\sigma,1}(u,s)}  \right|^2  \\
 +\left | \frac{\varphi_\sigma(z_2(u,s)-z_2(u',s))}{(\eps + m_{\sigma,1}(u,s))(\eps + m_{\sigma,2}(u,s))}  \left(  m_{\sigma,1}(u,s)- m_{\sigma,2}(u,s) \right) \right|^2 \bigg).
 \end{multline*}
Moreover, we have: 
\begin{align*}
| m_{\sigma,1}(u,s)- m_{\sigma,2}(u,s)|
&\leq \int_0^1 |\varphi_\sigma^2(z_1(u,s)-z_1(v,s))-\varphi_\sigma^2(z_2(u,s)-z_2(v,s))  |\mathrm dv \\
&\leq \operatorname{Lip}(\varphi_\sigma^2) \int_0^1 |(z_1(u,s)-z_1(v,s))-(z_2(u,s)-z_2(v,s))   |\mathrm dv\\
&\leq \operatorname{Lip}(\varphi_\sigma^2)\left( |z_1(u,s)-z_2(u,s)| +\int_0^1 |z_1(v,s)-z_2(v,s)| \mathrm dv \right).
\end{align*}
We obtain the following upper bound: 
\begin{multline*}
\left|\frac{\varphi_\sigma(z_1(u,s)-z_1(u',s))}{\eps + m_{\sigma,1}(u,s)} - \frac{\varphi_\sigma(z_2(u,s)-z_2(u',s))}{\eps + m_{\sigma,2}(u,s)}  \right|^2 \\
\leq  \left( 4 \left( \frac{\operatorname{Lip} \varphi_\sigma}{\eps} \right)^2+4\left(\frac{ \operatorname{Lip} (\varphi_\sigma^2)}{\eps^2}\right)^2\right)
\bigg(|z_1(u,s)-z_2(u,s)|^2 \\
+|z_1(u',s)-z_2(u',s)|^2+\int_0^1|z_1(v,s)-z_2(v,s)|^2 \mathrm dv \bigg).
\end{multline*}
Finally, we deduce that there is a constant $C_{\sigma,\eps}$ depending only on $\sigma$ and $\eps$ such that
\begin{align*}
h_1(t)&\leq C_{\sigma,\eps} \E{\int_0^t \!\! \int_0^1 |z_1(u,s)-z_2(u,s)|^2 \mathrm du \mathrm ds } \\
&\leq C_{\sigma,\eps} \int_0^t  \E{ \sup_{r\leq s}\int_0^1 |z_1(u,r)-z_2(u,r)|^2 \mathrm du }\mathrm ds = C_{\sigma,\eps} \int_0^t h_0(s) \mathrm ds.
\end{align*}
Applied to $\psi^n(z_1)$ and $\psi^n(z_2)$ instead of $z_1$ and $z_2$, those computations show that for every $t\in [0,T]$, 
$h_{n+1}(t) \leq C_{\sigma,\eps} \int_0^t h_n(s)\mathrm ds$. 
Using the fact that $h_0$ is non-decreasing with respect to $t$, it follows that
$h_n(T) \leq \frac{(C_{\sigma,\eps}T)^n}{n!}h_0(T)$,
whence we have: 
\[
\|\psi^n(z_1)-\psi^n(z_2)  \|^2_{\mathcal M} \leq \frac{(C_{\sigma,\eps}T)^n}{n!}\|z_1-z_2  \|^2_{\mathcal M}.
\]
Thus, for $n$ large enough, the map $\psi^n$ is a contraction. By completeness of $\mathcal M$ under the norm~$\|\cdot\|_{\mathcal M}$ (remark that $\mathcal M$ is a closed subset of $L_2(\Omega,\mathcal C([0,T], L_2(0,1)))$, 
it follows that $\psi$ has a unique fixed point in $\mathcal M$. 
\end{proof}

We denote by $y^g_{\sigma,\eps}$ the unique fixed point of $\psi$. Remark that by construction it satisfies equation~\eqref{equation avec phi} almost surely and for every $t\in [0,T]$.

\subsection{Non-decreasing property}

Define, for each $t \in [0,T]$,  $Y^g_{\sigma,\eps}(t):=y^g_{\sigma,\eps}(\cdot,t)$.
So far, by Proposition~\ref{prop point fixe}, 
we have established that $(Y^g_{\sigma,\eps}(t))_{t \in [0,T]}$ is an $L_2[0,1]$-valued process, satisfying property $(A1)$ of Proposition~\ref{proprietes A1-A3}. Since $Y^g_{\sigma,\eps}$ belongs to $\mathcal M$, and by Proposition~\ref{prop belong to M}, $(Y^g_{\sigma,\eps}(t))_{t \in [0,T]}$ is a square integrable continuous $L_2[0,1]$-valued martingale, with respect to the filtration $(\mathcal G_t)_{t \in [0,T]}$. Therefore, it is also an $(\mathcal F^{\sigma,\eps}_t)_{t \in [0,T]}$-martingale, where $\mathcal F^{\sigma,\eps}_t:=\sigma(Y^g_{\sigma,\eps}(s),s\leq t)$. 

In order to obtain property $(A2)$, it remains to prove the following statement:
\begin{prop}
\label{prop non-decreasing property}
$(Y^g_{\sigma,\eps}(t))_{t \in [0,T]}$ is an $L_2^\uparrow[0,1]$-valued process. 
\end{prop}

We will start by proving three Lemmas and then we will conclude the proof of Proposition~\ref{prop non-decreasing property}. 
For every $x \in \R$, we consider the following stochastic differential equation:
\begin{align}
\label{equation en x}
z(x,t)
= x+ \int_0^t\!\!\int_0^1 \frac{\mathds \varphi_\sigma(z(x,s)-y^g_{\sigma,\eps}(u',s))}
{\eps + \int_0^1\varphi_\sigma^2(z(x,s)-y^g_{\sigma,\eps}(v,s)) \mathrm dv}    \mathrm dw(u',s),
\end{align}
where $y^g_{\sigma,\eps}$ is the unique solution of equation~\eqref{equation avec phi}. 
\begin{lemme}
\label{lemme 3.9 bis}
Let $x \in \R$. For almost every $\omega \in \Omega$, equation~\eqref{equation en x} has a unique solution in $\mathcal C[0,T]$, denoted by $(z(\omega,x,t))_{t\in [0,T]}$.
Moreover, $(z(x,t))_{t\in[0,T]}$ is a real-valued $(\mathcal G_t)_{t\in[0,T]}$-martingale. 
\end{lemme}
\begin{proof}
We get existence and uniqueness of the solution by applying a fixed-point argument. The proof is the same as the proof of Proposition~\ref{prop point fixe}. 
We obtain the martingale property by the same argument as in Proposition~\ref{prop belong to M}. 
\end{proof}

Then, take $x_1$, $x_2 \in \R$. After some computations similar to those of the proof of Proposition~\ref{prop point fixe}, we have for every $t\in[0,T]$: 
\begin{align*}
\E{\sup_{ s \leq t} | z(x_1,s)-z(x_2,s) |^2} 
\leq 2|x_1-x_2|^2 +C_{\sigma,\eps}\int_0^t \E{\sup_{r \leq s} \left|z(x_1,r)-z(x_2,r)\right|^2} \mathrm ds.
\end{align*}
By Gronwall's Lemma, we deduce that:
\begin{align*}
\E{\sup_{t \leq T} |z(x_1,t)-z(x_2,t)|^2}\leq C_{\sigma,\eps} |x_1-x_2|^2. 
\end{align*}
By Kolmogorov's Lemma, there is a modification $\widetilde{z}$ of $z$ in $\mathcal C(\R\times[0,T])$. 
We define $\widetilde{y}^g_{\sigma,\eps}(u,t):= \widetilde{z}(g(u),t)$. In particular, $u \mapsto \widetilde{y}^g_{\sigma,\eps}(u,\cdot)$ is measurable and, since $g$ is a càdlàg function, $\widetilde{y}^g_{\sigma,\eps}$ belongs to $\mathcal D((0,1), \mathcal C[0,T])$. 

\begin{rem}
In the case where $g$ is continuous, it is straightforward to see that $\widetilde{y}^g_{\sigma,\eps}$ belongs to $\mathcal C([0,1] \times [0,T])$. 
\end{rem}

 Furthermore,  
$\widetilde{y}^g_{\sigma,\eps}$ belongs to $\mathcal M$. Indeed, 
\begin{align*}
\E{\sup_{t\leq T} \int_0^1\left|\widetilde{y}^g_{\sigma,\eps}(u,t)\right|^2 \mathrm du}\leq \E{ \int_0^1\sup_{t\leq T}\left|\widetilde{y}^g_{\sigma,\eps}(u,t)\right|^2 \mathrm du}
=\int_0^1 \E{ \sup_{t\leq T}\left|\widetilde{y}^g_{\sigma,\eps}(u,t)\right|^2} \mathrm du.
\end{align*}
By Lemma~\ref{lemme 3.9 bis}, for every $u\in [0,1]$,  $(\widetilde{y}^g_{\sigma,\eps}(u,t))_{t\in[0,T]}$ is a martingale, we have by Doob's inequality:
\begin{align*}
\E{ \sup_{t\leq T}\left|\widetilde{y}^g_{\sigma,\eps}(u,t)\right|^2} 
&\leq C \E{\left|\widetilde{y}^g_{\sigma,\eps}(u,T)\right|^2} \\
&\leq 2C g(u)^2 +2C \E{\int_0^T\!\!\int_0^1 \left|\frac{\mathds \varphi_\sigma(\widetilde{y}^g_{\sigma,\eps}(u,s)-y^g_{\sigma,\eps}(u',s))}
{\eps + \int_0^1\varphi_\sigma^2(\widetilde{y}^g_{\sigma,\eps}(u,s)-y^g_{\sigma,\eps}(v,s)) \mathrm dv}\right|^2    \mathrm du' \mathrm ds}\\
&\leq 2C g(u)^2 +2C\frac{T}{\eps^2}.
\end{align*}
Therefore, 
$\|\widetilde{y}^g_{\sigma,\eps}\|_{\mathcal M} \leq 2C \|g\|_{L_2}^2 +2C\frac{T}{\eps^2}<+\infty$.
Moreover, $(\widetilde{y}^g_{\sigma,\eps}(\cdot,t))_{t\in[0,T]}$ is an $L_2[0,1]$-valued $(\mathcal G_t)_{t\in[0,T]}$-martingale. Indeed, for every $h \in L_2[0,1]$, for every $t \in [0,T]$,  $\E{(\widetilde{y}^g_{\sigma,\eps}(\cdot,t), h)_{L_2}}$ is finite.  Fix $0\leq s \leq t \leq T$, and $A_s \in \mathcal G_s$. We have:
\begin{align*}
\E{\left( \int_0^1 \widetilde{y}^g_{\sigma,\eps}(u,t) h(u) \mathrm du - \int_0^1 \widetilde{y}^g_{\sigma,\eps}(u,s) h(u) \mathrm du \right)\! \mathds 1_{A_s}}
&= \int_0^1\!\!\E{(\widetilde{y}^g_{\sigma,\eps}(u,t) - \widetilde{y}^g_{\sigma,\eps}(u,s) ) \mathds 1_{A_s}} h(u) \mathrm du\\
&=0. 
\end{align*}

\begin{lemme}
\label{lemme egalite dans M}
We have $\E{\sup_{t \leq T}\int_0^1 \left| \widetilde{y}^g_{\sigma,\eps}(u,t)-y^g_{\sigma,\eps}(u,t) \right|^2 \mathrm du}=0$. Therefore, $\widetilde{y}^g_{\sigma,\eps}=y^g_{\sigma,\eps}$ in~$\mathcal M$. 
\end{lemme}

\begin{proof}
Since $(\widetilde{y}^g_{\sigma,\eps}(\cdot,t)-y^g_{\sigma,\eps}(\cdot,t))_{t\in [0,T]}$ is an $L_2[0,1]$-valued martingale, then by~\cite[p.21-22]{gawareckimandrekar11}  $\int_0^1 \left| \widetilde{y}^g_{\sigma,\eps}(u,t)-y^g_{\sigma,\eps}(u,t) \right|^2 \mathrm du$ is a real-valued submartingale. By Doob's inequality, 
\begin{multline*}
\E{\sup_{s \leq t}\int_0^1 \left| \widetilde{y}^g_{\sigma,\eps}(u,s)-y^g_{\sigma,\eps}(u,s) \right|^2 \mathrm du}
\leq C \E{\int_0^1 \left| \widetilde{y}^g_{\sigma,\eps}(u,t)-y^g_{\sigma,\eps}(u,t) \right|^2 \mathrm du}\\
\begin{aligned}
&\leq C \E{\int_0^1\left| \int_0^t\!\!\int_0^1 (\theta_{\sigma,\eps} (\widetilde{y}^g_{\sigma,\eps}(u,s),u',s) - \theta_{\sigma,\eps} (y^g_{\sigma,\eps}(u,s),u',s))   \mathrm dw(u',s) \right|^2\mathrm du} \\
&\leq C \E{\int_0^1\!\! \int_0^t\!\!\int_0^1 \left|\theta_{\sigma,\eps} (\widetilde{y}^g_{\sigma,\eps}(u,s),u',s) - \theta_{\sigma,\eps} (y^g_{\sigma,\eps}(u,s),u',s)\right|^2 \mathrm du' \mathrm ds\mathrm du},
\end{aligned}
\end{multline*} 
where $\theta_{\sigma,\eps}(x,u',s)= \frac{\varphi_\sigma(x-y^g_{\sigma,\eps}(u',s))}{\eps + \int_0^1 \varphi_\sigma^2(x-y^g_{\sigma,\eps}(v,s))\mathrm dv}$. Using the same constant $C_{\sigma,\eps}$ as in the proof of Proposition~\ref{prop point fixe}, we have:
\begin{align*}
\E{\sup_{s \leq t}\int_0^1 \left| \widetilde{y}^g_{\sigma,\eps}(u,s)-y^g_{\sigma,\eps}(u,s) \right|^2 \mathrm du}
&\leq C_{\sigma,\eps} \E{\int_0^1\!\! \int_0^t \left|\widetilde{y}^g_{\sigma,\eps}(u,s)-y^g_{\sigma,\eps}(u,s)\right|^2 \mathrm ds\mathrm du}\\
&\leq C_{\sigma,\eps}\int_0^t\E{\sup_{r\leq s}\int_0^1 \left|\widetilde{y}^g_{\sigma,\eps}(u,r)-y^g_{\sigma,\eps}(u,r)\right|^2\mathrm du} \mathrm ds.
\end{align*}
By Gronwall's Lemma, we deduce that $\E{\sup_{s \leq t}\int_0^1 \left| \widetilde{y}^g_{\sigma,\eps}(u,s)-y^g_{\sigma,\eps}(u,s) \right|^2 \mathrm du}=0$ for every $t\in [0,T]$.
This implies the statement of the Lemma. 
\end{proof}

\begin{lemme}
\label{lemme croissance}
Almost surely, for every $u_1,u_2 \in \mathds Q$ such that $u_1<u_2$, we have for every $t \geq 0$, $\widetilde{y}^g_{\sigma,\eps}(u_1,t) \leq \widetilde{y}^g_{\sigma,\eps}(u_2,t)$. 
Furthermore, if $g(u_1) < g(u_2)$ (resp. $g(u_1) = g(u_2)$), then for every $t \geq 0$, $\widetilde{y}^g_{\sigma,\eps}(u_1,t) < \widetilde{y}^g_{\sigma,\eps}(u_2,t)$ (resp. $\widetilde{y}^g_{\sigma,\eps}(u_1,t) = \widetilde{y}^g_{\sigma,\eps}(u_2,t)$).
\end{lemme}
\begin{proof}
Let $(u_1,u_2) \in \Q^2$ such that $0\leq u_1<u_2 \leq 1$. For $u=u_1,u_2$, we have:
\begin{equation*}
\widetilde{y}^g_{\sigma,\eps} (u,t)=g(u)+\int_0^t \!\!\int_0^1 \theta_{\sigma,\eps} (\widetilde{y}^g_{\sigma,\eps}(u,s),u',s) \mathrm dw(u',s),
\end{equation*}
where $\theta_{\sigma,\eps}(x,u',s)= \frac{\varphi_\sigma(x-y^g_{\sigma,\eps}(u',s))}{\eps + \int_0^1 \varphi_\sigma^2(x-y^g_{\sigma,\eps}(v,s))\mathrm dv}$. 
Therefore, we have (writing $\widetilde{y}$ instead of $\widetilde{y}^g_{\sigma,\eps}$ and $\theta$ instead of $\theta_{\sigma,\eps}$):
\begin{align}\label{eds 1}
\widetilde{y} (u_2,t)-\widetilde{y} (u_1,t)
&=g(u_2)-g(u_1)+\int_0^t \!\!\int_0^1 (\theta (\widetilde{y}(u_2,s),u',s)-\theta (\widetilde{y}(u_1,s),u',s)) \mathrm dw(u',s)\notag
\\
&=g(u_2)-g(u_1) +\int_0^t (\widetilde{y} (u_2,s)-\widetilde{y} (u_1,s))\mathrm dM_s
\end{align}
where
$M_t =\int_0^t \int_0^1 \mathds 1_{\{\widetilde{y}(u_2,s)\neq \widetilde{y} (u_1,s)\} } \frac{\theta (\widetilde{y}(u_2,s),u',s)-\theta (\widetilde{y}(u_1,s),u',s)}{\widetilde{y}(u_2,s)-\widetilde{y}(u_1,s)} \mathrm dw(u',s)$. 
Observe that:
\begin{align*}
\theta (\widetilde{y}(u_2,s),u',s)-\theta (\widetilde{y}(u_1,s),u',s)
=\int_{\widetilde{y}(u_1,s)}^{\widetilde{y}(u_2,s)} \partial_x \theta (x,u',s)\mathrm dx,
\end{align*}
and that $\partial_x \theta (x,u',s)=\frac{\varphi_\sigma'(x-y(u',s))}{\eps + \int_0^1 \varphi_\sigma^2(x-y(v,s))\mathrm dv}-\frac{\varphi_\sigma(x-y(u',s))\int_0^1 (\varphi_\sigma^2)'(x-y(v,s))\mathrm dv}{(\eps + \int_0^1 \varphi_\sigma^2(x-y(v,s))\mathrm dv)^2}$. Therefore, $\partial_x \theta$ is bounded uniformly in $(x,u',s)\in \R \times[0,1]\times [0,T]$ by $ C_{\sigma,\eps}:=\frac{\| \varphi_\sigma '\|_{L_\infty}}{\eps}+\frac{\| \varphi_\sigma \|_{L_\infty}\| (\varphi_\sigma^2)' \|_{L_\infty}}{\eps^2}$. We deduce that
\begin{align*}
\E{\langle M,M \rangle_T}
&=\E{\int_0^T \!\!\int_0^1 \mathds 1_{\{\widetilde{y}(u_2,s)\neq \widetilde{y} (u_1,s) \}} \left(\frac{\theta (\widetilde{y}(u_2,s),u',s)-\theta (\widetilde{y}(u_1,s),u',s)}{\widetilde{y}(u_2,s)-\widetilde{y}(u_1,s)} \right)^2\mathrm du' \mathrm ds}\\
&\leq \E{\int_0^T\!\! \int_0^1  \left(C_{\sigma,\eps}\right)^2\mathrm du' \mathrm ds}
\leq T \left(C_{\sigma,\eps}\right)^2,
\end{align*}
and thus $M$ is a $(\mathcal G_t)_{t \in [0,T]}$-martingale on $[0,T]$. 
We resolve the stochastic differential equation~\eqref{eds 1}:
$\widetilde{y}^g_{\sigma,\eps} (u_2,t)-\widetilde{y}^g_{\sigma,\eps} (u_1,t)
=(g(u_2)-g(u_1))\exp\left(M_t - \frac{1}{2}\langle M,M\rangle_t  \right)$.
If $g(u_1)< g(u_2)$ (resp. $g(u_1)=g(u_2)$), then almost surely for every $t\in [0,T]$,  $\widetilde{y}^g_{\sigma,\eps}(u_1,t) < \widetilde{y}^g_{\sigma,\eps}(u_2,t)$ (resp. $=$). Thus it is true almost surely for every $(u_1,u_2) \in \mathds Q^2$ such that $u_1<u_2$. 
\end{proof}

Therefore the proof of Proposition~\ref{prop non-decreasing property} is complete:
\begin{proof}[Proof (Proposition~\ref{prop non-decreasing property})]
For each $t \in [0,T]$, $Y^g_{\sigma,\eps}(t)=y^g_{\sigma,\eps}(\cdot,t)$ has a modification $\widetilde{y}^g_{\sigma,\eps}(\cdot,t)$ belonging to $L_2^\uparrow[0,1]$. 
\end{proof}

We precise the properties of $\widetilde{y}^g_{\sigma,\eps}$ in the following Corollary, which derives directly from Proposition~\ref{prop non-decreasing property}. From now on, we will always use this version of the process. 
\begin{coro}
\label{coro version continue ps}
The following two statements hold:
\begin{itemize}
\item for almost every $u \in (0,1)$,  $(\widetilde{y}^g_{\sigma,\eps}(\omega,u,t))_{t \in [0,T]}$ is a $(\mathcal F^{\sigma,\eps}_t)_{t \in [0,T]}$-martingale, and it is continuous for almost every $(u,\omega) \in (0,1)\times \Omega$. 
\item almost surely, for every $t\in [0,T]$, $u \mapsto \widetilde{y}^g_{\sigma,\eps}(u,t)$ is càdlàg and non-decreasing. 
\end{itemize} 
\end{coro}

We complete the proof of Proposition~\ref{proprietes A1-A3}.
\begin{proof}[Proof (Proposition~\ref{proprietes A1-A3}).]
Thanks to Proposition~\ref{prop non-decreasing property}, the proof of properties $(A1)$ and $(A2)$ has been completed. 
It remains to compute the quadratic variation. 
Recall that for every $u\in [0,1]$, $(\widetilde{y}^g_{\sigma,\eps}(u,t))_{t\in[0,T]}$ is a $(\mathcal G_t)_{t\in[0,T]}$-martingale and that 
\begin{equation*}
\widetilde{y}_{\sigma,\eps}^g (u,t)=g(u)+\int_0^t \!\! \int_0^1 \theta_{\sigma,\eps} (\widetilde{y}_{\sigma,\eps}^g(u,s),u',s) \mathrm dw(u',s).
\end{equation*}
Therefore, for every $u,u' \in [0,1]$, 
\begin{align*}
\langle \widetilde{y}_{\sigma,\eps}^g (u,\cdot), \widetilde{y}_{\sigma,\eps}^g (u',\cdot) \rangle_t
&= \langle \int_0^{\cdot}  \!\! \int_0^1 \theta_{\sigma,\eps} (\widetilde{y}_{\sigma,\eps}^g(u,s),v,s) \mathrm dw(v,s), \int_0^{\cdot} \!\! \int_0^1 \theta_{\sigma,\eps} (\widetilde{y}_{\sigma,\eps}^g(u',s),v,s) \mathrm dw(v,s)
\rangle_t \\
&= \int_0^t  \!\! \int_0^1 \theta_{\sigma,\eps} (\widetilde{y}_{\sigma,\eps}^g(u,s),v,s)  \theta_{\sigma,\eps} (\widetilde{y}_{\sigma,\eps}^g(u',s),v,s) \mathrm dv \mathrm ds.
\end{align*}
Therefore, for every $h,k \in L_2[0,1]$, 
\begin{align*}
\langle (Y^g_{\sigma,\eps},h)_{L_2},(Y^g_{\sigma,\eps},k)_{L_2}\rangle_t 
&\!\!= \int_0^t \!\! \int_0^1 \!\! \int_0^1\!\! h(u)k(u')\!\!\int_0^1\!\!\theta_{\sigma,\eps} (\widetilde{y}_{\sigma,\eps}^g(u,s),v,s)  \theta_{\sigma,\eps} (\widetilde{y}_{\sigma,\eps}^g(u',s),v,s)\mathrm dv\mathrm du\mathrm du' \mathrm ds \\
&\!\!= \int_0^t \!\! \int_0^1 \!\! \int_0^1 \!\!h(u)k(u')\!\!\int_0^1\!\!\theta_{\sigma,\eps} (y_{\sigma,\eps}^g(u,s),v,s)  \theta_{\sigma,\eps} (y_{\sigma,\eps}^g(u',s),v,s)\mathrm dv\mathrm du\mathrm du' \mathrm ds \\
&\!\!= \int_0^t \!\! \int_0^1 \!\! \int_0^1 h(u)k(u')\frac{m^g_{\sigma,\eps}(u,u',s)}{(\eps+m^g_{\sigma,\eps}(u,s))(\eps+m^g_{\sigma,\eps}(u',s))}\mathrm du\mathrm du' \mathrm ds, 
\end{align*}
which completes the proof. 
\end{proof} 

We conclude this Section with a property on the quadratic variation of two fixed particles, which will be useful to obtain lower bounds on the mass in the next Section. 
\begin{coro}
\label{coro crochet}
For almost every $u,u' \in [0,1]$, 
 \begin{align}
 \label{coro 3.14}
\langle \widetilde{y}_{\sigma,\eps}^g (u,\cdot), \widetilde{y}_{\sigma,\eps}^g (u',\cdot) \rangle_t
= \int_0^t  \!\! \int_0^1 \frac{m^g_{\sigma,\eps}(u,u',s)}{(\eps+m^g_{\sigma,\eps}(u,s))(\eps+m^g_{\sigma,\eps}(u',s))} \mathrm dv \mathrm ds.
\end{align}
\end{coro}
\begin{proof}
This statement follows clearly from the proof of Proposition~\ref{proprietes A1-A3}, from the fact that for almost every $u \in (0,1)$, $(\widetilde{y}^g_{\sigma,\eps}(u,t))_{t \in [0,T]}$ is a continuous martingale.
\end{proof}

\section{Convergence of the process $(y^g_{\sigma,\eps})_{\sigma, \eps \in \Q_+}$}
\label{sec convergence}

From now on, for the sake of simplicity, we fix a function $g$ in $\mathcal  L^\uparrow_{2+}[0,1]$ and  $y_{\sigma,\eps}$ will denote the version $\widetilde{y}^g_{\sigma,\eps}$ starting from $g$.
We denote by $p$ a number such that $p>2$ and $g \in L_p(0,1)$. 

We begin by proving the tightness of the sequence  $(y_{\sigma,\eps})_{\sigma, \eps \in \Q_+}$ in the space $ L_2([0,1],\mathcal C[0,T])$ in Paragraph~\ref{subsec tightness}. We will then pass to the limit in distribution, first when $\eps \to 0$ and then when $\sigma \to 0$ and prove, in Paragraph~\ref{par conv sigma 0}, that the limit process is also a martingale.

\subsection{Tightness of the collection $(y_{\sigma,\eps})_{\sigma>0,\eps>0}$ in $ L_2([0,1],\mathcal C[0,T])$ }
\label{subsec tightness}

Recall that for all $\sigma>0$, the map $\varphi_\sigma$ is smooth, even, bounded by 1, equal to 1 on $\left [0,\frac{\sigma-\eta}{2}\right]$ and equal to 0 on $\left [ \frac{\sigma}{2},+\infty \right)$, where $\eta$ is chosen so that $\eta < \frac{\sigma}{3}$. Recall that $y_{\sigma,\eps}$ is solution of the following equation: 
\[
y_{\sigma,\eps} (u,t)=g(u)+\int_0^t \!\! \int_0^1 \frac{\varphi_\sigma(y_{\sigma,\eps}(u,s)-y_{\sigma,\eps}(u',s))}{\eps + \int_0^1 \varphi_\sigma^2(y_{\sigma,\eps}(u,s)-y_{\sigma,\eps}(v,s))\mathrm dv} \mathrm dw(u',s).
\]

We begin by proving that the collection $(y_{\sigma,\eps})_{\sigma>0,\eps>0}$ satisfies a compactness criterion in the space $L_2([0,1],\mathcal C[0,T])$. We recall the following criterion (see~\cite[Theorem 1, p.71]{simon87}):
\begin{prop} 
\label{critere simon}
Let $K$ be a subset of $L_2([0,1],\mathcal C[0,T])$. 

$K$ is relatively compact in $L_2([0,1],\mathcal C[0,T])$ if and only if:
\begin{itemize}
\item[$(H1)$] for every $0\leq u_1 < u_2 \leq 1$, $\left\{ \int_{u_1}^{u_2} f(u,\cdot)\mathrm du, f \in K \right\}$ is relatively compact in $C[0,T]$,

\item[$(H2)$] $\lim_{h\rightarrow 0^+} \sup_{f \in K} \int_0^{1-h} \|f(u+h,\cdot)-f(u,\cdot)\|^2_{\mathcal C[0,T]}\mathrm du =0$.
\end{itemize}
\end{prop}
By Ascoli's Theorem, $(H1)$ is satisfied if and only if for every $0\leq u_1 <u_2 \leq 1$, 
\begin{itemize}
\item[-] for every $t \in [0,T]$, $\int_{u_1}^{u_2} f(u,t) \mathrm du $ is uniformly bounded,
\item[-] $\lim_{\eta \to 0^+} \sup_{f \in K} \sup_{ |t_2-t_1| < \eta} \left| \int_{u_1}^{u_2} (f(u,t_2)-f(u,t_1)) \mathrm du \right|=0$. 
\end{itemize}

In order to prove tightness for the collection $(y_{\sigma,\eps})_{\sigma>0,\eps>0}$, we will prove the following Proposition: 
\begin{prop}
\label{prop crit tight}
Let $\delta>0$. The following statements hold:
\begin{itemize}
\item[$(K1)$] there exists $M>0$ such that for all $\sigma>0$, $\eps>0$,  
$\P{\int_0^1 \|y_{\sigma,\eps}(u,\cdot)\|_{\mathcal C[0,T]}^2\mathrm du \leq M}\geq 1-\delta$,
\item[$(K2)$] for all $k \geq 1$, there exists $\eta_k>0$ such that for all $\sigma>0$, $\eps>0$,  
\begin{align*}
\P{\int_0^1 \sup_{|t_2-t_1| < \eta_k} |y_{\sigma,\eps}(u,t_2)-y_{\sigma,\eps}(u,t_1)| \mathrm du \leq \frac{1}{k}} \geq 1-\frac{\delta}{2^k},
\end{align*}
\item[$(K3)$] for all $k \geq 1$, there exists $h_k>0$ such that for all $\sigma>0$, $\eps>0$,
\begin{align*}
\P{\forall h \in (0,h_k), \int_0^{1-h} \|y_{\sigma,\eps}(u+h,\cdot)-y_{\sigma,\eps}(u,\cdot)\|_{\mathcal C[0,T]}^2 \mathrm du \leq \frac{1}{k}} \geq 1-\frac{\delta}{2^k}.
\end{align*}
\end{itemize}
\end{prop}

Proposition~\ref{prop crit tight} will be proved in Paragraph~\ref{subsec proof prop 4.2}.  It implies tightness of $(y^g_{\sigma,\eps})_{\sigma>0,\eps>0}$ in $L_2([0,1],\mathcal C[0,T])$:
\begin{coro}
\label{coro tension y sig eps}
For all $g \in \mathcal L^\uparrow_{2+}[0,1]$, the collection $(y^g_{\sigma,\eps})_{\sigma>0,\eps>0}$ is tight in $L_2([0,1],\mathcal C[0,T])$. 
\end{coro}

\begin{proof}[Proof (Corollary~\ref{coro tension y sig eps})]
Let $\delta>0$. Let $M$, $(h_k)_{k \geq 1}$, $(\eta_k)_{k \geq 1}$ be such that the statements of Proposition~\ref{prop crit tight} hold for~$\delta$. 

Denote $K_\delta$ the closed set of all functions $f \in L_2([0,1],\mathcal C[0,T])$ satisfying:
\begin{itemize}
\item[$(L1)$] 
$\displaystyle \int_0^1 \|f(u,\cdot)\|_{\mathcal C[0,T]}^2\mathrm du \leq M$. 
\item[$(L2)$] for all $k \geq 1$, 
$\displaystyle \int_0^1 \sup_{|t_2-t_1| < \eta_k} |f(u,t_2)-f(u,t_1)| \mathrm du \leq \frac{1}{k}$. 
\item[$(L3)$] for all $k\geq 1$, 
$\displaystyle \forall h \in (0,h_k), \int_0^{1-h} \|f(u+h,\cdot)-f(u,\cdot)\|_{\mathcal C[0,T]}^2 \mathrm du \leq \frac{1}{k}$. 
\end{itemize}

Let $0\leq u_1<u_2\leq 1$. We deduce from~$(L1)$ that for every $t \in [0,T]$, and every $f \in K_\delta$, 
$\left|\int_{u_1}^{u_2} f(u,t) \mathrm du\right| \leq \left( \int_{u_1}^{u_2} f(u,t)^2 \mathrm du\right)^{1/2} \leq \left( \int_0^1 \|f(u,\cdot)\|_{\mathcal C[0,T]}^2 \mathrm du\right)^{1/2}\leq \sqrt{M}$. 
We deduce from~$(L2)$ that for every $k \geq 1$, 
\begin{align*}
\sup_{f \in K_\delta} \sup_{ |t_2-t_1| < \eta_k} \left| \int_{u_1}^{u_2} (f(u,t_2)-f(u,t_1)) \mathrm du \right|
\leq \sup_{f \in K_\delta} \int_0^1 \sup_{|t_2-t_1| < \eta_k} |f(u,t_2)-f(u,t_1)| \mathrm du
\leq \frac{1}{k}.
\end{align*}
Therefore, by Ascoli's Theorem, condition~$(H1)$ of Proposition~\ref{critere simon} is satisfied. 

Furthermore, by~$(L3)$, condition~$(H2)$ is also satisfied uniformly for $f\in K_\delta$. Therefore, $K_\delta$ is compact in $L_2([0,1],\mathcal C[0,T])$. 
By Proposition~\ref{prop crit tight}, for all $\sigma>0$, $\eps>0$, $\P{y_{\sigma,\eps} \in K_\delta} \geq 1-3\delta$. This concludes the proof. 
\end{proof}

To prove Proposition~\ref{prop crit tight}, we will first give in the next Paragraph an estimation of the inverse of the mass function (see Lemma~\ref{lemme inverse masse approchee sig eps}). This Lemma is an equivalent in our case of short-range interacting particles of Lemma~2.16 in~\cite{konarovskyi17system}, stated in the case of a system of coalescing particles. 

\subsubsection{Estimation of the inverse of mass}

Recall that $m_{\sigma,\eps}(u,t)=\int_0^1 \varphi_\sigma^2(y_{\sigma,\eps}(u,t)-y_{\sigma,\eps}(v,t))\mathrm dv $.
We  define a modified mass 
\begin{align*}
 M_{\sigma,\eps}(u,t):=
\left \{\begin{aligned}
&\frac{(\eps + m_{\sigma,\eps})^2 }{m_{\sigma,\eps}}(u,t) \text{ if }m_{\sigma,\eps}(u,t)>0,\\
&+\infty \text{ otherwise.}
\end{aligned}\right.
\end{align*}
Clearly, $M_{\sigma,\eps}(u,t)\geq m_{\sigma,\eps}(u,t)$ for every $u \in [0,1]$ and $t \in [0,T]$.

By Corollary~\ref{coro version continue ps}, there exists a (non-random) Borel set $\A$ in $[0,1]$, $\Leb(\A)=1$, such that for all $u \in \A$, $(y_{\sigma,\eps}(u,t))_{t \in [0,T]}$ is almost surely a continuous $(\mathcal F^{\sigma,\eps}_t)_{t \in [0,T]}$-martingale. 
Recall also that almost surely, for every $t\in [0,T]$, $u \mapsto y_{\sigma,\eps}(u,t)$ is càdlàg and non-decreasing. 
Moreover, we assume that for every $u,u' \in  \A$, equality~\eqref{coro 3.14} holds.

\begin{lemme}
There exist $C>0$ and $\gamma \in (0,1)$ such that for each $\sigma , \eps >0$, $t \in (0,T]$ and for every $u \in \A$ and every $h>0$ satisfying $u-h \in (0,1)$, 
\begin{align}
\label{proba M}
\P{\int_0^T \mathds 1_{\left\{M_{\sigma,\eps}(u,s)<\gamma h\right\}} \mathrm ds \geq t} \leq C \left[ g(u)-g(u-h) \right]\sqrt{\frac{h}{t}}.
\end{align}
\label{lemme controle masse}
\end{lemme}

\begin{proof}
Fix $\sigma>0$ and $\eps >0$. Let $h>0$ be such that $u-h$ belongs to $\A$. 
If $g(u-h)=g(u)$, then for every $t \in [0,T]$, $y_{\sigma,\eps}(u-h,t)=y_{\sigma,\eps}(u,t)$. By the non-decreasing and càdlàg property, for every $v \in (u-h,u)$, we have $y_{\sigma,\eps}(v,t)=y_{\sigma,\eps}(u,t)$. We deduce that $m_{\sigma,\eps} (u,t) \geq \int_{u-h}^u \varphi_\sigma^2(y_{\sigma,\eps}(u,t)-y_{\sigma,\eps}(v,t))\mathrm dv=\int_{u-h}^u\varphi_\sigma^2(0) \mathrm dv =h $. Therefore, $M_{\sigma,\eps} (u,t) \geq h \geq \gamma h$ for every $t \in [0,T]$, and~\eqref{proba M} is satisfied. 

Consider now the case where $g(u-h)<g(u)$.  Choose $k$ in $(\frac{h}{3},\frac{2h}{3})$ such that $u-k \in \A$.  Denote by $N$ and $\widetilde{N}$ the two following $(\mathcal F^{\sigma,\eps}_t)_{t \in [0,T]}$-martingales: 
\begin{align*}
\textstyle N_t &= \textstyle y_{\sigma,\eps}(u,t)-y_{\sigma,\eps}(u-h,t), \\
\widetilde{N}_t &= \textstyle y_{\sigma,\eps}(u,t)-y_{\sigma,\eps}(u-k,t).
\end{align*}
Denote by $G_s$ and $H_s$  respectively the events $\left\{M_{\sigma,\eps}(u,s)<\frac{h}{2^6}\right\}$ and $\{\widetilde{N}_s > \frac{\sigma+\eta}{2}\}$. We want to prove the existence of a constant $C_1$ independent of $h$ and $u$ such that for all $\sigma>0$, $\eps >0$ and $t>0$, 
\begin{align}
\P{\int_0^T \mathds 1_{\{G_s\}} \mathrm ds \geq t} \leq C_1\left[ g(u)-g(u-h) \right]\sqrt{\frac{h}{t}}.
\label{inégalité C_1}
\end{align}

Decompose this probability in two terms: 
\begin{equation}
\P{\int_0^T \mathds 1_{\{G_s\}} \mathrm ds \geq t} \leq \P{\int_0^T \mathds 1_{\{G_s \cap H_s\}} \mathrm ds \geq \frac{t}{2}}+\P{\int_0^T \mathds 1_{\{G_s \cap H_s^\complement\}} \mathrm ds \geq \frac{t}{2}},
\label{inég deux parties}
\end{equation}
where $H_s^\complement$ denotes the complement of the event $H_s$.

\begin{description}
\item[\textbullet \, First step: ]Study of $G_s\cap H_s$.

Fix $s \in [0,T]$. 
Under $G_s\cap H_s$, we have $M_{\sigma,\eps}(u,s)<\frac{h}{2^6}$ and $\widetilde{N}_s >\frac{\sigma+\eta}{2}$.
We want to show that it implies the following inequality: 
\begin{equation}
\frac{ 2m_{\sigma,\eps}(u,u-h,s)}{(\eps + m_{\sigma,\eps}(u,s))(\eps + m_{\sigma,\eps}(u-h,s))} \leq \frac{1}{M_{\sigma,\eps}(u,s)^{3/4} M_{\sigma,\eps}(u-h,s)^{1/4}}. 
\label{eq inég masse corrélée 2}
\end{equation}
Suppose, by contradiction, that~\eqref{eq inég masse corrélée 2} is false. Using Cauchy-Schwarz inequality, $m_{\sigma,\eps}(u,u-h,s)\leq  m_{\sigma,\eps}(u,s)^{1/2} m_{\sigma,\eps}(u-h,s)^{1/2}$, and we would deduce that:
\[
\frac{1}{M_{\sigma,\eps}(u,s)^{3/4}M_{\sigma,\eps}(u-h,s) ^{1/4}}
\leq \frac{2}{M_{\sigma,\eps}(u,s)^{1/2}M_{\sigma,\eps}(u-h,s) ^{1/2}},
\]
and thus $M_{\sigma,\eps}(u-h,s) \leq 2^4 M_{\sigma,\eps}(u,s)$. Using the fact that $M_{\sigma,\eps} \geq m_{\sigma,\eps}$, we can deduce that
\begin{equation}
\textstyle m_{\sigma,\eps}(u,s)+m_{\sigma,\eps}(u-h,s) \leq M_{\sigma,\eps}(u,s)+M_{\sigma,\eps}(u-h,s)\leq (1+2^4) \frac{h}{2^6}<\frac{h}{3}.
\label{eq maj masse}
\end{equation}
We distinguish three cases depending on the value of $N_s= y_{\sigma, \eps}(u,s)-y_{\sigma, \eps}(u-h,s)$.
\begin{description}
\item[\textbullet \,  $N_s \leq \sigma-\eta$:]
For each $v \in [u-h,u]$, one of the two terms $y_{\sigma,\eps}(u,s)-y_{\sigma,\eps}(v,s)$ and $y_{\sigma,\eps}(v,s)-y_{\sigma,\eps}(u-h,s)$ is lower than $\frac{\sigma-\eta}{2}$, which means that one of those terms belongs to the preimage of 1 by the function $\varphi_\sigma$. Hence $m_{\sigma,\eps}(u,s)+ m_{\sigma,\eps}(u-h,s) = \int_0^1 \left(\varphi_\sigma^2(y_{\sigma,\eps}(u,s)-y_{\sigma,\eps}(v,s)) +\varphi_\sigma^2(y_{\sigma,\eps}(u-h,s)-y_{\sigma,\eps}(v,s))\right)\mathrm dv\geq \int_{u-h}^u \mathrm dv=h$. 

This is in contradiction with~(\ref{eq maj masse}). Therefore inequality~(\ref{eq inég masse corrélée 2}) is satisfied in this case.  
\item[\textbullet \, $N_s \in (\sigma-\eta,\sigma)$:] Introduce $\operatorname{Med}:=\{v: y_{\sigma,\eps}(u,s)-y_{\sigma,\eps}(v,s) \in [\frac{\sigma-\eta}{2},\frac{\sigma+\eta}{2}]\}$, which is a set of particles more or less at half distance between particle $u$ and particle $u-h$. Since $\eta<\frac{\sigma}{3}$, we have $N_s > \sigma-\eta \geq \frac{\sigma+\eta}{2}$ and thus $\operatorname{Med} \subset [u-h,u]$.
Let $v\in [u-h,u]$. We distinguish three new cases: 
\begin{itemize}
\item[-] if $y_{\sigma,\eps}(u,s)-y_{\sigma,\eps}(v,s) < \frac{\sigma-\eta}{2}$, then $\varphi_\sigma(y_{\sigma,\eps}(u,s)-y_{\sigma,\eps}(v,s) )=1$. 
\item[-] if $y_{\sigma,\eps}(u,s)-y_{\sigma,\eps}(v,s) > \frac{\sigma+\eta}{2}$, and since $N_s\leq \sigma$, 
$y_{\sigma,\eps}(v,s)-y_{\sigma,\eps}(u-h,s)$ is lower than $\frac{\sigma-\eta}{2}$ and thus
 $\varphi_\sigma(y_{\sigma,\eps}(u-h,s)-y_{\sigma,\eps}(v,s) )=1$.
\item[-] otherwise, $v$ belongs to $\operatorname{Med}$.
\end{itemize}
It follows that: 
\begin{align*}
h
&=\int_{u-h}^u( \mathds 1_{\{y_{\sigma,\eps}(u,s)-y_{\sigma,\eps}(v,s) < \frac{\sigma-\eta}{2}\}}+\mathds 1_{\{y_{\sigma,\eps}(u,s)-y_{\sigma,\eps}(v,s) > \frac{\sigma+\eta}{2}\}}+\mathds 1_{\{v\in \operatorname{Med}\}}) \mathrm dv\\
&\leq \int_{u-h}^u( \varphi_\sigma^2(y_{\sigma,\eps}(u,s)-y_{\sigma,\eps}(v,s) )
+\varphi_\sigma^2(y_{\sigma,\eps}(u-h,s)-y_{\sigma,\eps}(v,s) )+\mathds 1_{\{v\in \operatorname{Med}\}}) \mathrm dv\\
&\leq m_{\sigma,\eps}(u,s)+m_{\sigma,\eps}(u-h,s)+\Leb(\operatorname{Med}).
\end{align*}
By inequality~\eqref{eq maj masse}, we deduce that $\Leb(\operatorname{Med}) >\frac{2h}{3}$.
As $\operatorname{Med}$ is an interval included in $[u-h,u]$ and  since $k \in (\frac{h}{3},\frac{2h}{3})$ we deduce that $u-k \in \operatorname{Med}$, \emph{i.e.} $\widetilde{N}_s \in [\frac{\sigma-\eta}{2},\frac{\sigma+\eta}{2}]$, which is in contradiction with the hypothesis $\widetilde{N}_s>\frac{\sigma+\eta}{2}$.
Thus inequality~(\ref{eq inég masse corrélée 2}) is also true in this case. 
\item[\textbullet \,  $N_s  \geq \sigma$:]
In this case, the two particles $u$ and $u-h$ do not have any interaction. In other words, since the support of $\varphi_\sigma$ is included in $[-\frac{\sigma}{2},\frac{\sigma}{2}]$,  $\varphi_\sigma(y_{\sigma,\eps}(u,s)-y_{\sigma,\eps}(v,s))$ and $\varphi_\sigma(y_{\sigma,\eps}(u-h,s)-y_{\sigma,\eps}(v,s))$ can not be simultaneously non-zero, whence we deduce that $m_{\sigma,\eps}(u,u-h,s)=0$. Inequality~(\ref{eq inég masse corrélée 2}) follows clearly. 
\end{description}

Therefore, inequality~\eqref{eq inég masse corrélée 2} is proved. 
By Corollary~\ref{coro crochet}, it follows that, on $G_s \cap H_s$:
\begin{align*}
\frac{\mathrm d}{\mathrm ds}\langle N,N \rangle_s
&=\frac{1}{M_{\sigma,\eps}(u,s)}
+\frac{1}{M_{\sigma,\eps}(u-h,s)}-\frac{2m_{\sigma,\eps}(u,u-h,s)}{(\eps + m_{\sigma,\eps}(u,s))(\eps + m_{\sigma,\eps}(u-h,s))} \\
&\geq \frac{1}{M_{\sigma,\eps}(u,s)}
+\frac{1}{M_{\sigma,\eps}(u-h,s)} -\frac{1}{M_{\sigma,\eps}(u,s)^{3/4}M_{\sigma,\eps}(u-h,s) ^{1/4}} \\
&\geq\frac{1}{4M_{\sigma,\eps}(u,s)}
+\frac{3}{4M_{\sigma,\eps}(u-h,s)}
\geq \frac{1}{4M_{\sigma,\eps}(u,s)}
\geq \frac{2^4}{h} ,  
\end{align*}
where we have applied a convexity inequality: $\forall a,b >0,  a^{3/4}b^{1/4}\leq \frac{3a}{4}+\frac{b}{4}$. 

To sum up, we showed that $G_s \cap H_s$ implies $\frac{\mathrm d}{\mathrm ds}\langle N,N\rangle_s \geq \frac{2^4}{h}$. 
If $\int_0^T \mathds 1_{\{G_s \cap H_s\}} \mathrm ds \geq \frac{t}{2}$, we get 
\[\langle N, N\rangle_T 
=\int_0^T \frac{\mathrm d}{\mathrm ds}\langle N,N\rangle_s \mathrm ds 
\geq \int_0^T \frac{\mathrm d}{\mathrm ds}\langle N,N\rangle_s \mathds 1_{\{G_s \cap H_s\}}\mathrm ds
\geq  \frac{2^4}{h} \int_0^T  \mathds 1_{\{G_s \cap H_s\}}\mathrm ds
\geq \frac{2^3 t }{h}.\]
Hence, since $N$ is a continuous square integrable $(\mathcal F_t^{\sigma,\eps})_{t \in [0,T]}$-martingale, there exists a standard $(\mathcal F_t^{\sigma,\eps})_{t \in [0,T]}$-Brownian motion $\beta$ such that $N_t=g(u)-g(u-h)-\beta(\langle N,N \rangle_t)$. Since $N$ remains positive on $[0,T]$ by Lemma~\ref{lemme croissance} (because $g(u-h)<g(u)$), we deduce that $\sup_{[0,\langle N,N\rangle_T]} \beta \leq g(u)-g(u-h)$. 
Therefore, 
\begin{align}
\P{\int_0^T \mathds 1_{\{G_s \cap H_s\}} \mathrm ds \geq \frac{t}{2}}
&\leq \P{\sup_{[0,\frac{2^3t}{h}]} \beta\leq g(u)-g(u-h)} \notag\\
&=\P{\sqrt{\frac{2^3}{h}}\sup_{[0,t]} \widehat{\beta}\leq g(u)-g(u-h)} \notag\\
&\leq C_2\left[g(u)-g(u-h)\right] \sqrt{ \frac{h}{t}},
\label{inég partie 1}
\end{align}
where $\widehat{\beta}$ is a rescaled Brownian motion and $C_2$ does not depend on $u$, $h$, $\sigma$, $\eps$ and $t$.

 \item[\textbullet \,  Second step: ]Study of $G_s\cap H_s^\complement$.
 
Under this event, we have  $M_{\sigma,\eps}(u,s)<\frac{h}{2^6}$ and $\widetilde{N}_s \leq \frac{\sigma+\eta}{2}$. In particular, by the assumption $\eta <\frac{\sigma}{3}$, we have $\widetilde{N}_s \leq \sigma-\eta$. 
We claim that the following inequality holds true:
\begin{equation}
\frac{ 2m_{\sigma,\eps}(u,u-k,s)}{(\eps + m_{\sigma,\eps}(u,s))(\eps + m_{\sigma,\eps}(u-k,s))} \leq \frac{1}{M_{\sigma,\eps}(u,s)^{3/4}M_{\sigma,\eps}(u-k,s) ^{1/4}}. 
\label{eq inég masse corrélée}
\end{equation}
To prove it, it is sufficient to imitate the proof of the case $N_s \leq \sigma-\eta$ of the previous step. We should notice that we did not use the hypothesis $\widetilde{N}_s >\frac{\sigma+\eta}{2}$ in that case.

Using inequality~(\ref{eq inég masse corrélée}) as in the first step, we show that $\frac{\mathrm d}{\mathrm ds}\langle \widetilde{N},\widetilde{N}\rangle_s\geq \frac{2^4}{h}$. Therefore, $\P{\int_0^T \mathds 1_{\{G_s \cap H_s^\complement\}} \mathrm ds \geq \frac{t}{2}}
\leq \P{\langle \widetilde{N},\widetilde{N} \rangle_T \geq \frac{2^3t}{h}}$.
There exists a $(\mathcal F_t^{\sigma,\eps})_{t \in [0,T]}$-Brownian motion $\widetilde{\beta}$ such that $\widetilde{N}_t=g(u)-g(u-k)-\widetilde{\beta}(\langle \widetilde{N},\widetilde{N}\rangle_t)$. Finally, we obtain the existence of a constant $C_3$ independent of $u$, $h$, $k$, $\sigma$, $\eps$ and $t$ such that:  
\begin{align}
\P{\int_0^T \mathds 1_{\{G_s \cap H_s^\complement\}} \mathrm ds \geq \frac{t}{2}}\leq \P{\sup_{[0,\frac{2^3t}{h}]} \widetilde{\beta}\leq \textstyle g(u)-g(u-k)} &\leq C_3\left[\textstyle g(u)-g(u-k)\right] \sqrt{ \frac{h}{t}} \notag\\
&\leq C_3\left[g(u)-g(u-h)\right] \sqrt{ \frac{h}{t}}.
\label{inég partie 2}
\end{align}

\end{description}

Putting together inequality~(\ref{inég deux parties}) and inequalities~(\ref{inég partie 1}) and~(\ref{inég partie 2}), we conclude the proof of inequality~(\ref{inégalité C_1}). 
 Thus inequality~\eqref{proba M} is proved for every $h$ such that $u-h \in \A$. Let $h>0$ be such that $u-h \in (0,1)$. Let $h_1 \in (\frac{h}{2},h)$ be such that $u-h_1 \in \A$. 
\begin{align*}
\P{\int_0^T \mathds 1_{\left\{M_{\sigma,\eps}(u,s)<\frac{\gamma h}{2}\right\}} \mathrm ds \geq t} \leq\P{\int_0^T \mathds 1_{\left\{M_{\sigma,\eps}(u,s)<\gamma h_1\right\}} \mathrm ds \geq t}
&\leq C \left[ g(u)-g(u-h_1) \right]\sqrt{\frac{h_1}{t}} \\
&\leq C \left[ g(u)-g(u-h) \right]\sqrt{\frac{h}{t}}.
\end{align*}
Up to replacing $\gamma$ by $\frac{\gamma}{2}$, inequality~\eqref{proba M} follows for every $h>0$ such that $u-h \in (0,1)$. 
\end{proof}

\begin{rem}
Similarly, there exist $C>0$ and $\gamma \in (0,1)$ such that for each $\sigma >0$, $\eps >0$, $t \in (0,T]$ and for every $u \in \A$ and every $h>0$ satisfying $u+h \in (0,1)$, 
\begin{align*}
\P{\int_0^T \mathds 1_{\left\{M_{\sigma,\eps}(u,s)<\gamma h\right\}} \mathrm ds \geq t} \leq C \left[ g(u+h)-g(u) \right]\sqrt{\frac{h}{t}}.
\end{align*}
\end{rem}

Thanks to Lemma~\ref{lemme controle masse} and to the above remark, we obtain the following result, which has to be compared with Proposition~4.3 in~\cite{konarovskyi17behavior}:
\begin{lemme}
\label{lemme inverse masse approchee sig eps}
Let $g \in L_p(0,1)$.
For all  $\beta \in (0, \frac{3}{2}-\frac{1}{p})$, there is a constant $C >0$ depending only on $\beta$ and $\|g\|_{L_p}$ such that for all $\sigma, \eps >0$ and $0\leq s<t\leq T$,    we have the following inequality:
\begin{align}
\E{\int_s^t \!\! \int_0^1   \frac{1}{M_{\sigma,\eps}^\beta(u,r)}\mathrm du\mathrm dr}\leq C\sqrt{t-s}. 
\label{ineg du lemme}
\end{align}
\end{lemme}
\begin{rem}
Observe that by the assumption $p>2$, made at the beginning of Section~\ref{sec convergence}, there exists some $\beta >1$ such that~(\ref{ineg du lemme}) holds. 

\end{rem}
\begin{proof}
By Fubini-Tonelli Theorem, we have: 
\begin{align*}
\E{\int_s^t \!\!\int_0^1   \frac{1}{M_{\sigma,\eps}^\beta(u,r)}\mathrm du\mathrm dr}
&=\int_0^1\E{\int_s^t \!\! \int_0^{+\infty}  \mathds 1_{\{M^{-\beta}_{\sigma,\eps}(u,r)>x\}}  \mathrm dx\mathrm dr}\mathrm du\notag\\
&\leq 2^\beta (t-s)+\int_0^1\!\!\int_{2^\beta}^{+\infty} \E{\int_s^t \mathds 1_{\{M_{\sigma,\eps}(u,r) <x^{-1/\beta}\}} \mathrm dr}  \mathrm dx\mathrm du \\
&\leq 2^\beta \sqrt{T}\sqrt{t-s} +\int_0^1\!\!\int_{2^\beta \gamma^\beta}^{+\infty} \E{\int_s^t \mathds 1_{\{M_{\sigma,\eps}(u,r) <\gamma x^{-1/\beta}\}} \mathrm dr} \gamma^{-\beta} \mathrm dx\mathrm du. 
\end{align*}
Furthermore, we compute:
\begin{align*}
\E{\int_s^t \mathds 1_{\{M_{\sigma,\eps}(u,r) <\gamma x^{-1/\beta}\}} \mathrm dr} 
&=\int_0^{t-s} \P{\int_s^t \mathds 1_{\{M_{\sigma,\eps}(u,r) <\gamma x^{-1/\beta}\}} \mathrm dr>\alpha}  \mathrm d\alpha\\
&\leq \int_0^{t-s} \P{\int_0^T \mathds 1_{\{M_{\sigma,\eps}(u,r) <\gamma x^{-1/\beta}\}} \mathrm dr>\alpha}  \mathrm d\alpha.
\end{align*}
Using Lemma~\ref{lemme controle masse}, we obtain a constant $C_1$ independent of $\sigma$ and $\eps$ such that for all $x>2^\beta$:
\begin{align*}
\int_{\frac{1}{2}}^1 \E{\int_s^t \mathds 1_{\{M_{\sigma,\eps}(u,r) <\gamma x^{-1/\beta}\}} \mathrm dr} \mathrm du 
&\leq\int_{\frac{1}{2}}^1 \!\!\int_0^{t-s} C_1\left[g(u)-g(u-x^{-1/\beta})\right]\sqrt{\frac{x^{-1/\beta}}{\alpha}} \mathrm d\alpha \mathrm du \\
&\leq 2C_1\frac{\int_{1/2}^1 (g(u)-g(u-x^{-1/\beta}))\mathrm du}{x^{1/(2\beta)}}\sqrt{t-s}.
\end{align*}
Moreover, we have for each $x>2^\beta$, using Hölder's inequality:
\begin{align}
\int_{\frac{1}{2}}^1 \left(g(u)-g(u-x^{-1/\beta})\right) \mathrm du
&= \int_0^1 \left( \mathds 1_{[\frac{1}{2},1]}(u)- \mathds 1_{[\frac{1}{2}-x^{-1/\beta},1-x^{-1/\beta}]}(u)\right)g(u) \mathrm du \notag\\
&\leq \|g\|_{L_p} (2x^{-1/\beta})^{1-\frac{1}{p}}. 
\label{ineg Holder}
\end{align}
Therefore, 
\begin{align*}
\int_{\frac{1}{2}}^1\!\!\int_{2^\beta \gamma^\beta }^{+\infty} \E{\int_s^t \mathds 1_{\{M_{\sigma,\eps}(u,r) <x^{-1/\beta}\}} \mathrm dr}  \mathrm dx\mathrm du 
&\leq C_2 \int_{2^\beta \gamma^\beta }^{+\infty}  \frac{\|g \|_{L_p}\sqrt{t-s}}{x^{\frac{1}{2\beta}}x^{\frac{1}{\beta}(1-\frac{1}{p})}} \mathrm dx \\
& \leq C_3 \|g\|_{L_p} \sqrt{t-s},
\end{align*}
where $C_2$ and $C_3$ are independent of $\sigma$, $\eps$, and $t$. The last inequality holds because $\frac{1}{\beta}\left(\frac{3}{2}-\frac{1}{p}\right)>1$. 

We conclude the proof of the Lemma by using a similar argument for $u$ belonging to $[0,\frac{1}{2}]$ and using  $g(u+x^{-1/\beta})-g(u)$ instead of $g(u)-g(u-x^{-1/\beta})$. 
\end{proof}

\begin{coro}
\label{corollaire borne L_2 sig eps}
There is a constant $C$ such that for every $t \in [0,T]$ and for every $\sigma,\eps >0$, 
\begin{align*}
\E{\int_0^1 y^2_{\sigma,\eps}(u,t) \mathrm du}\leq C. 
\end{align*}
\end{coro}
\begin{proof}
We have:
\begin{align*}
\E{\int_0^1 y^2_{\sigma,\eps}(u,t) \mathrm du}^{1/2}
\leq \E{\int_0^1 g(u)^2 \mathrm du}^{1/2} + \E{\int_0^1 (y_{\sigma,\eps}(u,t) -g(u))^2\mathrm du}^{1/2}.
\end{align*}
Since $g$ belongs to $L_2(0,1)$, the first term of the right hand side is bounded. Furthermore, by Corollary~\ref{coro crochet} and Fubini-Tonelli Theorem: 
\begin{align*}
\E{\int_0^1 (y_{\sigma,\eps}(u,t) -g(u))^2\mathrm du}
=\int_0^1 \E{\langle y_{\sigma,\eps}(u,\cdot),  y_{\sigma,\eps}(u,\cdot)\rangle_t} \mathrm du
&= \int_0^1 \E{ \int_0^t \frac{1}{M_{\sigma,\eps}(u,s)}\mathrm ds}\mathrm du\\
&\leq C\sqrt{t},
\end{align*}
by Lemma~\ref{lemme inverse masse approchee sig eps}. 
\end{proof}

\subsubsection{Proof of Proposition~\ref{prop crit tight}}
\label{subsec proof prop 4.2}
We will now use Lemma~\ref{lemme inverse masse approchee sig eps} and its Corollary~\ref{corollaire borne L_2 sig eps} to prove Proposition~\ref{prop crit tight}. We start by~$(K1)$: 
\begin{prop}
Let $g \in \mathcal L^\uparrow_{2+}[0,1]$ and $\delta$ be positive. Then there exists $M>0$ such that for all $\sigma>0$ and $\eps>0$, 
$\P{\int_0^1 \|y_{\sigma,\eps}(u,\cdot)\|_{\mathcal C[0,T]}^2\mathrm du  \geq M} \leq \delta$. 
\label{prop tension 1}
\end{prop}

\begin{proof}
Using again Fubini-Tonelli Theorem, 
\begin{align*}
\E{\int_0^1 \sup_{t \leq T} |y_{\sigma,\eps}(u,t)|^2 \mathrm du} =\int_0^1 \E{\sup_{t\leq T} |y_{\sigma,\eps}(u,t)|^2}   \mathrm du.
\end{align*}
Moreover, for almost every $u\in [0,1]$, $y_{\sigma,\eps}(u,\cdot)$ is a $(\mathcal F^{\sigma,\eps}_t)_{t \in [0,T]}$-martingale. Hence by Doob's inequality, there is a constant $C_1$ independent of $u$, $\sigma$ and $\eps$ such that:
\begin{align*}
\E{\sup_{t \leq T}|y_{\sigma,\eps}(u,t)|^2}&\leq C_1 \E{|y_{\sigma,\eps}(u,T)|^2}.
\end{align*}
Therefore, by Corollary~\ref{corollaire borne L_2 sig eps}, 
\begin{align}
\label{borne L^2}
\E{\int_0^1 \sup_{t \leq T} |y_{\sigma,\eps}(u,t)|^2 \mathrm du} \leq C_1 \int_0^1 \E{|y_{\sigma,\eps}(u,T)|^2} \mathrm du 
\leq C_2,
\end{align}
where $C_2$ is independent of $\sigma$ and $\eps$. 
We conclude by Markov's inequality: there is a constant $C>0$ such that for all $\sigma,\eps>0$, 
\[
\P{\int_0^1 \|y_{\sigma,\eps}(u,\cdot)\|_{\mathcal C[0,T]}^2\mathrm du  \geq M}  \leq \frac{\E{\int_0^1 \sup_{t \leq T} |y_{\sigma,\eps}(u,t)|^2 \mathrm du}}{M}\leq\frac{C}{M}.
\]
For $M$ large enough, that last quantity is smaller than $\delta$. 
\end{proof}

Then, we show criterion~$(K2)$:

\begin{prop}
Let $g \in L_p[0,1]$ and $\delta>0$. Then for all $k \geq 1$, there exists $\eta_k>0$ such that for every $\sigma, \eps>0$, 
\begin{align*}
\P{\int_0^1 \sup_{|t_2-t_1| < \eta_k} |y_{\sigma,\eps}(u,t_2)-y_{\sigma,\eps}(u,t_1)| \mathrm du \geq \frac{1}{k}} \leq \frac{\delta}{2^k}.
\end{align*}
\label{propo critere K2}
\end{prop}

\begin{proof}
By Markov's inequality, it is sufficient to prove that:
\begin{align}
\label{egalite crit K2}
\lim_{\eta \to 0^+} \sup_{\sigma>0,\eps>0} \E{\int_0^1  \sup_{|t_2-t_1| < \eta} |y_{\sigma,\eps}(u,t_2)-y_{\sigma,\eps}(u,t_1)| \mathrm du} =0. 
\end{align}
Fix $\delta >0$ and $\beta \in (1, \frac{3}{2}-\frac{1}{p})$. 
For every $u \in (0,1)$, define
\begin{align*}
K_1(u)&:=\E{\|y_{\sigma,\eps}(u,\cdot)\|_{\mathcal C[0,T]}} ,\\
K_2(u)&:=\E{\int_0^T \frac{1}{M_{\sigma,\eps}^\beta(u,s)}\mathrm ds}.
\end{align*}
Since $y_{\sigma,\eps}$ is uniformly bounded for $\sigma>0$ and $\eps>0$ in $L_2([0,1],\mathcal C[0,T])$ (see inequality~\eqref{borne L^2}) and by Lemma~\ref{lemme inverse masse approchee sig eps},  $\int_0^1 K_1(u)\mathrm du$ and $\int_0^1 K_2(u)\mathrm du$ are uniformly bounded for $\sigma>0$ and $\eps>0$. 
Therefore, there exists $C>0$ such that $\int_0^1 \mathds 1_{\{K_1(u)\geq C\}} \mathrm du \leq \delta$ and $\int_0^1 \mathds 1_{\{K_2(u)\geq C\}} \mathrm du \leq \delta$. We define:
\begin{align*}
K_1&:=\{ u \in (0,1): K_1(u)\leq C  \},\\
K_2&:=\{u\in (0,1): K_2(u)\leq C\}. 
\end{align*}

The collection $(y_{\sigma,\eps}(u,\cdot))_{ \sigma>0,\eps>0, u\in K_1\cap K_2}$ is tight in $\mathcal C[0,T]$. We use Aldous' tightness criterion to prove this claim (see~\cite[Theorem 16.10]{billingsley99}). We prove the two following statements:
\begin{itemize}
\item[-] $\lim_{a\to \infty} \sup_{\sigma>0,\eps>0, u \in K_1\cap K_2} \P{\|y_{\sigma,\eps}(u,\cdot)\|_{\mathcal C[0,T]} \geq a} =0$.
\item[-] for all $\alpha>0$ and $r>0$, there is $\eta_0$ such that for all $\eta\in (0,\eta_0)$, for all $\sigma>0$, $\eps>0$ and $u \in K_1\cap K_2$, if $\tau$ is a stopping time for $y_{\sigma,\eps}(u,\cdot)$ such that $\tau \leq T$, then $\P{|y_{\sigma,\eps}(u,\tau+\eta)-y_{\sigma,\eps}(u,\tau)| \geq r} \leq\alpha$. 
\end{itemize}
By Markov's inequality, for all $a>0$, $\sigma>0$, $\eps>0$ and $u \in K_1\cap K_2$, 
\begin{align*}
\P{\|y_{\sigma,\eps}(u,\cdot)\|_{\mathcal C[0,T]} \geq a}
\leq \frac{1}{a}\E{\|y_{\sigma,\eps}(u,\cdot)\|_{\mathcal C[0,T]}}
=\frac{K_1(u)}{a}
\leq\frac{C}{a},
\end{align*}
whence we obtain the first statement. Moreover, for all $u\in K_1\cap K_2$, by Hölder's inequality, 
\begin{align*}
\E{|y_{\sigma,\eps}(u,\tau+\eta)-y_{\sigma,\eps}(u,\tau)|^2}=\E{\int_\tau^{\tau+\eta} \frac{1}{M_{\sigma,\eps}(u,s)}\mathrm ds}
\leq  K_2(u)^{\frac{1}{\beta}} \eta^{1-\frac{1}{\beta}}
 \leq C^{\frac{1}{\beta}} \eta^{1-\frac{1}{\beta}},
\end{align*}
whence we obtain the second statement. 

By Aldous' tightness criterion, there exists a compact $L$ of the set $\mathcal D[0,T]$ of càdlàg functions on $[0,T]$ such that for all $\sigma>0$, $\eps>0$ and $u\in K_1\cap K_2$, $\P{y_{\sigma,\eps}(u,\cdot) \in L} \geq 1-\delta$. 
Since $\mathcal C[0,T]$ is closed in $\mathcal D[0,T]$ with respect to Skorohod's topology, and $y_{\sigma,\eps}(u,\cdot) \in \mathcal C[0,T]$ almost surely, we may suppose that $L$ is a compact set of $\mathcal C[0,T]$. 

Back to~\eqref{egalite crit K2}, we have: 
\begin{align}
\label{egal K1 K2 L}
\E{\int_0^1  \sup_{|t_2-t_1| < \eta} |y_{\sigma,\eps}(u,t_2)-y_{\sigma,\eps}(u,t_1)| \mathrm du}
=\int_0^1 \E{  \sup_{|t_2-t_1| < \eta} |y_{\sigma,\eps}(u,t_2)-y_{\sigma,\eps}(u,t_1)|}\mathrm du\notag
\\ 
\begin{aligned}
&=\int_0^1 \E{ \mathds 1_{\{u \in K_1\cap K_2, y_{\sigma,\eps}(u,\cdot)\in L\}^\complement} \sup_{|t_2-t_1| < \eta} |y_{\sigma,\eps}(u,t_2)-y_{\sigma,\eps}(u,t_1)|}\mathrm du\\
&\quad+\int_0^1 \E{ \mathds 1_{\{u \in K_1\cap K_2, y_{\sigma,\eps}(u,\cdot)\in L\}} \sup_{|t_2-t_1| < \eta} |y_{\sigma,\eps}(u,t_2)-y_{\sigma,\eps}(u,t_1)|}\mathrm du.
\end{aligned}
\end{align}
The first term on the right hand side of~\eqref{egal K1 K2 L} is bounded by:
\begin{align*}
\left( \int_0^1 \E{ \mathds 1_{\{u \in K_1\cap K_2, y_{\sigma,\eps}(u,\cdot)\in L\}^\complement}}\mathrm du \right)^{1/2}
\left(\int_0^1 \E{ \sup_{|t_2-t_1| < \eta} |y_{\sigma,\eps}(u,t_2)-y_{\sigma,\eps}(u,t_1)|^2}\mathrm du \right)^{1/2}. 
\end{align*}
We have:
\begin{align*}
\int_0^1 \E{ \mathds 1_{\{u \in K_1\cap K_2, y_{\sigma,\eps}(u,\cdot)\in L\}^\complement}}\mathrm du
&\leq \int_0^1 \mathds 1_{\{u \in K_1\cap K_2\}} \P{y_{\sigma,\eps}(u,\cdot) \notin L}\mathrm du\\
&\quad+\int_0^1 \mathds 1_{\{K_1(u) \geq C\}}\mathrm du
+\int_0^1 \mathds 1_{\{K_2(u) \geq C\}}\mathrm du\\
&\leq 3 \delta. 
\end{align*}
Moreover, 
\begin{align*}
\int_0^1 \E{ \sup_{|t_2-t_1| < \eta} |y_{\sigma,\eps}(u,t_2)-y_{\sigma,\eps}(u,t_1)|^2}\mathrm du
\leq 4\int_0^1 \E{\sup_{t \leq T}|y_{\sigma,\eps}(u,t)|^2}\mathrm du\leq 4M,
\end{align*}
where $M$ is a constant independent of $\sigma>0$ and $\eps>0$ by inequality~\eqref{borne L^2}. 

It remains to handle the second term on the right hand side of~\eqref{egal K1 K2 L}. Since $L$ is a compact set of $\mathcal C[0,T]$, there exists $\eta>0$ such that for every $f \in L$, $\omega_f(\eta):=\sup_{|t-s|<\eta} |f(t)-f(s)| <\delta$. 
Therefore, there exists $\eta >0$ such that:
\begin{align*}
\int_0^1 \E{ \mathds 1_{\{u \in K_1\cap K_2, y_{\sigma,\eps}(u,\cdot)\in L\}} \sup_{|t_2-t_1| < \eta} |y_{\sigma,\eps}(u,t_2)-y_{\sigma,\eps}(u,t_1)|}\mathrm du \leq \delta.
\end{align*}
Back to equality~\eqref{egal K1 K2 L}, we have proved that there is $\eta>0$ such that for every $\sigma>0$ and $\eps>0$:
\begin{align*}
\E{\int_0^1  \sup_{|t_2-t_1| < \eta} |y_{\sigma,\eps}(u,t_2)-y_{\sigma,\eps}(u,t_1)| \mathrm du}
\leq \delta + \sqrt{12\delta M}.
\end{align*}
This proves convergence~\eqref{egalite crit K2} and thus concludes the proof of the Proposition. 
\end{proof}

Then, to obtain criterion~$(K3)$, we state the following Proposition:
\begin{prop}
Let $g \in \mathcal L^\uparrow_{2+}[0,1]$ and $\delta >0$. Then for all $k \geq1$, there is $h_k>0$ such that for all $\sigma, \eps >0$, 
\[
\P{\int_0^{1-h_k}\|y_{\sigma,\eps}(u+h_k,\cdot)-y_{\sigma,\eps}(u,\cdot)\|^2_{\mathcal C[0,T]}  \mathrm du\geq \frac{1}{k}}\leq \frac{\delta}{2^k}.
\]
\label{prop tension 2}
\end{prop}

If $\int_0^{1-h_k}\|y_{\sigma,\eps}(u+h_k,\cdot)-y_{\sigma,\eps}(u,\cdot)\|^2_{\mathcal C[0,T]}  \mathrm du\leq \frac{1}{k}$, we deduce by monotonicity of $u \mapsto y_{\sigma,\eps}(u,t)$ for every $t \in [0,T]$ that for every $h \in (0,h_k)$, 
\begin{multline*}
\int_0^{1-h} \|y_{\sigma,\eps}(u+h,\cdot)-y_{\sigma,\eps}(u,\cdot)\|^2_{\mathcal C[0,T]}  \mathrm du\\
\leq \int_0^{1-h_k} \!\!\!\!\!\!\!\!\|y_{\sigma,\eps}(u+h,\cdot)-y_{\sigma,\eps}(u,\cdot)\|^2_{\mathcal C[0,T]}  \mathrm du +\int_{1-2h_k+h}^{1-h_k}\!\!\!\!\!\!\!\! \|y_{\sigma,\eps}(u+h_k,\cdot)-y_{\sigma,\eps}(u+h_k-h,\cdot)\|^2_{\mathcal C[0,T]}  \mathrm du\\
\leq 2\int_0^{1-h_k} \|y_{\sigma,\eps}(u+h_k,\cdot)-y_{\sigma,\eps}(u,\cdot)\|^2_{\mathcal C[0,T]}  \mathrm du \leq \frac{2}{k}.
\end{multline*}
Therefore, the latter Proposition implies the following Corollary, which is equivalent to criterion~$(K3)$:

\begin{coro}
Let $g \in \mathcal L^\uparrow_{2+}[0,1]$ and $\delta >0$. Then for all $k \geq 1$, there is $h_k>0$ such that for all $\sigma ,\eps>0$,
\[
\P{\forall h \in (0,h_k), \int_0^{1-h} \|y_{\sigma,\eps}(u+h,\cdot)-y_{\sigma,\eps}(u,\cdot)\|^2_{\mathcal C[0,T]}  \mathrm du\leq \frac{2}{k}}\geq1- \frac{\delta}{2^k}.
\]
\label{coro tension 2}
\end{coro}

\begin{proof}[Proof (Proposition~\ref{prop tension 2})]
Let $h\in(0,1)$. By Corollary~\ref{coro version continue ps}, for almost every $u \in (0,1-h)$,  $N_{u,t}:=y_{\sigma,\eps}(u+h,t)-y_{\sigma,\eps}(u,t)$ is a martingale. By Fubini-Tonelli Theorem and Doob's inequality, we have:
\begin{align}\label{ineg initiale}
\E{\int_0^{1-h}\|N_{u,\cdot}\|^2_{\mathcal C[0,T]}  \mathrm du}
&=\int_0^{1-h}\E{\|N_{u,\cdot}\|^2_{\mathcal C[0,T]}}  \mathrm du
\leq C \int_0^{1-h}\E{N_{u,T}^2}  \mathrm du.
\end{align}
Let us split $\E{N_{u,T}^2}$ in two terms $\E{N_{u,T}^2 \mathds 1_{\{N_{u,T} \leq 1\}}}+\E{N_{u,T}^2 \mathds 1_{\{N_{u,T} >1\}}}$. 

\begin{description}
\item[Study of $\int_0^{1-h}\E{N_{u,T}^2 \mathds 1_{\{N_{u,T} \leq 1\}}}\mathrm du$. ]

Let $u\in (0,1-h)$ be such that $N_{u,\cdot}$ is a martingale.  
By Lemma~\ref{lemme croissance}, if $g(u+h)-g(u)=0$, then $N_{u,T}=0$ almost surely, thus $\E{N_{u,T}^2 \mathds 1_{N_{u,T} \leq 1}}=0$. From now on, we suppose that $g(u+h)-g(u)>0$.
$N_{u,\cdot}$ is a square integrable continuous martingale, starting from $g(u+h)-g(u)>0$ and positive by Lemma~\ref{lemme croissance}. 
Therefore, there exists a standard Brownian motion $\beta_u$ such that $N_{u,t}=N_{u,0}+\beta_u(\langle N_{u,\cdot},N_{u,\cdot} \rangle_t)$. Recall that $N_{u,0}=g(u+h)-g(u)$ is a deterministic quantity. 
If $N_{u,0}\geq 1$, then the inequality $\E{N_{u,T}^2 \mathds 1_{\{N_{u,T} \leq 1\}}}\leq N_{u,0}$ is obvious. Otherwise, we have
\begin{multline}
\E{N_{u,T}^2 \mathds 1_{\{N_{u,T} \leq 1\}}} 
=\int_0^{+\infty}\!\! \P{N_{u,T}^2 \mathds 1_{\{N_{u,T} \leq 1\}}\geq \lambda}   \mathrm d\lambda
\leq \int_0^1 \P{N_{u,T}^2 \geq \lambda} \mathrm d\lambda  \\
\leq N_{u,0}^2+\int_{N_{u,0}^2}^1 \P{N_{u,T} \geq \lambda^{1/2}} \mathrm d\lambda. 
\label{eq: lemme 2}
\end{multline}

Let us estimate $\P{N_{u,T} \geq \kappa}$ for a real number $\kappa > N_{u,0}$. 
We define the following stopping times: 
\begin{align*}
\tau_{-N_{u,0}}&:=\inf \{ t \geq 0: N_{u,0} +\beta_u(t) \leq 0 \}; \\
\tau_{\kappa-N_{u,0}}&:=\inf \{ t \geq 0: N_{u,0} +\beta_u(t) \geq \kappa \}; \\
\tau &:=\inf \{ t \geq 0: N_{u,t}\geq \kappa \}\wedge T.
\end{align*}

On the first hand, we know that almost surely, for all  $t \in [0,T]$, $N_{u,t} >0$, hence $\tau_{-N_{u,0}} \geq \langle N_{u,\cdot},N_{u,\cdot} \rangle_T$. 
On the other hand, if $N_{u,T} \geq \kappa$, $N_{u,\tau}$ is equal to $\kappa$ by continuity of $N_{u,\cdot}$, hence $\langle N_{u,\cdot},N_{u,\cdot} \rangle_\tau \geq \tau_{\kappa -N_{u,0}}$. It follows from both inequalities that $\tau_{\kappa-N_{u,0}} \leq \tau_{-N_{u,0}}$. Therefore, 
\begin{equation}
\label{eq proba N_T}
\P{N_{u,T} \geq \kappa} \leq \P{\tau_{\kappa-N_{u,0}} \leq \tau_{-N_{u,0}}} =\frac{N_{u,0}}{\kappa},
\end{equation}
by a usual martingale equality. 
Using inequality~(\ref{eq: lemme 2}) and $N_{u,0}\leq 1$, we obtain: 
\begin{align*}
\E{N_{u,T}^2\mathds 1_{\{N_{u,T}\leq 1\}}} \leq N_{u,0}^2+\int_{N_{u,0}^2}^1 \frac{N_{u,0}}{\lambda^{1/2}} \mathrm d\lambda
\leq N_{u,0}^2+2N_{u,0}\leq 3 N_{u,0}.
\end{align*}
Therefore, we have: $\int_0^{1-h}\E{N_{u,T}^2 \mathds 1_{\{N_{u,T} \leq 1\}}}\mathrm du\leq 3\int_0^{1-h} N_{u,0}\mathrm du$.

\item[Study of $\int_0^{1-h} \E{N_{u,T}^2 \mathds 1_{\{N_{u,T} > 1\}}}\mathrm du$.]
Recall that $g$ belongs to $L_p(0,1)$ for some $p>2$. Fix $\beta \in (1,\frac{3}{2}-\frac{1}{p})$. We compute:
\begin{align*}
\int_0^{1-h} \!\!\! \E{N_{u,T}^2 \mathds 1_{\{N_{u,T} > 1\}}}\mathrm du
&\leq 2\int_0^{1-h} \!\!\! \E{(N_{u,T}-N_{u,0})^2 \mathds 1_{\{N_{u,T} > 1\}}}\mathrm du \\
&\quad
+2\int_0^{1-h} \!\!\! \E{N_{u,0}^2 \mathds 1_{\{N_{u,T} > 1\}}}\mathrm du\\
&\leq 2\left(\int_0^{1-h} \!\!\!\!\!\!\E{(N_{u,T}-N_{u,0})^{2\beta} }\mathrm du\right)^{\frac{1}{\beta}}
 \left(\int_0^{1-h}\!\!\!\!\!\!\P{N_{u,T}>1}\mathrm du\right) ^{1-\frac{1}{\beta}}\\
 &\quad+2\int_0^{1-h} N_{u,0}^2\mathrm du.
\end{align*}
Furthermore, we have $\P{N_{u,T}>1}\leq N_{u,0}$: that inequality is obvious if $N_{u,0}\geq 1$ and otherwise, it is a consequence of inequality~\eqref{eq proba N_T}. 

Then, we want to give an upper bound for $\E{(N_{u,T}-N_{u,0})^{2\beta}}$. 
Using Burkholder-Davis-Gundy inequality, there exists $C_\beta$ such that $\E{(N_{u,T}-N_{u,0})^{2\beta}}\leq C_\beta\E{\langle N_{u,\cdot},N_{u,\cdot} \rangle_T^\beta}$. 
We compute the quadratic variation of the martingale $N_{u,t}=y_{\sigma,\eps}(u+h,t)-y_{\sigma,\eps}(u,t)$:
\begin{multline*}
\E{\langle N_{u,\cdot},N_{u,\cdot} \rangle_T^\beta}\\
=\E{\left|\int_0^T \left( \frac{1}{M_{\sigma,\eps}(u,s)}+\frac{1}{M_{\sigma,\eps}(u+h,s)} -\frac{2m_{\sigma,\eps} (u,u+h,s)}{(\eps + m_{\sigma,\eps}(u,s)) (\eps + m_{\sigma,\eps}(u+h,s))}  \right) \mathrm ds\right| ^\beta }. 
\end{multline*}
By Cauchy-Schwarz inequality $m_{\sigma,\eps} (u,u+h,s)\leq m_{\sigma,\eps}^{1/2}(u,s) m_{\sigma,\eps}^{1/2}(u+h,s)$, we deduce that the sum of the three terms in the integral is non-negative and thus that it is bounded by $\frac{1}{M_{\sigma,\eps}(u,s)}+\frac{1}{M_{\sigma,\eps}(u+h,s)} $, whence we obtain:
\begin{align*}
\E{\langle N_{u,\cdot},N_{u,\cdot} \rangle_T^\beta}
&\leq T^{\beta-1} \E{\int_0^T \left| \frac{1}{M_{\sigma,\eps}(u,s)}+\frac{1}{M_{\sigma,\eps}(u+h,s)}  \right|^\beta \mathrm ds }\\
&\leq C_{\beta,T} \left( \E{\int_0^T \frac{\mathrm ds}{M_{\sigma,\eps}^\beta(u,s)}}+ \E{\int_0^T \frac{\mathrm ds}{M_{\sigma,\eps}^\beta(u+h,s)}}\right).
\end{align*}
By Lemma~\ref{lemme inverse masse approchee sig eps}, we deduce that $\int_0^{1-h}\E{\langle N_{u,\cdot},N_{u,\cdot} \rangle_T^\beta}\mathrm du$ is bounded, because $\beta<\frac{3}{2}-\frac{1}{p}$.
Therefore, 
we can conclude that there is a constant $C_{T,\beta}$ such that:
\[
\int_0^{1-h}\E{N_{u,T}^2\mathds 1_{\{N_{u,T} > 1\}}}\mathrm du
\leq 2 C_{T,\beta} \left(\int_0^{1-h} \!N_{u,0}\;\mathrm du\right)^{1-1/\beta}+2\int_0^{1-h}\! N_{u,0}^2 \;\mathrm du.
\]
\end{description}
\textbf{Conclusion:} Putting together the studies of both cases, we have proved that there is a positive constant $C$ satisfying, for all $\sigma$, $\eps$ and $h \in (0,1)$:
\begin{align}\label{ineg finale}
\int_0^{1-h} \E{N_{u,T}^2} \mathrm du 
\leq C\int_0^{1-h} \!N_{u,0}\;\mathrm du + C \left(\int_0^{1-h} \!N_{u,0}\;\mathrm du\right)^{1-1/\beta}+C\int_0^{1-h} \!N_{u,0}^2 \;\mathrm du.
\end{align}
Recall that there is $p>2$ such that $g\in L_p(0,1)$. As for inequality~\eqref{ineg Holder}, we get:
\begin{align*}
\int_0^{1-h}\! N_{u,0}\;\mathrm du=\int_0^{1-h} (g(u+h)-g(u))\mathrm du \leq \|g\|_{L_p} (2h)^{1-\frac{1}{p}}. 
\end{align*}
Furthermore, define $\alpha:=\frac{p-2}{p-1} \in (0,1)$. We have
\begin{align*}
\int_0^{1-h} \!\!N_{u,0}^2\; \mathrm du
&=\int_0^{1-h} (g(u+h)-g(u))^\alpha (g(u+h)-g(u))^{2-\alpha} \mathrm du \\
&\leq \left(\int_0^{1-h} (g(u+h)-g(u))\mathrm du \right)^\alpha \left( \int_0^{1-h} (g(u+h)-g(u))^{\frac{2-\alpha}{1-\alpha}}\mathrm du\right)^{1-\alpha}\\
&\leq \left( \|g\|_{L_p} (2h)^{1-\frac{1}{p}} \right)^\alpha \left( C_p \|g\|_{L_p}   \right)^{1-\alpha},
\end{align*}
because $\frac{2-\alpha}{1-\alpha}=p$. Therefore
\begin{align}
\int_0^{1-h} \!N_{u,0}^2 \;\mathrm du
=\int_0^{1-h} (g(u+h)-g(u))^2  \mathrm du
&\leq C_p ^{1-\alpha} \|g\|_{L_p} h^{\frac{p-2}{p}}. 
\label{ineg g carré}
\end{align}
It follows from~\eqref{ineg finale} that there is $C_\beta$ such that for each $\sigma ,\eps >0$, 
\begin{align*}
\int_0^{1-h} \E{(y_{\sigma,\eps}(u+h,T)-y_{\sigma,\eps}(u,T))^2} \mathrm du 
\leq C_\beta \|g\|_{L_p} \left( h^{\frac{p-1}{p}}+h^{\frac{p-1}{p}(1-\frac{1}{\beta})}+h^{\frac{p-2}{p}}\right),
\end{align*}
for every $\beta < \frac{3}{2}-\frac{1}{p}$, \textit{i.e.} such that  $0<1-\frac{1}{\beta}<\frac{p-2}{3p-2}$. 
Thus, there is $q>0$ depending on $p$ (e.g. $q=\frac{(p-1)(p-2)}{2p(3p-2)}$ by choosing $1-\frac{1}{\beta}=\frac{p-2}{2(3p-2)}$) and a constant $C$ such that for each $\sigma ,\eps >0$, 
\begin{align}
\label{h polynomial}
\int_0^{1-h} \E{(y_{\sigma,\eps}(u+h,T)-y_{\sigma,\eps}(u,T))^2} \mathrm du 
\leq C \|g\|_{L_p} h^q.
\end{align}
Therefore, by~\eqref{ineg initiale} and Markov's inequality, there is $C$ such that for each $\sigma ,\eps >0$, 
\begin{align*}
\P{\int_0^{1-h}\|y_{\sigma,\eps}(u+h,\cdot)-y_{\sigma,\eps}(u,\cdot)\|^2_{\mathcal C[0,T]}  \mathrm du\geq \frac{1}{k}}\leq k C \|g\|_{L_p} h^q,
\end{align*} 
whence it is sufficient to choose $h_k$ so that $k C \|g\|_{L_p} h_k^q< \frac{\delta}{2^k}$. 
\end{proof}

\subsection{Convergence when $\eps \to 0$}
\label{par conv eps 0}

Fix $\sigma \in \Q_+$. By Prokhorov's Theorem, it follows from Corollary~\ref{coro tension y sig eps} that the collection of laws of the sequence $(y_{\sigma,\eps})_{\eps \in \Q_+}$ is relatively compact in $\mathcal P(L_2([0,1],\mathcal C[0,T]))$. In particular, up to extracting a subsequence, we may suppose that $(y_{\sigma,\eps})_{\eps \in \Q_+}$ converges in distribution in $L_2([0,1],\mathcal C[0,T])$ to a limit, denoted by $y_\sigma$.

For every $t \in [0,T]$, let us denote by $e_t(f):=f(\cdot,t)$ the continuous  evaluation function: $L_2([0,1],\mathcal C[0,T]) \to  L_2[0,1]$.
We define $Y_\sigma(t):=e_t(y_\sigma)=y_\sigma(\cdot,t)$. Under the same model as Proposition~\ref{proprietes A1-A3}, we obtain:
\begin{prop}
\label{proprietes B1-B3}
Fix $\sigma \in \Q_+$. 
Suppose that $g \in \mathcal L^\uparrow_{2+}[0,1]$. $(Y_\sigma(t))_{t \in [0,T]}$ is a $L^\uparrow_2[0,1]$-valued process such that:
\begin{itemize}
\item[$(B1)$] $Y_\sigma(0)=g$;

\item[$(B2)$] $(Y_\sigma(t))_{t \in [0,T]}$ is a square integrable continuous $L^\uparrow_2[0,1]$-valued $(\mathcal F^\sigma_t)_{t \in [0,T]}$-martingale, where $\mathcal F^\sigma_t=\sigma(Y_\sigma(s),s\leq t)$; 

\item[$(B3)$]
for every $h,k \in L_2[0,1]$, 
\begin{align*}
\langle (Y_\sigma,h)_{L_2},(Y_\sigma,k)_{L_2}\rangle_t = \int_0^t \!\! \int_0^1\!\!\int_0^1 h(u)k(u') \frac{m_\sigma(u,u',s)}{m_\sigma(u,s)m_\sigma(u',s)}\mathrm du\mathrm du' \mathrm ds,
\end{align*}
where $m_\sigma(u,u',s) =\int_0^1 \varphi_\sigma (y_\sigma(u,s)-y_\sigma(v,s)) \varphi_\sigma (y_\sigma(u',s)-y_\sigma(v,s)) \mathrm dv$ and \newline $m_\sigma(u,s)=\int_0^1 \varphi_\sigma^2(y_\sigma(u,s)-y_\sigma(v,s))\mathrm dv$. 
\end{itemize}
\end{prop}

\begin{proof}
Fix $t \in [0,T]$. We want to prove that $Y_\sigma(t)$ belongs to $L^\uparrow_2[0,1]$. 
For each $\eps \in \Q_+$, $Y_{\sigma,\eps}(t)$ belongs with probability $1$ to the set $\mathcal K:=$
\[
\left\{f\in L_2(0,1): 
\forall u,u', \forall r,r', \text{if }0<u<u+r<u'<u'+r'<1,\text{then } \frac{1}{r}\int_u^{u+r} \!\!\! f \leq \frac{1}{r'} \int_{u'}^{u'+r'}\!\!\! f\right\}
\]
which is closed in $L_2(0,1)$. Recall that the sequence $(y_{\sigma,\eps})_{\eps\in \Q_+}$ converges in distribution to~$y_\sigma$ in $L_2([0,1],\mathcal C[0,T])$. Therefore, $(Y_{\sigma,\eps}(t))_{\eps\in \Q_+}$ converges in distribution to $Y_\sigma(t)$ in $L_2[0,1]$. 
Because $\mathcal K$ is closed, the limit $Y_\sigma(t)$ also belongs to $\mathcal K$ with probability 1.

Therefore, almost surely, for every $t\in [0,T] \cap \Q$, $Y_\sigma(t) \in \mathcal K$. 
Let $\omega \in \Omega'$, where $\Omega'$ is such that $\P{\Omega'}=1$ and for every $\omega \in \Omega'$, $\int_0^1 \sup_{s \leq T} |y_\sigma(v,s)|^2(\omega) \mathrm dv <+\infty$ and for every $t\in [0,T] \cap \Q$, $Y_\sigma(t)(\omega) \in \mathcal K$. 
Let $t \in [0,T]$ and $(t_n)$ be a sequence in $ [0,T]\cap \Q$ tending to $t$. For every $n \in \N$ and each $u,u',r,r'$ such that $0<u<u+r<u'<u'+r'<1$, $\frac{1}{r} \int_u^{u+r} y_\sigma(v,t_n)(\omega) \mathrm dv \leq \frac{1}{r'}\int_{u'}^{u'+r'} y_\sigma(v,t_n)(\omega) \mathrm dv$. 
Since $y_\sigma(\omega)$ belongs to $L_2([0,1], \mathcal C[0,T])$, and since $\int_u^{u+r} y_\sigma(v,t_n)^2(\omega) \mathrm dv \leq \int_0^1 \sup_{s \leq T} |y_\sigma(v,s)|^2(\omega) \mathrm dv <+\infty$, $\frac{1}{r} \int_u^{u+r} y_\sigma(v,t_n)(\omega) \mathrm dv$ tends to $\frac{1}{r} \int_u^{u+r} y_\sigma(v,t)(\omega) \mathrm dv$ (and the same is true for $u'$ and $r'$). Thus almost surely $Y_\sigma(t) $ belongs to $\mathcal K$ for every $t\in [0,T]$. 
It remains to prove that it implies that $Y_\sigma(t)$ belongs to $L^\uparrow_2[0,1]$.

Let $f \in \mathcal K$. 
Define, for each $u\in (0,1)$, $\widehat{f}(u):=\liminf_{h \to 0^+} \frac{1}{h} \int_u^{(u+h)\wedge 1} f(v) \mathrm dv$. 
First, remark that $\widehat{f}$ is non-decreasing. Then, since $h \mapsto \frac{1}{h}\int_u^{u+h} f$ is non-increasing, we have $\widehat{f}(u)=\lim_{h \to 0^+}  \frac{1}{h} \int_u^{(u+h)\wedge 1} f(v) \mathrm dv$. 
Choose a sequence $(u_n) \searrow u$. By monotonicity, $\widehat{f}(u)\leq \widehat{f}(u_n)$. Fix $\delta >0$. There exists $h>0$ such that $u+h <1$ and $|\widehat{f}(u)-\frac{1}{h}\int_u^{u+h}f| <\delta$. Since $f \in L_2$, there exists~$N$ such that for all $n \geq N$, $|\frac{1}{h}\int_{u_n}^{u_n+h}f -\frac{1}{h}\int_u^{u+h} f|<\delta$. 
Therefore, $\widehat{f}(u_n)\leq \frac{1}{h}\int_{u_n}^{u_n+h} f \leq \widehat{f}(u)+2\delta$ for all $n \geq N$. Thus $\widehat{f}(u_n) \to \widehat{f}(u)$. In addition, $\widehat{f}$ has left limits because of its monotonicity. Hence $\widehat{f}$ is a càdlàg function. 

Furthermore, $\widehat{f}=f$ almost everywhere. Indeed, for every $\delta>0$, there exists $F \in \mathcal C[0,1]$ such that $\|f-F\|_{L_1(0,1)} <\delta$. 
Define $\widehat{F}(u)=\lim_{h \to 0^+}  \frac{1}{h} \int_u^{(u+h)\wedge 1}F(v) \mathrm dv$. By continuity of $F$, $F(u)=\widehat{F}(u)$ for every $u \in (0,1)$. Thus we have:
\begin{align*}
\int_0^1 |f(u)- \widehat{f}(u)|\mathrm du 
&\leq \int_0^1 |f(u)-F(u)|\mathrm du +\int_0^1 |\widehat{f}(u)-\widehat{F}(u)|\mathrm du\\
&\leq \delta +\int_0^1 \lim_{h \to 0^+} \frac{1}{h}\int_u^{(u+h)\wedge 1} |f(v)-F(v)|\mathrm dv \mathrm du\\
&\leq \delta + \liminf_{h \to 0^+} \int_0^1 |f(v)-F(v)|\mathrm dv \leq 2\delta,
\end{align*}
where we used Fatou's Lemma to obtain the last line. Thus $\int_0^1 |f(u)- \widehat{f}(u)|\mathrm du=0$, whence $\widehat{f}=f$ almost everywhere. Thus $f$ belongs to $L^\uparrow_2[0,1]$: $Y_\sigma$ is a $L^\uparrow_2[0,1]$-valued process.  

\paragraph{Property $(B1)$.} 
$(Y_{\sigma,\eps}(0))_{\eps\in \Q_+}$ converges in law to $Y_\sigma(0)$ in $L_2[0,1]$. Therefore, $Y_\sigma(0)=g$.

\paragraph{Property $(B2)$.} 
By inequality~\eqref{borne L^2}, $\E{\|Y_{\sigma,\eps}\|^2_{L_2([0,1],\mathcal C[0,T])} }$ is bounded uniformly in $\eps \in \Q_+$. 
We deduce that for every $t \in [0,T]$, $\E{\|Y_{\sigma}(t)\|^2_{L_2([0,1])} }<+\infty$, thus the process $Y_\sigma$ is square integrable. 

Furthermore, $Y_\sigma$ is a continuous $L^\uparrow_2[0,1]$-valued process. Indeed, for each sequence $(t_n)_{n\geq 0}$ converging to a time $t$, 
$\|Y_\sigma(t_n)-Y_\sigma(t)\|_{L_2}^2=\int_0^1 (y_\sigma(u,t_n)-y_\sigma(u,t))^2\mathrm du \underset{n \to \infty}{\longrightarrow} 0$
by dominated convergence Theorem, since for almost every $u \in (0,1)$, $y_\sigma(u,\cdot)$ is continuous at time $t$, and $(y_\sigma(u,t_n)-y_\sigma(u,t))^2 \leq 4 \sup_{t \leq T}|y_\sigma(u,t)|^2$ which is almost surely integrable.

Moreover, we know from property~$(A2)$ that for each $h \in L_2(0,1)$, each $l \geq 1$, $0\leq s_1\leq s_2\leq \dots \leq s_l\leq s\leq t$ and each bounded and continuous function $f_l:(L_2(0,1))^l \rightarrow \R$:
\begin{align}
\label{martingale avant fubini}
\E{\int_0^1 h(u) (y_{\sigma,\eps}(u,t)-y_{\sigma,\eps}(u,s))\mathrm du \; f_l(y_{\sigma,\eps}(\cdot,s_1),\dots,y_{\sigma,\eps}(\cdot,s_l))}=0.
\end{align}
Since $\left| \int_0^1  h(u) b(u,t) \mathrm du \right| \leq \| h\|_{L_2} \left( \int_0^1 \sup_{[0,T]} |b(u,\cdot)|^2 \mathrm du\right)^{1/2}$ for every $b \in L_2([0,1],\mathcal C[0,T])$, the function $\varphi: b \in L_2([0,1],\mathcal C[0,T]) \mapsto \int_0^1 h(u) (b(u,t)-b(u,s))\mathrm du\; f_l(b(\cdot,s_1),\dots, b(\cdot,s_l)) $ is continuous. 
Furthermore, we prove that $(\varphi(y_{\sigma,\eps}))_{\eps \in \Q_+}$ is bounded in $L_2$:
\begin{align*}
\E{\varphi(y_{\sigma,\eps})^2}
&\leq \|f_l\|_{\infty}^2 \| h\|_{L_2}^2 \E{\int_0^1 (y_{\sigma,\eps}(u,t)-y_{\sigma,\eps}(u,s))^2 \mathrm du}\\
&\leq C \|f_l\|_{\infty}^2 \| h\|_{L_2}^2,
\end{align*}
where $C$ is independent of $\eps$ by Corollary~\ref{corollaire borne L_2 sig eps}. 
We deduce that $(\varphi(y_{\sigma,\eps}))_{\eps \in \Q_+}$ is uniformly integrable. By continuity of $\varphi$ and since $(y_{\sigma,\eps})_{\eps \in \Q_+}$ converges in law to $y_\sigma$ in $ L_2([0,1],\mathcal C[0,T])$, we get: 
$\E{\varphi(y_{\sigma,\eps})}\underset{\eps \to 0}{\longrightarrow} \E{\varphi(y_\sigma)}$. Since by equality~\eqref{martingale avant fubini}, $\E{\varphi(y_{\sigma,\eps})}=0$ for each $\eps \in \Q_+$, we have:
\begin{align}
\label{eq martingale limite}
\E{\int_0^1 h(u) (y_\sigma(u,t)-y_\sigma(u,s))\mathrm du f_l(Y_\sigma(s_1),\dots,Y_\sigma(s_l))}=0.
\end{align}
Therefore, $Y_\sigma(\cdot)$ is a square integrable continuous $(\mathcal F^\sigma_t)_{t \in [0,T]}$-martingale.

\paragraph{Property $(B3)$.} 
We know, by property~$(A3)$, that for every $l\geq 1$, for every $0\leq s_1\leq s_2\leq \dots \leq s_l\leq s\leq t$, for every bounded and continuous $f_l:(L_2(0,1))^l \rightarrow \R$ and for every $h$ and $k$ in $L_2(0,1)$:
\begin{multline}
\label{big equation}
\mathbb E \bigg[ \int_0^1 \!\!\int_0^1 h(u)k(u')  [ (y_{\sigma,\eps}(u,t)-g(u))(y_{\sigma,\eps}(u',t)-g(u'))
\\-(y_{\sigma,\eps}(u,s)-g(u))(y_{\sigma,\eps}(u',s)-g(u')) ] 
\mathrm du du' f_l(Y_{\sigma,\eps}(s_1),\dots,Y_{\sigma,\eps}(s_l))   \bigg]
\\ =\mathbb E \bigg[\int_0^1\!\!\int_0^1 h(u)k(u')  \int_s^t \frac{m_{\sigma,\eps}(u,u',r)}{(\eps+m_{\sigma,\eps}(u,r))(\eps+m_{\sigma,\eps}(u',r))}\mathrm dr \mathrm du du'
 f_l(Y_{\sigma,\eps}(s_1),\dots,Y_{\sigma,\eps}(s_l))  \bigg].
\end{multline}

First, we want to obtain the convergence of the left hand side of~\eqref{big equation}. We proceed in the same way as for the proof of equality~\eqref{eq martingale limite}; to get a uniform integrability property, we have now to prove the existence of $\beta >1$ such that
\begin{align}
\label{ui for t}
\sup_{\eps \in \Q_+}\E{\left(\int_0^1 h(u)(y_{\sigma,\eps}(u,t)-g(u))\mathrm du  \int_0^1  k(u')(y_{\sigma,\eps}(u',t)-g(u'))\mathrm du'\right)^\beta}
\end{align}
is finite. 
Therefore, it is sufficient to prove the existence of $\beta>1$ such that
\begin{align*}
\sup_{\eps \in \Q_+}\E{\left(\int_0^1 h(u)(y_{\sigma,\eps}(u,t)-g(u))\mathrm du \right)^{2\beta}}
\end{align*}
is finite for every $h \in L_2[0,1]$. 
By Cauchy-Schwarz inequality, 
\begin{align}
\label{CS ineq}
\E{\left(\int_0^1 h(u)(y_{\sigma,\eps}(u,t)-g(u))\mathrm du \right)^{2\beta}}
&\leq \E{\|h\|_{L_2}^{2\beta} \left( \int_0^1 (y_{\sigma,\eps}(u,t)-g(u))^2 \mathrm du \right)^{\beta}}\notag \\
&\leq \|h\|_{L_2}^{2\beta} \E{ \int_0^1 (y_{\sigma,\eps}(u,t)-g(u))^{2\beta} \mathrm du }.
\end{align}
We deduce by Burkholder-Davis-Gundy inequality and Fubini's Theorem that there are some constants independent of $\eps$ such that
\begin{align*}
\E{\int_0^1(y_{\sigma,\eps}(u,t)-g(u))^{2\beta}\mathrm du}
&\leq C_1\int_0^1\E{\langle y_{\sigma,\eps}(u,\cdot),y_{\sigma,\eps}(u,\cdot)\rangle_t^\beta}\mathrm du \notag
\\&\leq C_2\E{\int_0^1\!\!\int_0^t \frac{1}{M_{\sigma,\eps}^\beta(u,r)}\mathrm dr\mathrm du}.
\end{align*}
By Lemma~\ref{lemme inverse masse approchee sig eps}, there exists $\beta >1$ such that $\E{\int_0^1\int_0^t \frac{1}{M_{\sigma,\eps}^\beta(u,r)}\mathrm dr\mathrm du}$ is bounded uniformly for $\eps \in \Q_+$. Thus~\eqref{ui for t} is finite. It is also finite if we replace $t$ by $s$. 

To obtain the convergence of the right hand side of~\eqref{big equation}, we start by using Skorohod's representation Theorem\footnote{$L_2([0,1],\mathcal C[0,T])$ is a Polish space. Its separability can be proved using the separability of $C([0,1]\times[0,T])$ and the density of $C([0,1]\times[0,T])$ in $L_2([0,1],\mathcal C[0,T])$.}: there exists a sequence $(\widehat{y}_{\sigma,\eps})_{\eps \in \Q_+}$ defined on a common probability space $(\widehat{\Omega},\widehat{\mathbb P})$ that converges to $\widehat{y}_\sigma$ in $L_2([0,1],\mathcal C[0,T])$ almost surely, where $\widehat{y}_{\sigma,\eps}$ (resp. $\widehat{y}_\sigma$) has same distribution as $y_{\sigma,\eps}$ (resp. $y_\sigma$). 
We denote by $\widehat{m}_{\sigma,\eps}$ (resp. $\widehat{m}_\sigma$) the mass associated to $\widehat{y}_{\sigma,\eps}$ (resp. $\widehat{y}_\sigma$). 

Furthermore, on the probability space $(\widehat{\Omega}\times [0,1], \widehat{\mathbb P} \otimes \Leb|_{[0,1]})$, $\widehat{y}_{\sigma,\eps}$ converges in probability in the space $\mathcal C[0,T]$ to $\widehat{y}_\sigma$. Indeed, for every $\delta >0$, we have:
\begin{align*}
\widehat{\mathbb P} \otimes \Leb|_{[0,1]}
\{ (\omega,u): \|(\widehat{y}_{\sigma,\eps}-\widehat{y}_\sigma)(\omega,u)\|_{\mathcal C[0,T]} \geq \delta\}
&=\widehat{\mathbb E} \left[  \Leb 
\{ u: \|(\widehat{y}_{\sigma,\eps}-\widehat{y}_\sigma)(\omega,u)\|_{\mathcal C[0,T]} \geq \delta  \}    \right] \\
&\leq \widetilde{\mathbb E} \left[1\wedge \frac{1}{\delta^2} \int_0^1  \|(\widehat{y}_{\sigma,\eps}-\widehat{y}_\sigma)(\omega,u)\|_{\mathcal C[0,T]}^2 \mathrm du  \right].
\end{align*}
We know that, for every fixed $\delta>0$, $1\wedge \frac{1}{\delta^2}\int_0^1  \|(\widehat{y}_{\sigma,\eps}-\widehat{y}_\sigma)(\omega,u)\|_{\mathcal C[0,T]}^2 \mathrm du $ converges to $0$ almost surely, and it is bounded by $1$, so we deduce that the latter term tends to $0$. We deduce from the convergence in probability that there exists a subsequence $(\eps_n)_n$, $\eps_n \to 0$, such that for almost every $(\omega, u) \in \widehat{\Omega}\times[0,1]$, $\|(\widehat{y}_{\sigma,\eps_n}-\widehat{y}_\sigma)(\omega,u)\|_{\mathcal C[0,T]} \to 0$. 

We want to prove that,  
\begin{multline}
\label{conv masses}
\E {\int_0^1\!\!\int_0^1 h(u)k(u') \! \int_s^t \!\frac{\widehat{m}_{\sigma,\eps_n}(u,u',r)}{(\eps_n+\widehat{m}_{\sigma,\eps_n}(u,r))(\eps_n+\widehat{m}_{\sigma,\eps_n}(u',r))}\mathrm dr \mathrm du \mathrm du'f_l(\widehat{Y}_{\sigma,\eps_n}(s_1),\dots,\widehat{Y}_{\sigma,\eps_n}(s_l)) }
\\\underset{n \to \infty}{\longrightarrow}
\E {\int_0^1\!\!\int_0^1 h(u)k(u')  \int_s^t \frac{\widehat{m}_\sigma(u,u',r)}{\widehat{m}_\sigma(u,r)\widehat{m}_\sigma(u',r)}\mathrm dr \mathrm du \mathrm du'f_l(\widehat{Y}_\sigma(s_1),\dots,\widehat{Y}_\sigma(s_l)) }.
\end{multline}
On the one hand, almost surely and for almost every $u\in (0,1)$, $\widehat{y}_{\sigma,\eps_n}(u,\cdot) \to \widehat{y}_\sigma(u,\cdot)$ in $\mathcal C[0,T]$. Then for almost every $u,u' \in (0,1)$, 
\begin{align}
\label{convergence masse 1}
\widehat{m}_{\sigma,\eps_n}(u,u',r)\!
=\!\!\int_0^1 \!\!\!\varphi_\sigma(\widehat{y}_{\sigma,\eps_n}(u,r)-\widehat{y}_{\sigma,\eps_n}(v,r))\varphi_\sigma(\widehat{y}_{\sigma,\eps_n}(u',r)-\widehat{y}_{\sigma,\eps_n}(v,r))\mathrm dv
&\!\underset{n \to \infty }{\longrightarrow} \!\widehat{m}_{\sigma}(u,u',r),\\
\eps_n+\widehat{m}_{\sigma,\eps_n}(u,r)
=\eps_n+\int_0^1 \varphi_\sigma^2(\widehat{y}_{\sigma,\eps_n}(u,r)-\widehat{y}_{\sigma,\eps_n}(v,r))\mathrm dv
& \!\underset{n \to \infty }{\longrightarrow}\!\widehat{m}_{\sigma}(u,r).
\label{convergence masse 2}
\end{align}
Therefore, in order to obtain~\eqref{conv masses}, it remains to justify that there exists $\beta >1$ such that:
\begin{align*}
\sup_{n \in \mathds N}\E{\left( \int_0^1\!\!\int_0^1h(u)k(u') \int_s^t  \frac{\widehat{m}_{\sigma,\eps_n}(u,u',r)}{(\eps_n + \widehat{m}_{\sigma,\eps_n}(u,r))(\eps_n + \widehat{m}_{\sigma,\eps_n}(u',r))}  \mathrm dr \mathrm du \mathrm du'\right)^\beta}
\end{align*}
is finite. 
By Cauchy-Schwarz inequality, $\widehat{m}_{\sigma,\eps_n}(u,u',r)\leq \widehat{m}^{1/2}_{\sigma,\eps_n}(u,r)\widehat{m}^{1/2}_{\sigma,\eps_n}(u',r)$, so that it is sufficient to prove that there is $\beta >1$ such that 
\[
\sup_{n \in \mathds N}\E{\left(\int_0^1\!\!\int_0^1h(u)k(u')\int_s^t  \frac{1}{\widehat{M}_{\sigma,\eps_n}^{1/2}(u,r)\widehat{M}_{\sigma,\eps_n}^{1/2}(u',r)}  \mathrm dr \mathrm du \mathrm du'\right)^\beta}
\]
is finite, and thus that 
$\sup_{n \in \mathds N}\E{\int_0^1\int_s^t  \frac{1}{\widehat{M}_{\sigma,\eps_n}^{\beta}(u,r)}  \mathrm dr \mathrm du }$ is finite, using Cauchy-Schwarz inequality as in the proof of~\eqref{CS ineq}. By Lemma~\ref{lemme inverse masse approchee sig eps}, this statement holds. We conclude that we have the following equality:
\begin{multline}
\label{martingale crochet sigma}
\mathbb E \bigg[ \int_0^1\!\!\int_0^1 h(u)k(u')  [ (y_\sigma(u,t)-g(u))(y_\sigma(u',t)-g(u'))
\\-(y_\sigma(u,s)-g(u))(y_\sigma(u',s)-g(u')) ] 
\mathrm du du' f_l(Y_\sigma(s_1),\dots,Y_\sigma(s_l))   \bigg]
\\ =\mathbb E \bigg[\int_0^1\!\!\int_0^1 h(u)k(u')  \int_s^t \frac{m_\sigma(u,u',r)}{m_\sigma(u,r)m_\sigma(u',r)}\mathrm dr \mathrm du du'
 f_l(Y_\sigma(s_1),\dots,Y_\sigma(s_l))  \bigg],
\end{multline}
whence we obtain property~$(B3)$, since $\int_0^1\!\!\int_0^1 h(u)k(u')  \int_0^t \frac{m_\sigma(u,u',r)}{m_\sigma(u,r)m_\sigma(u',r)}\mathrm dr \mathrm du du'$ is $(\mathcal F^\sigma_t)_{t \in [0,T]}$-measurable. 
\end{proof}

Property~$(B3)$ implies the following Corollary:
\begin{coro}
\label{coro marting 2}
Let $\psi$ be a non-negative and bounded map: $[0,1]\to \R$. Then for every $l \in \mathds N \backslash \{0\}$, $0\leq s_1\leq s_2\leq \dots \leq s_l \leq s\leq t$ and for every bounded and continuous function $f_l:L_2[0,1]^l \to \R$, we have:
\begin{multline*}
\mathbb E\bigg[\int_0^1\psi(u)\left( (y_\sigma(u,t)-g(u))^2-(y_\sigma(u,s)-g(u))^2 -\int_s^t \frac{1}{m_\sigma(u,r)}\mathrm dr \right)\mathrm du \\ f_l(Y_\sigma(s_1),\dots,Y_\sigma(s_l))\bigg]=0.
\end{multline*}
\end{coro}

\begin{proof}
We use the following notations: $z(u,\cdot):=y_\sigma(u,\cdot)-g(u)$ and $F_l=f_l(Y_\sigma(s_1),\dots,Y_\sigma(s_l))$. 
Let us consider an orthonormal basis $(e_i)_{i \geq 1}$ in the Hilbert space $L_2(\psi(x)\mathrm dx)$. We denote by $[\cdot,\cdot]_{L_2(\psi)}$ the scalar product of $L_2(\psi(x) \mathrm dx)$: $[h,k]_{L_2(\psi)}=\int_0^1 hk\psi$. 
By Parseval's formula, we have:
\begin{multline*}
\begin{aligned}
\E{\int_0^1 \psi(u) (z(u,t)^2-z(u,s)^2) \mathrm du F_l}
&= \E{ \sum_{i \geq 1} ( [z(\cdot,t),e_i]^2_{L_2(\psi)} -[z(\cdot,s),e_i]^2_{L_2(\psi)}) F_l}\\
&= \sum_{i \geq 1} \E{  ( (z(\cdot,t),e_i\psi)^2_{L_2} -(z(\cdot,s),e_i\psi)^2_{L_2}) F_l} \end{aligned}\\
= \sum_{i \geq 1} \E{ \int_0^1\!\!\int_0^1 e_i(u)\psi(u)e_i(u')\psi(u') \int_s^t \frac{m_\sigma(u,u',r)}{m_\sigma(u,r)m_\sigma(u',r)}\mathrm dr \mathrm du \mathrm du'  F_l},
\end{multline*}
by applying equality~\eqref{martingale crochet sigma} with $h=k=e_i$. 
By definition of $m_\sigma(u,u',r)$, we have:
\begin{align*}
\E{\int_0^1 \psi(u) (z(u,t)^2-z(u,s)^2) \mathrm du F_l}&
= 
 \E{ \int_s^t\!\! \int_0^1\sum_{i \geq 1}
 \left[\frac{\varphi_\sigma(y_\sigma(\cdot,r)-y_\sigma(v,r))}{m_\sigma(\cdot,r)}  , e_i\right]_{L_2(\psi)}^2 
 \mathrm dv\mathrm dr F_l} \\
 &= \E{ \int_s^t\!\! \int_0^1\!\!\int_0^1
 \frac{\varphi_\sigma^2(y_\sigma(u,r)-y_\sigma(v,r))}{m^2_\sigma(u,r)}  
 \psi(u) \mathrm du\mathrm dv\mathrm dr F_l} \\
 &= \E{\int_0^1 \!\!  \int_s^t
 \frac{ 1 }{m_\sigma(u,r)}  \mathrm dr
  \psi(u) \mathrm du F_l},
 \end{align*}
 since $m_\sigma(u,r)=\int_0^1 \varphi_\sigma^2(y_\sigma(u,r)-y_\sigma(v,r)) \mathrm dv$. 
\end{proof}

We deduce the following estimation, by analogy with Lemma~\ref{lemme inverse masse approchee sig eps}:
\begin{lemme}
\label{lemme inverse masse approchee sigma}
For all $\beta \in (0, \frac{3}{2}-\frac{1}{p})$, there is a constant $C >0$  such that for all $\sigma>0$ and $0\leq s<t\leq T$, we have the following inequality:
\begin{align*}
\E{\int_s^t \!\! \int_0^1   \frac{1}{m_\sigma^\beta(u,r)}\mathrm du\mathrm dr}\leq C\sqrt{t-s}. 
\end{align*}
\end{lemme}

\begin{proof}
We use again the sequence $(\widehat{y}_{\sigma,\eps_n})_{n \in \mathds N}$ obtained by Skorohod's representation Theorem, as in the proof of convergence~\eqref{conv masses}.
Therefore, by Fatou's Lemma, 
\begin{align*}
\E{\int_s^t \!\! \int_0^1 \frac{1}{\widehat{m}_\sigma^\beta(u,r)}\mathrm du\mathrm dr}\leq
\liminf_{n \to \infty}  \;\E{\int_s^t \!\!\int_0^1  \frac{1}{\widehat{M}_{\sigma,\eps_n}^\beta(u,r)}\mathrm du\mathrm dr}\leq C \sqrt{t-s},
\end{align*}
where $C$ is obtained thanks to Lemma~\ref{lemme inverse masse approchee sig eps}. 
\end{proof}

By Burkholder-Davis-Gundy inequality, we obtain immediately the following Corollary:
\begin{coro}
For each $\beta \in (0, \frac{3}{2}-\frac{1}{p})$, $\displaystyle \sup_{\sigma \in \Q_+} \sup_{t \leq T}\E{\int_0^1(y_\sigma(u,t)-g(u))^{2\beta}\mathrm du}<+\infty$. 
\label{corollaire borne L_2 sig}
\end{coro}

\subsection{Convergence when $\sigma \to 0$}
\label{par conv sigma 0}

Recall that by Corollary~\ref{coro tension y sig eps} and Prokhorov's Theorem, the collection of laws of the sequence $(y_{\sigma,\eps})_{\sigma, \eps \in \Q_+}$ is relatively compact in $\mathcal P(L_2([0,1],\mathcal C[0,T]))$. By construction, the collection of laws of the sequence $(y_\sigma)_{\sigma \in \Q_+}$ inherits the same property. 

Thus, up to extracting a subsequence, we may suppose that $(y_\sigma)_{\sigma\in \Q_+}$ converges in distribution to a limit, denoted by $y$, in $L_2([0,1],\mathcal C[0,T])$. 
As before, we define $Y(t):=y(\cdot,t)$. We state the first part of Theorem~\ref{théo 2} in the following Proposition: 
\begin{prop}
\label{proprietes C1-C2}
Suppose that $g \in \mathcal L^\uparrow_{2+}[0,1]$. $(Y(t))_{t \in [0,T]}$ is a $L^\uparrow_2[0,1]$-valued process such that:
\begin{itemize}
\item[$(C1)$] $Y(0)=g$;

\item[$(C2)$] $(Y(t))_{t \in [0,T]}$ is a square integrable continuous $L^\uparrow_2[0,1]$-valued $(\mathcal F_t)_{t \in [0,T]}$-martingale, where $\mathcal F_t=\sigma(Y(s),s\leq t)$.
\end{itemize}
\end{prop}

\begin{proof}
We refer to the proof of Proposition~\ref{proprietes B1-B3}. 
\end{proof}

\begin{rem}
It should be noticed at this point that a new difficulty arises when we want to obtain a property analogous to~$(B3)$. Indeed, whereas it was straightforward to prove~\eqref{convergence masse 1} and~\eqref{convergence masse 2}, the convergence of $m_\sigma(u,t)=\int_0^1 \varphi_\sigma^2(y_\sigma(u,t)-y_\sigma(v,t))\mathrm dv$ to $m(u,t)=\int_0^1 \mathds 1_{\{ y(u,t)=y(v,t) \}} \mathrm dv$ is not obvious, due to the singularity of the indicator function. It will be the main goal of the next Section to prove this convergence. 

\end{rem}

In Section~\ref{sec properties}, we will study  the martingale properties of the limit process $Y$ and compute its quadratic variation (property~$(C5)$ of Theorem~\ref{théo 2}). To obtain this, we will first prove that for every \emph{positive} $t$, $Y(t)$ is a step function (see property~$(C3)$). It implies that $y$ has a version in $\mathcal D((0,1),\mathcal C[0,T])$ (see property~$(C4)$) by an argument given in~(\cite[Proposition~2.3]{konarovskyi17behavior}).

\section{Properties of the limit process $Y$}
\label{sec properties}

The aim of this Section is to complete the proof of Theorem~\ref{théo 2}. Properties~$(C3)$ and~$(C4)$ will be proved in Paragraph~\ref{para 1} and property~$(C5)$ will be proved in two steps in Paragraph~\ref{para 2} and Paragraph~\ref{para 3}.

\subsection{Coalescence properties and step functions}
\label{para 1}

In this Paragraph, we will prove the following Proposition:
\begin{prop}
\label{prop step function}
Almost surely, for every $t>0$, $Y(t)$ is a step function. 
\end{prop}
Recall that $Y(0)=g$ is not necessarily a step function, since $g$ can be chosen arbitrarily in $\mathcal L^\uparrow_{2+}[0,1]$. If we denote for each $t \in [0,T]$ by $\mu_t$ the measure associated to the quantile function $Y(t)$, that is $\mu_t= \Leb|_{[0,1]} \circ Y(t)^{-1}$, Proposition~\ref{prop step function} means that for every \emph{positive} time $t$, $\mu_t$ is a finite weighted sum of Dirac measures. 
We begin by the following Lemma. Recall the definition of the mass: $m(u,t)=\int_0^1 \mathds 1_{\{ y(u,t)=y(v,t) \}} \mathrm dv$. 
\begin{lemme}
\label{lemme limsup masse}
There exists a probability space $(\widetilde{\Omega}, \widetilde{\mathbb P})$ on which the sequence $(\widetilde{y}_\sigma)_{\sigma\in \Q_+}$ converges almost surely to  $\widetilde{y}$ in $L_2([0,1],\mathcal C[0,T])$ and where, for each $\sigma \in \Q_+$, $\widetilde{y}_\sigma$ (resp. $\widetilde{y}$) has same law as $y_\sigma$ (resp. $y$). 
Furthermore, there is a subsequence $(\sigma_n)_n$, $\sigma_n \to 0$, such that for almost every $(\omega,u)\in \Omega\times (0,1)$ and for every time $t \in [0,T]$, 
\begin{align*}
\limsup_{n\to \infty} \widetilde{m}_{\sigma_n}(u,t) \leq \widetilde{m}(u,t). 
\end{align*}
\end{lemme}

\begin{proof}
Recall that $(y_\sigma)_{\sigma \in \Q_+}$ converges in distribution in $L_2([0,1],\mathcal C[0,T])$ to $y$. By Skorohod's representation Theorem, we deduce that there exists a sequence $(\widetilde{y}_\sigma)_{\sigma\in \Q_+}$ and a random variable~$\widetilde{y}$ defined on a common probability space $(\widetilde{\Omega},\widetilde{\mathbb P})$ such that for every $\sigma \in \Q_+$, the laws of $\widetilde{y}_\sigma$ and $y_\sigma$ are the same, the laws of $\widetilde{y}$ and $y$ are also equal and the sequence $(\widetilde{y}_\sigma)_{\sigma\in \Q_+}$ converges almost surely to $\widetilde{y}$ in $L_2([0,1],\mathcal C[0,T])$.

For every $\eps>0$, we get by Markov's inequality:
\begin{align}
\label{Markov}
\widetilde{\mathbb P} \otimes \Leb 
\{ (\omega,u): \|(\widetilde{y}_\sigma-\widetilde{y})(\omega,u)\|_{\mathcal C[0,T]} \geq \eps  \} \notag
&=\widetilde{\mathbb E} \left[  \Leb 
\{ u: \|(\widetilde{y}_\sigma-\widetilde{y})(\omega,u)\|_{\mathcal C[0,T]} \geq \eps  \}    \right] \\
&\leq \widetilde{\mathbb E} \left[1\wedge \frac{1}{\eps^2} \int_0^1  \|(\widetilde{y}_\sigma-\widetilde{y})(\omega,u)\|_{\mathcal C[0,T]}^2 \mathrm du  \right].
\end{align}
Since $(\widetilde{y}_\sigma)_{\sigma\in \Q_+}$ converges almost surely to $\widetilde{y}$ in $L_2([0,1],\mathcal C[0,T])$, the right hand side tends to $0$. 
Therefore, $(\widetilde{y}_\sigma)_{\sigma \in \Q_+}$ converges in probability to $\widetilde{y}$ in $\mathcal C[0,T]$ on the probability space $(\widetilde{\Omega} \times [0,1], \widetilde{\mathbb P} \otimes \Leb)$. 
Thus there exists a subsequence $(\sigma_n)_n$ tending to $0$ along which $\widetilde{y}_{\sigma_n}$ converges on an almost sure event of $\widetilde{\Omega} \times [0,1]$ to $\widetilde{y}$ in $\mathcal C[0,T]$. Therefore, there is $\Omega'$, $\widetilde{\mathbb P}[\Omega']=1$, such that for every $\omega \in \Omega'$, there exists a Borel set $\A=\A(\omega)$ in $[0,1]$, $\Leb(\A)=1$, such that for all $u \in \A$, $\|\widetilde{y}_{\sigma_n}(u,\cdot)-\widetilde{y}(u,\cdot)\|_{\mathcal C[0,T]}$ tends to zero. Remark that the extraction $(\sigma_n)_n$ does not depend on $\omega$. From now on, we forget the tildes and the extraction in our notation.

Let $\omega \in \Omega$. Fix $u\in \A(\omega)$ and $t\in[0,T]$.
We set $v \in \A$ such that $y(v,t)\neq y(u,t)$. Then there exist $\sigma_0 >0$ and $\delta >0$ such that for all $\sigma \in (0,\sigma_0)\cap \Q_+$, $|y_\sigma(v,t)-y_\sigma(u,t)| \geq \delta$.
For all $\sigma \leq \min(\sigma_0,\delta)$, we have $|y_\sigma(v,t)-y_\sigma(u,t)| \geq \sigma$ and thus $\varphi_\sigma(y_\sigma(v,t)-y_\sigma(u,t))=0$. Hence, 
$\lim_{\sigma \rightarrow 0} \left(1-\varphi_\sigma^2(y_\sigma(v,t)-y_\sigma(u,t))\right)=1$. 
Thus we have shown that for all $v \in \A$, 
\[
\mathds 1_{\{y(v,t)\neq y(u,t)\}} \leq \liminf_{\sigma \rightarrow 0} \left(1-\varphi_\sigma^2(y_\sigma(v,t)-y_\sigma(u,t))\right),
\]
since $1-\varphi_\sigma^2$ is non-negative. By Fatou's Lemma and since $\Leb(\A)=1$, we deduce that:
\[
1-m(u,t)=\int_0^1 \mathds 1_{\{y(v,t)\neq y(u,t)\}} \mathrm dv\leq \liminf_{\sigma \rightarrow 0} \int_0^1\left(1-\varphi_\sigma^2(y_\sigma(v,t)-y_\sigma(u,t))\right)\mathrm dv,
\]
whence for all $u\in \A$ and $t\in [0,T]$, $\limsup_{n \to \infty} m_{\sigma_n}(u,t) \leq m(u,t)$.
\end{proof}

We deduce from Lemma~\ref{lemme limsup masse} the following Corollary. 
Set $N(t):=\int_0^1 \frac{\mathrm du}{m(u,t)}$. 
By a classical combinatorial argument, $N(t)$ is the number of equivalence classes at time $t$ relatively to the equivalence relation $u \underset{t}{\sim} v \iff y(u,t)=y(v,t)$. In other words, if $N(t)<\infty$, $Y(t)$ is a càdlàg step function taking $N(t)$ distinct values: there exist
$0=a_1<a_2<\dots<a_{N(t)}<a_{N(t)+1}=1$ and $y_1<y_2<\dots<y_{N(t)}$ such that for all $u\in [0,1]$
\begin{align*}
Y(t)(u)=\sum_{k=1}^{N(t)} y_k \mathds 1_{\{u \in [a_k,a_{k+1})\}}+y_{N(t)} \mathds 1_{\{u=1\}}.
\end{align*}

\begin{coro}
For every time $t \in [0,T]$, $\E{\int_0^t N(s) \mathrm ds}$ is finite. 
\label{coro int N(t) finite}
\end{coro}

\begin{proof}
By Lemma~\ref{lemme limsup masse},  there is a subsequence $(\sigma_n)$ such that almost surely, for every $t\in [0,T]$ and for almost every $u\in [0,1]$, $\limsup_{n \to \infty} m_{\sigma_n}(u,t)\leq m(u,t)$. Therefore, $\frac{1}{m(u,t)}\leq \liminf\limits_{n \to \infty}\frac{1}{m_{\sigma_n}(u,t)}$. By Fatou's Lemma, we deduce that:
\[
\E{\int_0^t N(s)\mathrm ds}\leq  \E{\int_0^t \!\!\int_0^1 \liminf_{n \to \infty}\frac{1}{m_{\sigma_n}(u,t)} \mathrm du\mathrm ds} \leq \liminf_{n \to \infty} \E{\int_0^t \!\!\int_0^1 \frac{1}{m_{\sigma_n}(u,s)} \mathrm du\mathrm ds} \leq C \sqrt{t},
\]
by Lemma~\ref{lemme inverse masse approchee sigma}. 
\end{proof}

\begin{coro}
Almost surely, for every $t >0$, $N(t)$ is finite and $t \mapsto N(t)$ is non-increasing on $(0,T]$. 
\label{coro N(t) finite}
\end{coro}

\begin{proof}
We begin by proving the coalescence property. Let $u_1$, $u_2$, $h \in \Q$ be such that $0<u_1<u_1+h<u_2<u_2+h<1$. 
Define $y^h(u_1,t)=\frac{1}{h} \int_{u_1}^{u_1+h} y(v,t) \mathrm dv=(Y(t),\frac{1}{h} \mathds 1_{(u_1,u_1+h)})_{L_2}$ and $y^h(u_2,t)=(Y(t),\frac{1}{h} \mathds 1_{(u_2,u_2+h)})_{L_2}$. 
By Proposition~\ref{proprietes C1-C2}, $Z(t)=y^h(u_2,t)-y^h(u_1,t)$ is a continuous $\R$-valued $(\mathcal F_t)_{t \in [0,T]}$-martingale, almost surely non-negative. As a consequence, $Z(t)=0$ for every $t \geq \tau_0=\inf\{ s \geq 0, Z(s)=0\}$. 
In other terms, the following coalescence property holds: for every $u_1$, $u_2$, $h \in \Q$ such that $0<u_1<u_1+h<u_2<u_2+h<1$, $y^h(u_1,t_0)=y^h(u_2,t_0)$ implies $y^h(u_1,t)=y^h(u_2,t)$ for every $t \geq t_0$ almost surely. 

On a full event $\Omega'$ of $(\Omega, \mathbb P)$, the latter statement is true and $\int_0^T N(s) \mathrm ds$ is finite (by Corollary~\ref{coro int N(t) finite}). 
Fix $\omega \in \Omega'$. In particular, for almost every $t \in (0,T)$, $N(t)$ is finite. Let $t_0 \in (0,T)$ be such that $N(t_0)<+\infty$. There exist $0=a_1<a_2<\dots<a_{N(t_0)}<a_{N(t_0)+1}=1$ and $z_1<z_2<\dots<z_{N(t_0)}$, depending on $\omega$,  such that for all $u\in [0,1]$,
\begin{align*}
Y(t_0)(u)=\sum_{k=1}^{N(t_0)} z_k \mathds 1_{\{u \in [a_k,a_{k+1})\}}+z_{N(t_0)} \mathds 1_{\{u=1\}}.
\end{align*}

Fix $k \in \{1,\dots,N(t_0)\}$. By the coalescence property, almost surely, for all $u_1$, $u_2$, $h \in \Q$ such that $a_k < u_1 <u_1+h <u_2<u_2+h<a_{k+1}$, since $y^h(u_1,t_0)=z_k=y^h(u_2,t_0)$, we have $y^h(u_1,t)=y^h(u_2,t)$ for every $t \geq t_0$. Fix $t \geq t_0$. By monotonicity of $Y(t)$, we deduce that $Y(t)$ is constant on $(u_1,u_2+h)$. Thus  $Y(t)$ is constant on $(a_k, a_{k+1})$. 
Therefore, since $Y(t)$ is càdlàg, there exist $\widetilde{z}_1 \leq \widetilde{z}_2 \leq \dots\leq \widetilde{z}_{N(t_0)}$, depending on $\omega$, such that for all $u \in [0,1]$, 
\begin{align*}
Y(t)(u)=\sum_{k=1}^{N(t_0)} \widetilde{z}_k \mathds 1_{\{u \in [a_k,a_{k+1})\}}+\widetilde{z}_{N(t_0)} \mathds 1_{\{u=1\}}.  
\end{align*}
We deduce that  $N(t)\leq N(t_0)<+\infty$, for every $t \geq t_0$. Therefore, for every $\omega \in \Omega'$, $t \mapsto N(t)$ is finite and non-increasing on $(0,T]$. 
This concludes the proof of the Lemma. 
\end{proof}

Therefore, Corollary~\ref{coro N(t) finite} concludes the proof of Proposition~\ref{prop step function}. 
Then, Proposition~\ref{proprietes C1-C2} and Proposition~\ref{prop step function} imply the following property, by applying Proposition 2.3 of~\cite{konarovskyi17behavior}:
\begin{prop}
\label{propriete B4}
There exists a modification $\widetilde{y}$ of $y$ in $L_2([0,1],\mathcal C[0,T])$ such that $\widetilde{y}$ belongs to $\mathcal D((0,1),\mathcal C[0,T])$. In particular, for every $t \in [0,T]$, $y(\cdot,t)$ and $\widetilde{y}(\cdot,t)$ are equal in $L_2[0,1]$ almost surely. Moreover, for every $u \in (0,1)$, $\widetilde{y}(u, \cdot)$ is a square integrable and continuous $(\mathcal F_t)_{t \in [0,T]}$-martingale and 
\begin{align*}
\P{\forall u,v \in (0,1), \forall s \in [0,T], \widetilde{y}(u,s)=\widetilde{y}(v,s) \text{ implies } \forall t \geq s, \widetilde{y}(u,t)=\widetilde{y}(v,t)}=1. 
\end{align*}
\end{prop}
From now on, we denote by $y$ (instead of $\widetilde{y}$) the version of the limit process in $\mathcal D((0,1),\mathcal C[0,T])$. 

\begin{rem}
The proof can be found in Appendix~B of~\cite{konarovskyi17behavior}. It should be noticed that the difficult part of the proof relies on the construction of a version $\widetilde{y}$ 
such that for every $u \in (0,1)$, $\widetilde{y}(u,\cdot)$ is continuous at time $t=0$. 
\end{rem}

This concludes the proof of properties~$(C3)$ and $(C4)$ of Theorem~\ref{théo 2}. The aim of the next two Paragraphs is to prove property~$(C5)$, in two steps.

\subsection{Quadratic variation of $y(u,\cdot)$}
\label{para 2}

The following Proposition shows that the quadratic variation of a particle is proportional to the inverse of its mass:
\begin{prop} Let $y$ be the version in $\mathcal D((0,1),\mathcal C[0,T])$ of the limit process given by Proposition~\ref{propriete B4}.   
For every $u \in (0,1)$,
\[
\langle y(u,\cdot),y(u,\cdot)\rangle_t = \int_0^t \frac{1}{m(u,s)}\mathrm ds,
\]
where $m(u,s)=\int_0^1 \mathds 1_{\{ y(u,s)=y(v,s) \}} \mathrm dv$.
\label{propriete C5 partie 1}
\end{prop}

\begin{proof}
By Corollary~\ref{coro marting 2}, for every positive $\psi \in L_\infty(0,1)$, we have:
\begin{multline}
\label{eg martingale C4}
\E{\int_0^1 \psi(u)  [(y_\sigma(u,t)-g(u))^2- (y_\sigma(u,s)-g(u))^2]   f_l(Y_\sigma(s_1),\dots,Y_\sigma(s_l))\mathrm du }
\\=\E{\int_0^1 \psi(u) \int_s^t \frac{1}{m_\sigma(u,r)} \mathrm dr \;  f_l(Y_\sigma(s_1),\dots,Y_\sigma(s_l))\mathrm du }. 
\end{multline}

To obtain the convergence of the left hand side of~\eqref{eg martingale C4}, we proceed in the same way as for the proof of equality~\eqref{martingale crochet sigma}. The uniform integrability property follows from Corollary~\ref{corollaire borne L_2 sig}. 
Therefore, the left hand side of~\eqref{eg martingale C4} converges when $\sigma \to 0$ to 
\begin{align*}
\E{\int_0^1 \psi(u)  [(y(u,t)-g(u))^2- (y(u,s)-g(u))^2]   f_l(Y(s_1),\dots,Y(s_l))\mathrm du }.
\end{align*} 

We also get a uniform integrability property for the right hand side of~\eqref{eg martingale C4} by the same argument as in the proof of property~$(B3)$ (see Proposition~\ref{proprietes B1-B3}). 
Assume that there exists a sequence $(\sigma_n)$ of rational numbers tending to $0$,  a probability space $(\widehat{\Omega},\widehat{\mathbb P})$, a modification $(\widehat{m}_{\sigma_n},\widehat{y}_{\sigma_n})_{n \in \N}$ of $(m_{\sigma_n},y_{\sigma_n})_{n \in \N}$ on $L_1([0,1],\mathcal C[0,T]) \times L_2([0,1],\mathcal C[0,T])$ and a modification $(\widehat{m},\widehat{y})$ of $(m,y)$ on the same space such that for almost each $\omega \in \Omega$ and almost every $(u,t) \in [0,1] \times [0,T]$, the sequence $(\widehat{m}_{\sigma_n}(\omega,u,t),$ $\widehat{y}_{\sigma_n}(\omega))_{n \in \N}$ converges to $(\widehat{m}(\omega,u,t),\widehat{y}(\omega))$ in $\R \times L_2([0,1],\mathcal C[0,T])$. 
This will be proved in Lemma~\ref{lemme conv masse}.

It follows that for every $\psi \in L_\infty(0,1)$:
\begin{align*}
\E{\int_0^1 \psi(u)  \left[(\widehat{y}(u,t)-g(u))^2- (\widehat{y}(u,s)-g(u))^2-\int_s^t \frac{\mathrm dr}{\widehat{m}(u,r)} \right]   f_l(\widehat{Y}(s_1),\dots,\widehat{Y}(s_l))\mathrm du }=0.
\end{align*}
By Fubini's Theorem, we deduce that for almost every $u \in (0,1)$, 
\begin{align}
\label{martingale crochet}
\E{ \left((\widehat{y}(u,t)-g(u))^2- (\widehat{y}(u,s)-g(u))^2-\int_s^t \frac{\mathrm dr}{\widehat{m}(u,r)} \right) f_l(\widehat{Y}(s_1),\dots,\widehat{Y}(s_l))}=0.
\end{align}

We want to prove that~\eqref{martingale crochet} holds for every $u\in (0,1)$. 
Let $u\in (0,1)$. Choose $\delta>0$ such that $u \in (\delta, 1-\delta)$. 
Let $(u_p)_{ p \in \N}$ be a decreasing sequence in $(\delta, 1-\delta)$ converging to~$u$ such that for every $p \in \N$, equality~\eqref{martingale crochet} holds at point $u_p$, $(y_{\sigma_n,\eps}(u_p,t))_{t \in [0,T]}$ is a square integrable continuous $(\mathcal F^{\sigma_n,\eps}_t)_{t \in [0,T]}$-martingale for every $n \in \N$ and $\eps \in \Q_+$ and  $\limsup_{n \to \infty} \widehat{m}_{\sigma_n}(u_p,t) \leq \widehat{m}(u_p,t)$ almost surely for all $t \in [0,T]$. Such a sequence exists by Corollary~\ref{coro version continue ps} and Lemma~\ref{lemme limsup masse}. We will use these different properties later in this proof. 

Almost surely, for every $r \in (0,T]$, $\widehat{y}(\cdot,r)$ is right-continuous at point $u$ and is a step function. Therefore, $\widehat{m}(\cdot,r)=\int_0^1 \mathds 1_{\{ \widehat{y}(\cdot,r)=\widehat{y}(v,r) \}} \mathrm dv$ is also right continuous  at point $u$ for every positive time~$r$.   In order to prove~\eqref{martingale crochet} at point $u$, it is thus sufficient to show the following uniform integrability property:
there exists $\beta >1$ such that
\begin{align}
\label{cond unif int}
\sup_{p \in \N} \E{ \left((\widehat{y}(u_p,t)-g(u_p))^2- (\widehat{y}(u_p,s)-g(u_p))^2-\int_s^t \frac{\mathrm dr}{\widehat{m}(u_p,r)} \right)^\beta} <+\infty. 
\end{align}
First, by monotonicity, for all $p \in \N$,  $\E{g(u_p)^{2\beta}} \leq g(\delta)^{2\beta} + g(1-\delta)^{2\beta}$. Then, the following statement holds: there exists $\beta >1$ such that for every $t \in [0,T]$, $\sup_{p \in \N} \E{ \widehat{y}(u_p,t)^{2\beta}}<+\infty$. 
Indeed, for every $p \in \N$, by monotonicity, 
\begin{align*}
\frac{1}{\delta}\int_0^\delta \widehat{y}(v,t) \mathrm dv \leq \widehat{y}(u_p,t) \leq \frac{1}{\delta}\int_{1-\delta}^1 \widehat{y}(v,t) \mathrm dv. 
\end{align*}
Therefore, we have:
\begin{align}
\label{inegalite delta 1-delta}
\E{ \widehat{y}(u_p,t)^{2\beta}}
&\leq \E{\left( \frac{1}{\delta}\int_0^\delta \widehat{y}(v,t) \mathrm dv \right)^{2\beta}}
+\E{\left( \frac{1}{\delta}\int_{1-\delta}^1  \widehat{y}(v,t) \mathrm dv \right)^{2\beta}}\notag\\
&\leq \frac{2}{\delta}\E{\int_0^1 \widehat{y}(v,t)^{2\beta}\mathrm dv},
\end{align}
by Hölder's inequality. 
By Fatou's Lemma
\begin{align*}
\E{\int_0^1 \widehat{y}(v,t)^{2\beta}\mathrm dv}\leq \liminf_{n \to \infty} \E{\int_0^1 \widehat{y}_{\sigma_n}(v,t)^{2\beta}\mathrm dv},
\end{align*}
which is finite by Corollary~\ref{corollaire borne L_2 sig}, for a $\beta$ chosen in $ (1,\frac{3}{2}-\frac{1}{p})$. 

Let us keep the same exponent $\beta\in (1,\frac{3}{2}-\frac{1}{p})$. It remains to show that  for every $t\in [0,T]$, $\sup_{p \in \N} \E{\left(\int_0^t \frac{\mathrm dr}{\widehat{m}(u_p,r)}\right)^{\beta}}<+\infty$. 
Since $\limsup_{n \to \infty} \widehat{m}_{\sigma_n}(u_p,t) \leq \widehat{m}(u_p,t)$ and by Fatou's Lemma, 
\begin{align*}
\E{\left(\int_0^t \frac{\mathrm dr}{\widehat{m}(u_p,r)}\right)^{\beta}}
\leq \E{\left(\int_0^t \liminf_{n \to \infty} \frac{\mathrm dr}{\widehat{m}_{\sigma_n}(u_p,r)}\right)^{\beta}}
&\leq \liminf_{n \to \infty} \E{\left(\int_0^t  \frac{\mathrm dr}{\widehat{m}_{\sigma_n}(u_p,r)}\right)^{\beta}}\\
&\leq \liminf_{n \to \infty, \eps \in \Q_+} \E{\left(\int_0^t  \frac{\mathrm dr}{\widehat{M}_{\sigma_n,\eps}(u_p,r)}\right)^{\beta}}.
\end{align*}

Because $(\widehat{y}_{\sigma_n,\eps}(u_p,t))_{t \in [0,T]}$ is a square integrable $(\mathcal F^{\sigma_n,\eps}_t)_{t \in [0,T]}$-martingale and $\langle \widehat{y}_{\sigma_n,\eps}(u_p,\cdot),$ $\widehat{y}_{\sigma_n,\eps}(u_p,\cdot) \rangle_t =\int_0^t \frac{\mathrm dr}{\widehat{M}_{\sigma_n,\eps}(u_p,r)}$, we obtain by Burkholder-Davis-Gundy inequality:
\begin{align*}
\E{\left(\int_0^t  \frac{\mathrm dr}{\widehat{M}_{\sigma_n,\eps}(u_p,r)}\right)^{\beta}}
\leq C \E{(\widehat{y}_{\sigma_n,\eps}(u_p,t)-g(u_p))^{2\beta}}.
\end{align*}
We have already seen that $\E{g(u_p)^{2\beta}}$ is uniformly bounded for $p \in \N$. By the same argument as for inequality~\eqref{inegalite delta 1-delta}, 
$\E{\widehat{y}_{\sigma_n,\eps}(u_p,t)^{2\beta}} \leq \frac{2}{\delta}\E{\int_0^1 \widehat{y}_{\sigma_n,\eps}(v,t)^{2\beta}\mathrm dv}$, which is uniformly bounded for $n \in \N$ and $\eps \in \Q_+$. 
This concludes the proof of~\eqref{cond unif int}. 

Therefore, equality~\eqref{martingale crochet} holds for every $u\in (0,1) $, for every bounded and continuous $f_l$ and for every $0\leq s_1\leq \dots \leq s_l\leq s\leq t$. Thus for every $u \in (0,1)$, the process $\Big((\widehat{y}(u,t)-g(u))^2- \int_0^t \frac{\mathrm ds}{\widehat{m}(u,s)}\Big)_{t \in [0,T]}$ is an $(\mathcal F_t)_{t \in [0,T]}$-martingale. This concludes the proof of the Proposition. 
\end{proof}

In the proof of Proposition~\ref{propriete C5 partie 1}, we used the following Lemma:
\begin{lemme}
\label{lemme conv masse}
There exists a sequence $(\sigma_n)$ of rational numbers tending to $0$, a sequence of processes $(\widehat{m}_{\sigma_n},\widehat{y}_{\sigma_n})_{n \in \N}$ and a process $(\widehat{m},\widehat{y})$ defined on the same probability space such that
\begin{itemize}
\item for all $n \in \N$,  $(\widehat{m}_{\sigma_n},\widehat{y}_{\sigma_n})$ and $(m_{\sigma_n},y_{\sigma_n})$ (resp. $(\widehat{m},\widehat{y})$ and $(m,y)$) have same law on $L_1([0,1],\mathcal C[0,T]) \times L_2([0,1],\mathcal C[0,T])$. 
\item for almost each $\omega \in \Omega$ and for almost every $(u,t)$ in $[0,1]\times[0,T]$, the sequence $(\widehat{m}_{\sigma_n}(\omega,u,t),$ $\widehat{y}_{\sigma_n}(\omega))_{n \in \N}$ converges to $(\widehat{m}(\omega,u,t),\widehat{y}(\omega))$ in $\R \times L_2([0,1],\mathcal C[0,T])$. 
\end{itemize}
\end{lemme}

\begin{rem}
The Borel subset of $[0,1]\times [0,T]$ on which we have the convergence can depend on~$\omega$. 
\end{rem}

Before giving the proof of Lemma~\ref{lemme conv masse}, we give the following definition and state the following Lemma, which will be useful in the proof. 
Let us define in $L_1([0,1]\times [0,1],\mathcal C[0,T])$:
\begin{align*}
\label{def C sigma}
C_\sigma(u_1,u_2,t)
&:=\int_0^t \left( \frac{1}{m_\sigma(u_1,s)}+\frac{1}{m_\sigma(u_2,s)}-\frac{2m_\sigma(u_1,u_2,s)}{m_\sigma(u_1,s)m_\sigma(u_2,s)}  \right) \mathrm ds.
\end{align*}
\begin{lemme}
\label{lemme conv des crochets}
There exists a sequence $(\sigma_n)$ in $\Q_+$ tending to $0$ such that 
$(y_{\sigma_n},C_{\sigma_n})_{n \in \N}$ converges in distribution to $(y,C)$ in $L_2([0,1],\mathcal C[0,T])\times L_1([0,1]\times [0,1],\mathcal C[0,T])$. 
For almost every $u_1,u_2 \in [0,1]$, the limit process $C(u_1,u_2,\cdot)$ is the quadratic variation of $y(u_1,\cdot)-y(u_2,\cdot)$ relatively to the filtration generated by $Y$ and $C$. 
\end{lemme}

We start by giving the proof of Lemma~\ref{lemme conv masse} and then we give the proof of Lemma~\ref{lemme conv des crochets}. 
\begin{proof}[Proof (Lemma~\ref{lemme conv masse})]
By Skorohod's representation Theorem, we deduce from Lemma~\ref{lemme conv des crochets} that there exists a sequence $(\widehat{y}_{\sigma_n},\widehat{C}_{\sigma_n})_n$ and a random variable $(\widehat{y},\widehat{C})$ defined on the same probability space such that 
\begin{itemize}
\item for all $n \in \N$,  $(\widehat{y}_{\sigma_n},\widehat{C}_{\sigma_n})$ and $(y_{\sigma_n},C_{\sigma_n})$ (resp. $(\widehat{y},\widehat{C})$ and $(y,C)$) have same law,
\item the sequence $(\widehat{y}_{\sigma_n},\widehat{C}_{\sigma_n})_n$ converges almost surely to $(\widehat{y},\widehat{C})$ in the space $L_2([0,1],\mathcal C[0,T])\times L_1([0,1]\times [0,1],\mathcal C[0,T])$.
\end{itemize}

We apply to $(\widehat{y}_{\sigma_n})_n$ the argument in the proof of Lemma~\ref{lemme limsup masse} and we prove that, up to extracting another subsequence (independent of $\omega$), for almost every $u \in [0,1]$ and almost surely, $\limsup_{n \to \infty} \widehat{m}_{\sigma_n}(u,t) \leq \widehat{m}(u,t)$ for every $t \in [0,T]$.

For each $t \in [0,T]$, we may suppose that for each $n \in \N$, $\widehat{y}_{\sigma_n}(\cdot,t)$ is a càdlàg function, so that for every $u \in (0,1)$,
\[\widehat{m}_{\sigma_n}(u,t)=\int_0^1 \varphi_{\sigma_n}^2 (\widehat{y}_{\sigma_n}(u,t)-\widehat{y}_{\sigma_n}(v,t))\mathrm dv=\lim_{p \to \infty} p\int_u^{(u+\frac{1}{p})\wedge 1} \!\!\int_0^1 \varphi_{\sigma_n}^2 (\widehat{y}_{\sigma_n}(u',t)-\widehat{y}_{\sigma_n}(v,t))\mathrm dv \mathrm du'\] is a measurable function with respect to $\widehat{y}_{\sigma_n}(\cdot,t)$. We deduce that $(\widehat{m}_{\sigma_n}(u,t),\widehat{y}_{\sigma_n})$ has the same law as $(m_{\sigma_n}(u,t),y_{\sigma_n})$ for every $u \in (0,1)$.

From now on, we forget the hats in our notation. We may suppose that~$y$ is the version in $\mathcal D((0,1),\mathcal C[0,T])$ given by Proposition~\ref{propriete B4}.  Let $\Omega'$  be such that $\P{\Omega'}=1$ and for all $\omega \in \Omega'$, we have the following convergences in $\R$:
\begin{align}
\label{conv y sigma}
\int_0^1 \sup_{t \leq T} |y_{\sigma_n}(u,t)-y(u,t)|^2(\omega)\mathrm du &\underset{n \to \infty}{\longrightarrow} 0,\\
\label{conv C sigma}
\int_0^1 \!\!\int_0^1 \sup_{t \leq T} |C_{\sigma_n}(u_1,u_2,t)-C(u_1,u_2,t)|(\omega)\mathrm du_1\mathrm du_2 &\underset{n \to \infty}{\longrightarrow} 0.
\end{align}

Fix $\omega \in \Omega'$. 
Thanks to~\eqref{conv y sigma}, we already have the convergence of $(y_{\sigma_n}(\omega))_{n}$ to $y(\omega)$ in $L_2([0,1],\mathcal C[0,T])$. 
It remains to show that for almost every $(u,t)\in [0,1]\times [0,T]$, $(m_{\sigma_n}(\omega,u,t))_{n}$ converges to $m(\omega,u,t)=\int_0^1 \mathds 1_{\{ y(u,t)=y(v,t) \}} (\omega)\mathrm dv$. 
We already know that for every $\omega \in \Omega'$, every $t \in [0,T]$ and almost every $u \in (0,1)$, $\limsup_{n \to \infty} m_{\sigma_n}(\omega,u,t) \leq m(\omega, u,t)$.

\paragraph{Proof of inequality: }$\liminf_{n \to \infty} m_{\sigma_n}(\omega,u,t) \geq m(\omega,u,t)$.

By the coalescence property given by Proposition~\ref{propriete B4}, for every $u_1,u_2$ and for all $t> \tau_{u_1,u_2}$, $y(u_1,t)=y(u_2,t)$. Therefore, since $C(u_1,u_2,\cdot)$ is the  quadratic variation of $y(u_1,\cdot)-y(u_2,\cdot)$, $t \mapsto C(u_1,u_2,t)$ remains constant on $(\tau_{u_1,u_2},T)$. Thus we obtain:
\begin{multline*}
\left|\int_0^1 \!\!\int_0^1\!\! \int_{\tau_{u_1,u_2}}^T \left( \frac{1}{m_{\sigma_n}(u_1,t)}+\frac{1}{m_{\sigma_n}(u_2,t)}-\frac{2m_{\sigma_n}(u_1,u_2,t)}{m_{\sigma_n}(u_1,t)m_{\sigma_n}(u_2,t)}  \right)  \mathrm dt \mathrm du_1\mathrm du_2 \right| \\
\begin{aligned}
&=\left|\int_0^1\!\!\int_0^1(C_{\sigma_n}(u_1,u_2,T)-C_{\sigma_n}(u_1,u_2,\tau_{u_1,u_2})) \mathrm du_1\mathrm du_2 \right|\\
&=\left |\int_0^1\!\!\int_0^1(C_{\sigma_n}(u_1,u_2,T)-C(u_1,u_2,T)+C(u_1,u_2,\tau_{u_1,u_2})-C_{\sigma_n}(u_1,u_2,\tau_{u_1,u_2})) \mathrm du_1\mathrm du_2\right| \\
&\leq 2 \int_0^1 \!\!\int_0^1 \sup_{t \leq T} |C_{\sigma_n}(u_1,u_2,t)-C(u_1,u_2,t)|\mathrm du_1\mathrm du_2.
\end{aligned}
\end{multline*}
By~\eqref{conv C sigma}, the latter term tends to $0$. 
We also recall that 
\begin{multline*}
\frac{1}{m_{\sigma_n}(u_1,t)}+\frac{1}{m_{\sigma_n}(u_2,t)}-\frac{2m_{\sigma_n}(u_1,u_2,t)}{m_{\sigma_n}(u_1,t)m_{\sigma_n}(u_2,t)}\\
=\frac{\int_0^1 \left| \varphi_{\sigma_n}(y_{\sigma_n}(u_1,t_0)-y_{\sigma_n}(v,t_0))-\varphi_{\sigma_n}(y_{\sigma_n}(u_2,t_0)-y_{\sigma_n}(v,t_0)) \right|^2 \mathrm dv}{m_{\sigma_n}(u_1,t_0)m_{\sigma_n}(u_2,t_0)}
\end{multline*}
is non-negative. 
We define $f_{\sigma_n}(t,u_1,u_2):=\!\left( \frac{1}{m_{\sigma_n}(u_1,t)}+\frac{1}{m_{\sigma_n}(u_2,t)}-\frac{2m_{\sigma_n}(u_1,u_2,t)}{m_{\sigma_n}(u_1,t)m_{\sigma_n}(u_2,t)}  \right)\! \mathds 1_{ \{t\geq \tau_{u_1,u_2}\}}$. For every $\omega \in \Omega'$, 
$\int_0^T \int_0^1 \int_0^1 f_{\sigma_n}(t,u_1,u_2)(\omega) \mathrm du_1 \mathrm du_2 \mathrm dt\underset{n \to \infty}{\longrightarrow} 0 $. 
Therefore, for every $\eps >0$, using Markov's inequality as in~\eqref{Markov}, and since $f_{\sigma_n} \geq 0$: 
\begin{multline*}
\mathbb P \otimes \frac{1}{T} \Leb|_{[0,T]} \otimes \Leb|_{[0,1]}\otimes \Leb|_{[0,1]}\left\{ (\omega,t,u_1,u_2): f_{\sigma_n}(t,u_1,u_2)(\omega) \geq \eps \right\}
\\ \leq \E{1 \wedge \frac{1}{\eps T} \int_0^T\!\! \int_0^1 \!\!\int_0^1  f_{\sigma_n}(t,u_1,u_2) \mathrm du_1 \mathrm du_2 \mathrm dt},
\end{multline*}
which tends to $0$ when $n \to \infty$, whence we obtain a convergence in probability with respect to the probability space $\Omega \times [0,T] \times [0,1] \times [0,1]$. 
Up to extracting another subsequence (independent of the choice of $\omega$), we deduce the existence of  an almost sure event on which $(f_{\sigma_n})$ converges to~$0$. 

Let $\Omega''$, $\P{\Omega''}=1$, be such that for every $\omega \in \Omega''$, we have $f_{\sigma_n}(t,u_1,u_2)(\omega) \to 0$ for almost every $(t,u_1,u_2) \in [0,T] \times [0,1] \times [0,1]$. Fix $\omega \in \Omega''$. Let us consider a Borel set $\B=\B(\omega)$ in $[0,T]$, $\Leb(\B)=T$, such that for every $t\in\B$, $f_{\sigma_n}(t,u_1,u_2) \to 0$ for almost every $(u_1,u_2) \in [0,1] \times [0,1]$. 

Let $t_0 \in \B$. 
Let us consider a Borel set $\A$ (depending on $\omega$ and $t_0$) of measure $1$ such that for all $u_1,u_2 \in \A$,
\begin{align}
\label{f sigma}
f_{\sigma_n}(t_0,u_1,u_2) \underset{n \to \infty}{\longrightarrow} 0.
\end{align}
Let $u \in \A$. 
We want to prove that $\liminf_{n \to \infty} m_{\sigma_n}(u,t_0)\geq m(u,t_0)$. 
Define $u_{\sup} = \sup \{v \in [0,1]: y(v,t_0)=y(u,t_0)\}$ and $u_{\inf}$ the infimum of that set. Since $v \mapsto y(v,t_0)$ is non-decreasing, $m(u,t_0)=u_{\sup}-u_{\inf}$.
If $m(u,t_0)=0$, then we clearly have: 
$\liminf_{n \to \infty} m_{\sigma_n}(u,t_0) \geq m(u,t_0)$.
Suppose now that $m(u,t_0)>0$. Choose $\delta >0$ such that $\delta <\frac{u_{\sup}-u_{\inf}}{6}$. 
Let $\umax \in \A  \cap (u_{\sup}-\delta,u_{\sup})$, $\umin \in \A\cap (u_{\inf},u_{\inf}+\delta)$ and $\umed \in \A \cap \left( \frac{\umin+\umax}{2}-\delta,\frac{\umin+\umax}{2}+\delta \right)$. 
We have: $\umax-\umin \geq u_{\sup}-u_{\inf}-2\delta=m(u,t_0)-2\delta$ and by definition of $u_{\sup}$ and $u_{\inf}$ and since $\umax$, $\umin$ and $\umed$ belongs to $(u_{\inf},u_{\sup})$, we have $t_0 \geq \tau_{u_1,u_2}$ for $(u_1,u_2)=(u,\umax)$, $(u,\umin)$, $(\umax,\umin)$ and $(u,\umed)$.

We deduce from~\eqref{f sigma} and the fact that $u, \umax, \umin, \umed$ belongs to $\A$ that there exists $N$ such that for each $n \geq N$, $f_{\sigma_n}(t_0,u_1,u_2)\leq \delta$ for $(u_1,u_2)=(u,\umax)$, $(u,\umin)$, $(\umax,\umin)$ and $(u,\umed)$. It implies that for each $n \geq N$, 
\begin{equation}
\frac{\int_0^1 \left| \varphi_{\sigma_n}(y_{\sigma_n}(u_1,t_0)-y_{\sigma_n}(v,t_0))-\varphi_{\sigma_n}(y_{\sigma_n}(u_2,t_0)-y_{\sigma_n}(v,t_0)) \right|^2 \mathrm dv}{m_{\sigma_n}(u_1,t_0)m_{\sigma_n}(u_2,t_0)}= f_{\sigma_n}(t_0,u_1,u_2) \leq \delta.
\label{eq crochet inf delta}
\end{equation}
Since the mass $m_{\sigma_n}$ is bounded by $1$, we deduce in particular that for all $n \geq N$, 
\begin{equation}
\int_0^1 \left| \varphi_{\sigma_n}(y_{\sigma_n}(u_1,t_0)-y_{\sigma_n}(v,t_0))-\varphi_{\sigma_n}(y_{\sigma_n}(u_2,t_0)-y_{\sigma_n}(v,t_0)) \right|^2 \mathrm dv\leq \delta.
\label{eq intég inf delta}
\end{equation}
Inequalities~\eqref{eq crochet inf delta} and~\eqref{eq intég inf delta} are satisfied for $(u_1,u_2)=(u,\umax)$, $(u,\umin)$, $(\umax,\umin)$ and $(u,\umed)$.

Let $n \geq N$ and $d:=y_{\sigma_n}(\umax,t_0)-y_{\sigma_n}(\umin,t_0) \geq 0$. We distinguish three cases:
\begin{description}
\item[\textbullet \,  $d \geq {\sigma_n}$:] Recall that $\varphi_{\sigma_n}$ is equal to $0$ on $[\frac{{\sigma_n}}{2},+\infty)$. Thus for all $v\in [0,1]$, $\varphi_{\sigma_n}(y_{\sigma_n}(\umax,t_0)-y_{\sigma_n}(v,t_0))$ and $\varphi_{\sigma_n}(y_{\sigma_n}(\umin,t_0)-y_{\sigma_n}(v,t_0))$ can not be simultaneously different from $0$ because $d \geq {\sigma_n}$. Therefore, selecting $(u_1,u_2)=(\umax,\umin)$, inequality~(\ref{eq crochet inf delta}) implies:
 \begin{equation*}
\frac{\int_0^1  \varphi_{\sigma_n}^2(y_{\sigma_n}(\umax,t_0)-y_{\sigma_n}(v,t_0)) \mathrm dv+\int_0^1  \varphi_{\sigma_n}^2(y_{\sigma_n}(\umin,t_0)-y_{\sigma_n}(v,t_0)) \mathrm dv}{m_{\sigma_n}(\umax,t_0)m_{\sigma_n}(\umin,t_0)}\leq \delta,
\end{equation*}
that is: 
 \begin{equation*}
\frac{1}{m_{\sigma_n}(\umin,t_0)}+\frac{1}{m_{\sigma_n}(\umax,t_0)}\leq \delta.
\end{equation*}
Thus, we obtain $\delta \geq 2$, which is excluded by definition of $\delta$. 
\item[\textbullet \,  $d \leq {\sigma_n}-\eta$:]
Recall that $\eta$ is chosen so that $\eta < \frac{{\sigma_n}}{3}$. 
Define the two following sets
\begin{align*}
\vmax &=\{ v \in [\umin,\umax]: y_{\sigma_n}(\umax,t_0)-y_{\sigma_n}(v,t_0) \leq \textstyle\frac{{\sigma_n}-\eta}{2} \},\\
\vmin &=\{ v \in [\umin,\umax]: y_{\sigma_n}(\umax,t_0)-y_{\sigma_n}(v,t_0) > \textstyle \frac{{\sigma_n}-\eta}{2} \}.
\end{align*}
Clearly, we have: $\Leb(\vmax)+\Leb(\vmin)=\umax-\umin\geq m(u,t_0)-2\delta$. 
Recall that $\varphi_{\sigma_n}$ is equal to $1$ on $[0,\frac{{\sigma_n}-\eta}{2}]$. Thus, for each $v \in \vmax$, $\varphi_{\sigma_n}(y_{\sigma_n}(\umax,t_0)-y_{\sigma_n}(v,t_0))=1$, and for each $v \in \vmin$, using $d \leq {\sigma_n}-\eta$, $\varphi_{\sigma_n}(y_{\sigma_n}(\umin,t_0)-y_{\sigma_n}(v,t_0))=1$. 
We have
\begin{equation}
\textstyle m_{\sigma_n}(u,t_0) \geq \int_{\vmax} \varphi_{\sigma_n}^2 (y_{\sigma_n}(u,t_0)-y_{\sigma_n}(v,t_0))\mathrm dv
+\int_{\vmin} \varphi_{\sigma_n}^2 (y_{\sigma_n}(u,t_0)-y_{\sigma_n}(v,t_0))\mathrm dv.
\label{eq vmax vmin}
\end{equation}

We can deduce from inequality~(\ref{eq intég inf delta}) applied to $(u_1,u_2)=(u,\umax)$ that:
\[
\textstyle \int_{\vmax} \left| \varphi_{\sigma_n}(y_{\sigma_n}(u,t_0)-y_{\sigma_n}(v,t_0))-\varphi_{\sigma_n}(y_{\sigma_n}(\umax,t_0)-y_{\sigma_n}(v,t_0)) \right|^2 \mathrm dv\leq \delta.
\]
By Minkowski's inequality $\left|  \|f_1\|_{L_2}-\|f_2\|_{L_2} \right| \leq \|f_1-f_2\|_{L_2}$, we obtain: 
\[
\left|\left(\int_{\vmax}  \varphi_{\sigma_n}^2(y_{\sigma_n}(u,t_0)-y_{\sigma_n}(v,t_0)) \mathrm dv\right)^{1/2}-\Leb(\vmax)^{1/2} \right|\leq \sqrt{\delta},
\]
whence
\[
 \left|\int_{\vmax}  \!\!\!\!\varphi_{\sigma_n}^2(y_{\sigma_n}(u,t_0)-y_{\sigma_n}(v,t_0)) \mathrm dv -\Leb(\vmax)\right|
\leq (m^{1/2}_{\sigma_n}(u,t_0)+\Leb(\vmax)^{1/2})\sqrt{\delta}\leq 2 \sqrt{\delta}.
\]
Similarly, applying inequality~(\ref{eq intég inf delta}) to $(u,\umin)$, we obtain:
\[
\textstyle \left|\int_{\vmin}  \varphi_{\sigma_n}^2(y_{\sigma_n}(u,t_0)-y_{\sigma_n}(v,t_0)) \mathrm dv -\Leb(\vmin)\right|\leq 2 \sqrt{\delta}.
\]
Thus, by inequality~(\ref{eq vmax vmin}), we conclude:
\begin{align*}
m_{\sigma_n}(u,t_0) &\geq \Leb(\vmax)+\Leb(\vmin) -4\sqrt{\delta} \\
& \geq m(u,t_0)-2\delta-4\sqrt{\delta}.
\end{align*}

\item[\textbullet \,  $d \in ({\sigma_n}-\eta,{\sigma_n})$:]
We now define three distinct sets
\begin{align*}
 \vmax &=\{ v \in [\umin,\umax]: y_{\sigma_n}(\umax,t_0)-y_{\sigma_n}(v,t_0) < \textstyle\frac{{\sigma_n}-\eta}{2} \},\\
\vmed &=\{ v \in [\umin,\umax]: y_{\sigma_n}(\umax,t_0)-y_{\sigma_n}(v,t_0) \in \textstyle[\frac{{\sigma_n}-\eta}{2} ,\frac{{\sigma_n}+\eta}{2} ] \},\\
\vmin &=\{ v \in [\umin,\umax]: y_{\sigma_n}(\umax,t_0)-y_{\sigma_n}(v,t_0) > \textstyle\frac{{\sigma_n}+\eta}{2} \}.
\end{align*}
By definition of those sets, and since $d \in ({\sigma_n}-\eta,{\sigma_n})$, we have 
\begin{align*}
\forall v \in \vmax,\,\,\,&\varphi_{\sigma_n}(y_{\sigma_n}(\umax,t_0)-y_{\sigma_n}(v,t_0))=1, \\
\forall v \in \vmin, \,\,\,&\varphi_{\sigma_n}(y_{\sigma_n}(\umin,t_0)-y_{\sigma_n}(v,t_0))=1.
\end{align*}

Moreover, we have $y_{\sigma_n}(\umax,t_0)-y_{\sigma_n}(\umed,t_0) \in [\frac{{\sigma_n}-\eta}{2} ,\frac{{\sigma_n}+\eta}{2} ]$. 

Indeed, if  $y_{\sigma_n}(\umax,t_0)-y_{\sigma_n}(\umed,t_0) $ was greater than $\frac{{\sigma_n}+\eta}{2}$, 
we would have, for all $v \in [\umin,\umed]$, $\varphi_{\sigma_n}(y_{\sigma_n}(\umax,t_0)-y_{\sigma_n}(v,t_0))=0$ and $\varphi_{\sigma_n}(y_{\sigma_n}(\umin,t_0)-y_{\sigma_n}(v,t_0))=1$. By inequality~(\ref{eq intég inf delta}) applied to $(u_1,u_2)=(\umax,\umin)$, we would deduce that: 
\begin{align*}
\delta &\geq \int_0^1 \left| \varphi_{\sigma_n}(y_{\sigma_n}(\umax,t_0)-y_{\sigma_n}(v,t_0))-\varphi_{\sigma_n}(y_{\sigma_n}(\umin,t_0)-y_{\sigma_n}(v,t_0)) \right|^2 \mathrm dv \\
&\geq \int_{\umin}^{\umed} \mathrm dv =\umed-\umin \geq \frac{\umax-\umin}{2}-\delta. 
\end{align*}

However, since $\delta < \frac{u_{\sup}-u_{\inf}}{6}$ and $\umax-\umin \geq u_{\sup}-u_{\inf} -2\delta$, we have $\umax-\umin > 4 \delta$, which is in contradiction with the above inequality. 
Similarly, $y_{\sigma_n}(\umax,t_0)-y_{\sigma_n}(\umed,t_0) $ can not be smaller than $\frac{{\sigma_n}-\eta}{2}$, otherwise  $y_{\sigma_n}(\umed,t_0)-y_{\sigma_n}(\umin,t_0)$ would be greater than $\frac{{\sigma_n}+\eta}{2}$ and we would obtain the same contradiction. 
Therefore, $y_{\sigma_n}(\umax,t_0)-y_{\sigma_n}(\umed,t_0) \in [\frac{{\sigma_n}-\eta}{2} ,\frac{{\sigma_n}+\eta}{2} ]$, which implies that $\umed \in \vmed$ and in particular that
\begin{align*}
\forall v \in \vmed,\,\,\,&\varphi_{\sigma_n}(y_{\sigma_n}(\umed,t_0)-y_{\sigma_n}(v,t_0))=1.
\end{align*}
As in the previous case, we deduce that
\begin{align*}
m_{\sigma_n}(u,t_0) &\geq \Leb(\vmax)+\Leb(\vmed)+\Leb(\vmin) -6\sqrt{\delta} \\
&= \umax-\umin -6\sqrt{\delta} \\
& \geq m(u,t_0)-2\delta-6\sqrt{\delta}.
\end{align*}
\end{description}

Actually, putting all the cases together, we have proved that for each $n \geq N$, $m_{\sigma_n}(u,t_0)  \geq m(u,t_0)-2\delta-6\sqrt{\delta}$. 
Hence, for all $\delta< \frac{u_{\sup}-u_{\inf}}{6}$, we have:
\[
\liminf_{n \to \infty} m_{\sigma_n}(u,t_0)\geq m(u,t_0) -2\delta -6\sqrt{\delta}.
\]
By letting $\delta$ converge to $0$, we have for every $t_0 \in \B$,  $\liminf_{n \to \infty} m_{\sigma_n}(u,t_0)\geq m(u,t_0)$ for every $u \in \A$.  
Therefore, there exists a subsequence $(\sigma_n)$ such that for almost every $\omega$, for almost every $t \in [0,T]$ and almost every $u\in [0,1]$, 
$m_{\sigma_n}(\omega,u,t) \to_{ n \to \infty} m(\omega,u,t)$.
\end{proof}

It remains to give the proof of Lemma~\ref{lemme conv des crochets}.
\begin{proof}[Proof (Lemma~\ref{lemme conv des crochets})]

The first step will be to prove that the sequence $(y_\sigma,C_\sigma)_{\sigma\in \Q_+}$ is tight in $L_2([0,1],\mathcal C[0,T])\times L_1([0,1]\times [0,1],\mathcal C[0,T])$. We have already proved that $(y_\sigma)_{\sigma\in \Q_+}$ is tight in $L_2([0,1],\mathcal C[0,T])$. We will use a tightness criterion to prove that the sequence $(C_\sigma)_{\sigma\in \Q_+}$ is tight in $L_1([0,1]\times [0,1],\mathcal C[0,T])$.
The space changed in comparison with $L_2([0,1], \mathcal C[0,T])$, but the criterion remains very semilar to the one of Proposition~\ref{prop crit tight}. 
\medskip

We have, similarly to Proposition~\ref{prop crit tight}, three criteria to prove. We want to show the following criterion:

\textbf{ First criterion: }
Let $\delta >0$. There is $M>0$ such that for all $\sigma$ in $\Q_+$, $\P{\|C_\sigma\| \geq M}\leq \delta$, where $\|C_\sigma\|:= \int_0^1\int_0^1 \sup_{t\leq T} |C_\sigma(u_1,u_2,t)| \mathrm du_1\mathrm du_2$. 

That statement follows from Markov's inequality and the existence of a constant $C$ independent of $\sigma$ such that:
\begin{multline*}
\E{\int_0^1\!\!\int_0^1 \!\!\sup_{t\leq T} |C_\sigma(u_1,u_2,t)| \mathrm du_1\mathrm du_2}
\leq 2 \E{\int_0^1 \!\! \int_0^T \frac{\mathrm dt  \mathrm du_1}{m_\sigma(u_1,t)}  }
+2 \E{\int_0^1 \!\! \int_0^T \frac{\mathrm dt  \mathrm du_2}{m_\sigma(u_2,t)} }
\leq C. 
\end{multline*}
The existence of $C$ is a consequence of Lemma~\ref{lemme inverse masse approchee sigma}. 
\medskip

Then, we prove the following criterion:

\textbf{Second criterion: }
Let $\delta >0$. For each $k \geq 1$, there exists $\eta_k>0 $ such that for all $\sigma$  in~$\Q_+$, 
\begin{align*}
\P{\int_0^1 \!\! \int_0^1 \sup_{|t_2-t_1|<\eta_k} |C_\sigma(u_1,u_2,t_2)-C_\sigma(u_1,u_2,t_1)| \mathrm du_1\mathrm du_2 \geq \frac{1}{k}}\leq \frac{\delta}{2^k}. 
\end{align*}

The proof is very close to Proposition~\ref{propo critere K2}. We start by defining for every $u_1$, $u_2 \in (0,1)$: $K_1(u_1,u_2):=\E{\|C_\sigma(u_1,u_2,\cdot)\|_{\mathcal C[0,T]}}$ and $K_2(u_1):=\E{\int_0^T \frac{1}{m^\beta_\sigma(u_1,s)}\mathrm ds}$. Fix $\delta>0$. There exists $C>0$ such that $\int_0^1\int_0^1 \mathds 1_{\{K_1(u_1,u_2) \geq C\}} \mathrm du_1\mathrm du_2 \leq \delta$ and $\int_0^1 \mathds 1_{\{K_2(u) \geq C\}} \mathrm du \leq \delta$. 
Define the following set $K:=\{ (u_1,u_2): K_1(u_1,u_2) \leq C, K_2(u_1) \leq C, K_2(u_2) \leq C\}$. 

By Aldous' tightness criterion, the collection $(C_\sigma(u_1,u_2,\cdot))_{\sigma \in \Q_+, (u_1,u_2)\in K}$ is tight in $\mathcal C[0,T]$. This fact relies on the following inequality, where $ \eta >0$ and $\tau$ is a stopping time for $C_\sigma(u_1,u_2,\cdot)$:
\begin{multline*}
\E{|C_\sigma(u_1,u_2,\tau+\eta)-C_\sigma(u_1,u_2,\tau)|}\\
\begin{aligned}
&=\E{\left|\int_\tau^{\tau+\eta} \left( \frac{1}{m_\sigma(u_1,s)}+\frac{1}{m_\sigma(u_2,s)}-\frac{2m_\sigma(u_1,u_2,s)}{m_\sigma(u_1,s)m_\sigma(u_2,s)} \right) \mathrm ds\right|}\\
&\leq 2\E{\int_\tau^{\tau+\eta} \left( \frac{1}{m_\sigma(u_1,s)}+\frac{1}{m_\sigma(u_2,s)} \right) \mathrm ds},
\end{aligned}
\end{multline*}
and the rest of the proof is an adaptation of the proof of Proposition~\ref{propo critere K2}. 

\medskip

Finally we show the third criterion: 

\textbf{Third criterion: }
Let $\delta>0$. For each $k\geq 1$, there is $H>0$ such that for all $\sigma$ in $\Q_+$, 
\begin{multline}
\mathbb P \Bigg[\forall h=(h_1,h_2), 0<h_1<H, 0<h_2<H,
\\ \int_0^{1-h_1}\!\!\int_0^{1-h_2}\!\! \sup_{t \leq T} |C_\sigma(u_1+h_1,u_2+h_2,t)-C_\sigma(u_1,u_2,t)|\mathrm du_1\mathrm du_2 \leq \frac{1}{k}\Bigg ]\geq 1-\frac{\delta}{2^k}.
\label{second criterion}
\end{multline}
Let $h_1>0$ and begin by estimating
\begin{align*}
E_\sigma:=\E{\int_0^{1-h_1} \!\! \int_0 ^1 \sup_{t \leq T} |C_\sigma(u_1+h_1,u_2,t)-C_\sigma(u_1,u_2,t)|\mathrm du_1\mathrm du_2 }.
\end{align*}
We compute (for the sake of simplicity, we will write from now on $y_\sigma(u)$ instead of $y_\sigma(u,\cdot)$ if there is no possibility of confusion):
\begin{align*}
C_\sigma(u_1+h_1,u_2,t)-C_\sigma(u_1,u_2,t)
&=\langle y_\sigma(u_1+h_1)-y_\sigma(u_2),y_\sigma(u_1+h_1)-y_\sigma(u_2)\rangle_t\\
&\quad -\langle y_\sigma(u_1)-y_\sigma(u_2),y_\sigma(u_1)-y_\sigma(u_2)\rangle_t\\
&=\langle y_\sigma(u_1+h_1)-y_\sigma(u_1),y_\sigma(u_1+h_1)-y_\sigma(u_2)\rangle_t\\
&\quad +\langle y_\sigma(u_1)-y_\sigma(u_2),y_\sigma(u_1+h_1)-y_\sigma(u_1)\rangle_t.
\end{align*}
Therefore, 
\begin{multline}
\label{ineg 2 sup}
\sup_{t\leq T} |C_\sigma(u_1+h_1,u_2,t)-C_\sigma(u_1,u_2,t)|\\
\begin{aligned}
&\leq \sup_{t\leq T} |\langle y_\sigma(u_1+h_1)-y_\sigma(u_1),y_\sigma(u_1+h_1)-y_\sigma(u_2)\rangle_t|\\
&\quad +\sup_{t\leq T} |\langle y_\sigma(u_1)-y_\sigma(u_2),y_\sigma(u_1+h_1)-y_\sigma(u_1)\rangle_t|.
\end{aligned}
\end{multline}
Then, we use Kunita-Watanabe's inequality on the first term of the right hand side:
\begin{multline*}
|\langle y_\sigma(u_1+h_1)-y_\sigma(u_1),y_\sigma(u_1+h_1)-y_\sigma(u_2)\rangle_t|
\\ \leq |\langle y_\sigma(u_1+h_1)-y_\sigma(u_1),y_\sigma(u_1+h_1)-y_\sigma(u_1)\rangle_t|^{\frac{1}{2}}\\
\shoveright{
 |\langle y_\sigma(u_1+h_1)-y_\sigma(u_2),y_\sigma(u_1+h_1)-y_\sigma(u_2)\rangle_t|^{\frac{1}{2}}}\\
 \leq |\langle y_\sigma(u_1+h_1)-y_\sigma(u_1),y_\sigma(u_1+h_1)-y_\sigma(u_1)\rangle_T|^{\frac{1}{2}}\\
 |\langle y_\sigma(u_1+h_1)-y_\sigma(u_2),y_\sigma(u_1+h_1)-y_\sigma(u_2)\rangle_T|^{\frac{1}{2}}.
\end{multline*}
By doing the same computation on the second term of the right hand side of~\eqref{ineg 2 sup}, by Cauchy-Schwarz inequality and by the substitution of $u_1+h_1$ by $u_1$, we obtain:
\begin{align*}
E_\sigma &\leq 2 \E{\int_0^{1-h_1}\!\!\int_0^1 \langle y_\sigma(u_1+h_1)-y_\sigma(u_1),y_\sigma(u_1+h_1)-y_\sigma(u_1)\rangle_T  \mathrm du_1\mathrm du_2}^{1/2}\\
&\quad \times \E{\int_0^1 \!\! \int_0^1 \langle y_\sigma(u_1)-y_\sigma(u_2),y_\sigma(u_1)-y_\sigma(u_2)\rangle_T  \mathrm du_1\mathrm du_2}^{1/2}\\
&\leq 2 \E{\int_0^{1-h_1} \langle y_\sigma(u_1+h_1)-y_\sigma(u_1),y_\sigma(u_1+h_1)-y_\sigma(u_1)\rangle_T  \mathrm du_1}^{1/2}C^{1/2},
\end{align*}
where $C$ is the same constant as the one in the first criterion. 
By Fubini's Theorem:
\begin{align*}
E_\sigma 
&\leq 2 C^{1/2}  \E{\int_0^{1-h_1}( y_\sigma(u_1+h_1,T)-y_\sigma(u_1,T) +g(u_1)-g(u_1+h_1))^2  \mathrm du_1}^{1/2}\\
&\leq 2 C^{1/2} \E{\int_0^{1-h_1}( y_\sigma(u_1+h_1,T)-y_\sigma(u_1,T))^2  \mathrm du_1}^{1/2}\\
&\quad + 2 C^{1/2}\E{\int_0^{1-h_1}(g(u_1+h_1)-g(u_1))^2  \mathrm du_1}^{1/2}.
\end{align*}
We recall inequalities~\eqref{ineg g carré} and~\eqref{h polynomial}.  Therefore, there are $\alpha>0$ and $C>0$  such that for each $\sigma \in \Q_+$ and each $h_1>0$,
\begin{align*}
E_\sigma \leq C h_1^\alpha. 
\end{align*}
We deduce that for each $n \in \mathds N$, by Markov's inequality, 
\begin{align*}
p_n:=\P{\int_0^{1-\frac{1}{2^n}} \!\! \int_0 ^1 \sup_{t \leq T} |C_\sigma(u_1+\textstyle \frac{1}{2^n},u_2,t)-C_\sigma(u_1,u_2,t)|\mathrm du_1\mathrm du_2\geq \frac{1}{2^{\frac{n\alpha}{2}}}}
&\leq 2^{\frac{n\alpha}{2}}C \left(\frac{1}{2^n}\right)^\alpha\\
&=\frac{C}{2^{\frac{n\alpha}{2}}}. 
\end{align*}
Since $\alpha >0$, $\sum_{n \geq 0} p_n$ converges. By Borel-Cantelli's Lemma, for each $k \geq 1$, there is $n_0 \geq 0$ such that, with probability greater than $1-\frac{\delta}{2^k}$, for all $n \geq n_0$, 
\begin{align*}
\int_0^{1-\frac{1}{2^n}} \!\! \int_0 ^1 \sup_{t \leq T} |C_\sigma(u_1+\textstyle \frac{1}{2^n},u_2,t)-C_\sigma(u_1,u_2,t)|\mathrm du_1\mathrm du_2\leq \displaystyle\frac{1}{2^{\frac{n\alpha}{2}}}.
\end{align*}
Furthermore, up to choosing a greater $n_0$, we can suppose that for all $n \geq n_0$, we also have:
\begin{align*}
 \int_0 ^1  \int_0^{1-\frac{1}{2^n}} \!\!\sup_{t \leq T} |C_\sigma(u_1,u_2+\textstyle \frac{1}{2^n},t)-C_\sigma(u_1,u_2,t)|\mathrm du_1\mathrm du_2
&\leq \displaystyle\frac{1}{2^{\frac{n\alpha}{2}}}. 
\end{align*}
We will now extend these estimations to more general perturbations. Let $h=(h_1,h_2)$ be such that $0<h_1 <\frac{1}{2^{n_0}}$, $0<h_2 <\frac{1}{2^{n_0}}$. We decompose:
\begin{multline}
\int_0^{1-h_1}\!\!\int_0^{1-h_2}\!\! \sup_{t \leq T} |C_\sigma(u_1+h_1,u_2+h_2,t)-C_\sigma(u_1,u_2,t)|\mathrm du_1\mathrm du_2 \\
\leq 
\int_0^{1-h_1}\!\!\int_0^1 \sup_{t \leq T} |C_\sigma(u_1+h_1,u_2,t)-C_\sigma(u_1,u_2,t)|\mathrm du_1\mathrm du_2 
\\+\int_0^1  \!\! \int_0^{1-h_2}\!\! \sup_{t \leq T} |C_\sigma(u_1,u_2+h_2,t)-C_\sigma(u_1,u_2,t)|\mathrm du_1\mathrm du_2.
\label{decomposition crochet}
\end{multline}
Suppose $h_1\geq 0$. Since $h_1< \frac{1}{2^{n_0}}$, there exists a sequence $(\eps_n)_{n > n_0}$ with values in $\{0,1\}$ such that $h_1=\sum_{n \geq n_0+1} \frac{\eps_n}{2^n}$. 
Moreover, we have for every $q\geq 1$:
\begin{multline}
\label{ineg avant conv}
\int_0^{1-h_1} \!\!\int_0^1 \sup_{t \leq T} |C_\sigma(u_1+h_1,u_2,t)-C_\sigma(u_1+\textstyle \sum_{n \geq n_0+q} \frac{\eps_n}{2^n},u_2,t)|\mathrm du_1\mathrm du_2 \\
\begin{aligned}
&\leq \sum_{k=1}^{q-1} \int_0^{1-h_1} \!\!\int_0^1 \sup_{t \leq T} |C_\sigma(u_1+\textstyle\sum_{n \geq n_0+k} \frac{\eps_n}{2^n},u_2,t)-C_\sigma(u_1+\textstyle\sum_{n \geq n_0+k+1} \frac{\eps_n}{2^n},u_2,t)|\mathrm du_1\mathrm du_2 \\
&\leq \sum_{k=1}^{q-1} \int_0^{1-\frac{1}{2^{n_0+k}}} \!\!\int_0^1 \sup_{t \leq T} |C_\sigma(u_1+\textstyle\frac{1}{2^{n_0+k}},u_2,t)-C_\sigma(u_1,u_2,t)|\mathrm du_1\mathrm du_2 \leq \displaystyle \sum_{k=1}^{q-1} \frac{1}{2^{(n_0+k)\frac{\alpha}{2}}}.
\end{aligned}
\end{multline}
We want to let $q$ tend to $+\infty$ in~\eqref{ineg avant conv}. To do that, we prove that:
\begin{align}
\label{conv en p}
\int_0^{1-h_1} \!\!\int_0^1 \sup_{t \leq T} |C_\sigma(u_1+\textstyle \sum_{n \geq n_0+q} \frac{\eps_n}{2^n},u_2,t)-C_\sigma(u_1,u_2,t)|\mathrm du_1\mathrm du_2 \underset{q \to +\infty}{\longrightarrow}0. 
\end{align}
By definition of $C_\sigma$, 
\begin{multline}
\label{masse continue à droite}
\int_0^{1-h_1} \!\!\int_0^1 \sup_{t \leq T} |C_\sigma(u_1+\textstyle \sum_{n \geq n_0+q} \frac{\eps_n}{2^n},u_2,t)-C_\sigma(u_1,u_2,t)|\mathrm du_1\mathrm du_2
 \\
 \begin{aligned}
 &\leq \int_0^{1-h_1} \!\!\int_0^T  \left|\frac{1}{m_\sigma(u_1+\textstyle \sum_{n \geq n_0+q} \frac{\eps_n}{2^n},s)}-\frac{1}{m_\sigma(u_1,s)} \right| \mathrm ds\mathrm du_1
 \\
 &+\int_0^{1-h_1} \!\!\!\!\int_0^1\!\!\int_0^T \frac{2}{m_\sigma(u_2,s)}\left|\frac{m_\sigma(u_1+ \sum_{n \geq n_0+q} \frac{\eps_n}{2^n},u_2,s)}{m_\sigma(u_1+ \sum_{n \geq n_0+q} \frac{\eps_n}{2^n},s)}-\frac{m_\sigma(u_1,u_2,s)}{m_\sigma(u_1,s)} \right| \mathrm ds\mathrm du_1\mathrm du_2.
 \end{aligned}
\end{multline}
For each $s \in [0,T]$, $m_\sigma(\cdot,s)$ is right-continuous. Therefore, $m_\sigma(u_1+\textstyle \sum_{n \geq n_0+q} \frac{\eps_n}{2^n},s)$ converges to $m_\sigma(u_1,s)$ when $q \to +\infty$. 
Furthermore, there is $\beta >1$ such that almost surely, 
\begin{align*}
\sup_{u \in \left[0,\frac{1}{2^{n_0-1}}\right]} \int_0^{1-u} \!\!\int_0^T  \left|\frac{1}{m_\sigma(u_1+u,s)}-\frac{1}{m_\sigma(u_1,s)} \right|^\beta \mathrm ds\mathrm du_1<+\infty.
\end{align*}
Indeed, 
\begin{multline*}
\E{\sup_{u \in \left[0,\frac{1}{2^{n_0-1}}\right]}\int_0^{1-u} \!\! \int_0^T  \left|\frac{1}{m_\sigma(u_1+u,s)}-\frac{1}{m_\sigma(u_1,s)} \right|^\beta \mathrm ds\mathrm du_1}\\
\leq C_\beta \E{\int_0^1 \!\!\int_0^T \frac{1}{m_\sigma(u_1,s)^\beta}\mathrm ds \mathrm du_1}<+\infty,
\end{multline*}
by Lemma~\ref{lemme inverse masse approchee sig eps}. 
Therefore, since $\sum_{n \geq n_0+q} \frac{\eps_n}{2^n}\leq h_1<\frac{1}{2^{n_0-1}}$ for every $q \geq 1$, 
\begin{multline*}
\int_0^{1-h_1} \!\! \int_0^T  \left|\frac{1}{m_\sigma(u_1+\textstyle \sum_{n \geq n_0+q} \frac{\eps_n}{2^n},s)}-\frac{1}{m_\sigma(u_1,s)} \right|^\beta \mathrm ds\mathrm du_1
\\ \leq \int_0^{1-\sum_{n \geq n_0+q} \frac{\eps_n}{2^n}} \!\! \int_0^T  \left|\frac{1}{m_\sigma(u_1+\textstyle \sum_{n \geq n_0+q} \frac{\eps_n}{2^n},s)}-\frac{1}{m_\sigma(u_1,s)} \right|^\beta \mathrm ds\mathrm du_1 \\
\leq \sup_{u \in \left[0,\frac{1}{2^{n_0-1}}\right]}\int_0^{1-u} \!\! \int_0^T  \left|\frac{1}{m_\sigma(u_1+u,s)}-\frac{1}{m_\sigma(u_1,s)} \right|^\beta \mathrm ds\mathrm du_1,
\end{multline*}
which is almost surely finite. Thus the first term of the right hand side of~\eqref{masse continue à droite} tends almost surely to $0$ for every $h_1 < \frac{1}{2^{n_0}}$. 
A similar argument shows that the second term of the right hand side of~\eqref{masse continue à droite} also converges to $0$. 
Hence we have justified convergence~\eqref{conv en p}. 

When $q \to \infty$ in inequality~\eqref{ineg avant conv}, we obtain:
\begin{align*}
\int_0^{1-h_1} \!\!\int_0^1 \sup_{t \leq T} |C_\sigma(u_1+h_1,u_2,t)-C_\sigma(u_1,u_2,t)|\mathrm du_1\mathrm du_2 
\leq \sum_{k=1}^{+\infty} \frac{1}{2^{(n_0+k)\frac{\alpha}{2}}}
\leq \frac{C_\alpha}{2^{\frac{n_0\alpha}{2}}}.
\end{align*}
Then, we proceed similarly for the second term of the right hand side of~\eqref{decomposition crochet} and we finally obtain, for each $h=(h_1,h_2)$ such that $0<h_1< \frac{1}{2^{n_0}}$ and $0<h_2<\frac{1}{2^{n_0}}$,
\begin{align*}
\int_0^{1-h_1}\!\!\int_0^{1-h_2}\!\! \sup_{t \leq T} |C_\sigma(u_1+h_1,u_2+h_2,t)-C_\sigma(u_1,u_2,t)|\mathrm du_1\mathrm du_2
\leq \frac{C}{2^{\frac{n_0\alpha}{2}}}. 
\end{align*}
Choosing $H=\frac{1}{2^{n_0}}$ such that $CH^{\alpha/2} \leq \frac{1}{k}$, we get~\eqref{second criterion} for each $\sigma$  in $\Q_+$.

\paragraph{Conclusion of the proof.}
By Simon's tightness criterion on $L_1([0,1]\times [0,1],\mathcal C[0,T])$, the collection of laws of $(C_\sigma)_{\sigma \in \Q_+}$ is relatively compact in $\mathcal P(L_1([0,1]\times [0,1],\mathcal C[0,T]))$.
Therefore the collection of laws of $(y_\sigma,C_\sigma)_{\sigma \in \Q_+}$ is also relatively compact in $\mathcal P(L_2([0,1],\mathcal C[0,T])\times L_1([0,1]\times [0,1],\mathcal C[0,T]))$. Thus there is a subsequence,   $(y_{\sigma_n},C_{\sigma_n})_{n \geq 1}$ converges in distribution in $L_2([0,1],\mathcal C[0,T])\times L_1([0,1]\times [0,1],\mathcal C[0,T])$. We denote by $(y,C)$ the limit. We want to prove that for almost every $u_1,u_2 \in [0,1]$, $C(u_1,u_2,\cdot)$ is the quadratic variation of $y(u_1,\cdot)-y(u_2,\cdot)$ relatively to the filtration generated by $Y$ and $C$.

Let $l \geq 1$, $0\leq s_1\leq s_2\leq \dots \leq s_l\leq s\leq t$ and $f_l:(L_2(0,1))^{l}\times L_1([0,1]\times [0,1])^l \to \R$ be a bounded and continuous function. 
For every non-negative $\psi_1,\psi_2 \in L_\infty (0,1)$, we have for every $n \geq 1$:
\begin{multline*}
\mathbb E \Bigg [\int_0^1 \!\!  \int_0^1 \psi_1(u_1)\psi_2(u_2) \Big( (y_{\sigma_n}(u_1,t)-y_{\sigma_n}(u_2,t)-g(u_1)+g(u_2))^2\\
- (y_{\sigma_n}(u_1,s)-y_{\sigma_n}(u_2,s)-g(u_1)+g(u_2))^2 -C_{\sigma_n}(u_1,u_2,t)+C_{\sigma_n}(u_1,u_2,s) \Big) \mathrm du_1\mathrm du_2\\
f_l(Y_{\sigma_n}(s_1),\dots,Y_{\sigma_n}(s_l),C_{\sigma_n}(s_1),\dots,C_{\sigma_n}(s_l))
\Bigg]=0,
\end{multline*}
since the process $(C_{\sigma_n}(t))_{t \in [0,T]}:=(C_{\sigma_n}(\cdot,\cdot,t))_{t \in [0,T]}$ is $(\mathcal F^{\sigma_n}_t)_{t \in [0,T]}$-adapted. 
By the convergence in distribution, we obtain when $n$ goes to $\infty$:
\begin{multline*}
\mathbb E \Bigg [\int_0^1 \!\!  \int_0^1 \psi_1(u_1)\psi_2(u_2) \Big( (y (u_1,t)-y (u_2,t)-g(u_1)+g(u_2))^2\\
- (y (u_1,s)-y (u_2,s)-g(u_1)+g(u_2))^2 -C (u_1,u_2,t)+C (u_1,u_2,s) \Big) \mathrm du_1\mathrm du_2\\
f_l(Y (s_1),\dots,Y (s_l),C (s_1),\dots,C (s_l))
\Bigg]=0.
\end{multline*}
By Fubini's Theorem, we obtain that for almost every $u_1,u_2 \in (0,1)$,  for all rational numbers $(s_1, \dots, s_l,s,t)$ such that $0\leq s_1\leq s_2\leq \dots \leq s_l\leq s\leq t$:
\begin{multline*}
\mathbb E \bigg [ \Big( (y(u_1,t)-y(u_2,t)-g(u_1)+g(u_2))^2
- (y(u_1,s)-y(u_2,s)-g(u_1)+g(u_2))^2 \\-C(u_1,u_2,t)+C(u_1,u_2,s) \Big) 
f_l(Y(s_1),\dots,Y(s_l),C (s_1),\dots,C (s_l))
\bigg]=0.
\end{multline*}
By continuity in time, the latter equality remains true for every $0\leq s_1\leq s_2\leq \dots \leq s_l\leq s\leq t$. 
Furthermore, for almost every $u_1,u_2$, $(C_{\sigma_n}(u_1,u_2,t))_{t \in [0,T]}$ is a non-decreasing bounded variation process. This remains true for the limit $(C(u_1,u_2,t))_{t \in [0,T]}$. Therefore, we deduce that 
\begin{align*}
C(u_1,u_2,t)=\langle y(u_1)-y(u_2),y(u_1)-y(u_2)\rangle_t,
\end{align*}
for almost every $u_1,u_2 \in (0,1)$, with respect to the filtration generated by $(Y,C)$. 
\end{proof}

We conclude this Paragraph by using Fatou's Lemma to extend the statement of Lemma~\ref{lemme inverse masse approchee sigma} to the limit process: 
\begin{prop}
\label{prop inverse masse}
Let $g \in L_p(0,1)$.
For all  $\beta \in (0, \frac{3}{2}-\frac{1}{p})$, there is a constant $C >0$ depending only on $\beta$ and $\|g\|_{L_p}$ such that for all $0\leq s<t\leq T$, we have the following inequality:
\begin{align*}
\E{\int_s^t  \!\! \int_0^1   \frac{1}{m(u,r)^\beta}\mathrm du\mathrm dr}\leq C\sqrt{t-s}. 
\end{align*}
\end{prop}

By Burkholder-Davis-Gundy inequality, we deduce the following estimation:
\begin{coro}
\label{coro estimation y puissance 2p}
For each $\beta \in (0, \frac{3}{2}-\frac{1}{p})$, $\displaystyle  \sup_{t \leq T}\E{\int_0^1(y(u,t)-g(u))^{2\beta}\mathrm du}<+\infty$. 
\end{coro}

\subsection{Covariation of $y(u,\cdot)$ and $y(u',\cdot)$}
\label{para 3}

In this Paragraph, we want to complete the proof of property~$(C5)$ of Theorem~\ref{théo 2}. It remains to prove the following Proposition:
\begin{prop} Let $y$ be the version in $\mathcal D((0,1),\mathcal C[0,T])$ of the limit process given by Proposition~\ref{propriete B4}.   For every $u,u' \in (0,1)$, 
\begin{align}
\langle y(u,\cdot),y(u',\cdot)\rangle_{t\wedge \tau_{u,u'}}=0,
\label{equ prop 5.12}
\end{align}
where $\tau_{u,u'}=\inf\{t\geq 0: y(u,t)=y(u',t)\}\wedge T$.
\label{propriete B5 partie 2}
\end{prop}

As in the previous Paragraph, we will need to prove the convergence of the joint law of $y_\sigma$ and a quadratic covariation. More precisely, define:
\begin{align*}
K_ \sigma(u,u',t):=\int_0^t \frac{m_\sigma(u,u',s)}{m_\sigma(u,s) m_\sigma(u',s)}\mathrm ds.
\end{align*}
We state the following result:
\begin{lemme}
\label{K lemme conv des crochets}
For every sequence $(\sigma_n)_n$ of rational numbers tending to $0$, we can extract a subsequence $(\widetilde{\sigma}_n)_n$ such that 
the sequence $(y_{\widetilde{\sigma}_n},K_{\widetilde{\sigma}_n})_{n \to \infty}$ converges in distribution to $(y,K)$ in $L_2([0,1],\mathcal C[0,T])$ $\times L_1([0,1]\times [0,1],\mathcal C[0,T])$, where
\begin{align*}
\label{K identification limites 2}
K(u,u',t):=\langle y(u,\cdot),y(u',\cdot) \rangle_t.
\end{align*}
\end{lemme}

\begin{proof}[Proof (Lemma~\ref{K lemme conv des crochets})]
We follow the same structure as in the proof of Lemma~\ref{lemme conv des crochets}. 
First, we define $K_{\sigma,\eps}= \langle y_{\sigma,\eps}(u,\cdot),y_{\sigma,\eps}(u',\cdot) \rangle_t= \int_0^t \frac{m_{\sigma,\eps}(u,u',s)}{(\eps+m_{\sigma,\eps}(u,s))(\eps+ m_{\sigma,\eps}(u',s))}\mathrm ds$.
We show that $K_{\sigma,\eps}$ satisfies the three criteria of tightness in $L_1([0,1]\times [0,1],\mathcal C[0,T])$. For the first criterion, we want to bound
\begin{align*}
\E{\int_0^1\!\!\int_0^1 \sup_{t\leq T} |K_{\sigma,\eps}(u,u',t)| \mathrm du\mathrm du'}
\end{align*}
uniformly for $\sigma,\eps  \in \Q_+$. This follows from Kunita-Watanabe's inequality:
\begin{align*}
|K_{\sigma,\eps}(u,u',t)|
=|\langle y_{\sigma,\eps}(u),y_{\sigma,\eps}(u')\rangle_t|
&\leq \langle y_{\sigma,\eps}(u),y_{\sigma,\eps}(u)\rangle_t^{1/2}
\langle y_{\sigma,\eps}(u'),y_{\sigma,\eps}(u')\rangle_t^{1/2}\\
&\leq \langle y_{\sigma,\eps}(u),y_{\sigma,\eps}(u)\rangle_T^{1/2}
\langle y_{\sigma,\eps}(u'),y_{\sigma,\eps}(u')\rangle_T^{1/2}
\end{align*}
and from Cauchy-Schwarz inequality:
\begin{align*}
\E{\int_0^1\!\!\int_0^1 \sup_{t\leq T} |K_{\sigma,\eps}(u,u',t)| \mathrm du\mathrm du'}
&\leq \E{\int_0^1  \langle y_{\sigma,\eps}(u),y_{\sigma,\eps}(u)\rangle_T\mathrm du}\\
&= \E{\int_0^1 (y_{\sigma,\eps}(u,T)-g(u))^2\mathrm du},
\end{align*}
which is bounded uniformly for $\sigma, \eps \in \Q_+$ by Corollary~\ref{corollaire borne L_2 sig eps}.

We refer to the proof of Lemma~\ref{lemme conv des crochets} for the second and the third criteria of tightness, and for the rest of the proof, which follows in the same way. 
It remains to explain why $(K(u,u',t))_{t \in [0,T]}$ is a bounded variation process for almost every $u, u' \in (0,1)$. It follows from Kunita-Watanabe's inequality that:
\begin{multline*}
\sum_{k=0}^{p-1}
|K_{\sigma,\eps}(u,u',t_{k+1})-K_{\sigma,\eps}(u,u',t_k)|
=\sum_{k=0}^{p-1}
|\langle y_{\sigma,\eps}(u),y_{\sigma,\eps}(u') \rangle_{t_{k+1}}-\langle y_{\sigma,\eps}(u),y_{\sigma,\eps}(u') \rangle_{t_k} |\\
\begin{aligned}
&\leq \sum_{k=0}^{p-1}
\left(\int_{t_k}^{t_{k+1}} \!\!\mathrm d\langle y_{\sigma,\eps}(u),y_{\sigma,\eps}(u) \rangle_s \right)^{\frac{1}{2}}
\left(\int_{t_k}^{t_{k+1}} \!\!\mathrm d\langle y_{\sigma,\eps}(u'),y_{\sigma,\eps}(u') \rangle_s \right)^{\frac{1}{2}}\\
&\leq \frac{1}{2} \int_{t_0}^{t_p} \mathrm d\langle y_{\sigma,\eps}(u),y_{\sigma,\eps}(u) \rangle_s +\frac{1}{2} \int_{t_0}^{t_p} \mathrm d\langle y_{\sigma,\eps}(u'),y_{\sigma,\eps}(u') \rangle_s\\
&=\frac{1}{2} \int_{t_0}^{t_p} \frac{\mathrm ds}{M_{\sigma,\eps}(u,s)} +\frac{1}{2} \int_{t_0}^{t_p} \frac{\mathrm ds}{M_{\sigma,\eps}(u',s)},
\end{aligned}
\end{multline*}
%which is bounded independently of $\sigma$ and $\eps$ almost surely for almost every $u, u' \in (0,1)$. 
Therefore, for every $p \geq 1$ and $0\leq t_0 \leq t_1\leq \dots \leq t_p$, $\sum_{k=0}^{p-1}|K(u,u',t_{k+1})-K(u,u',t_k)| \leq \frac{1}{2} \int_{t_0}^{t_p} \frac{\mathrm ds}{m(u,s)} +\frac{1}{2} \int_{t_0}^{t_p} \frac{\mathrm ds}{m(u',s)}$. 
By Proposition~\ref{prop inverse masse}, we know that almost surely and for almost every $u\in (0,1)$, $\int_0^T \frac{\mathrm ds}{m(u,s)}$ is finite. Thus for almost every $u$ and $u'$ in $(0,1)$, $K(u,u',\cdot)$ is a bounded variation process. 
This concludes the proof of the Lemma. 
\end{proof}

We use the latter Lemma to prove Proposition~\ref{propriete B5 partie 2}. 
\begin{proof}[Proof (Proposition~\ref{propriete B5 partie 2})]
By Lemma~\ref{K lemme conv des crochets} and Skorohod's representation Theorem, we may suppose that $(y_\sigma,K_\sigma)_{\sigma\in \Q_+}$ converges almost surely in $L_2([0,1],\mathcal C[0,T])\times L_1([0,1]\times [0,1],\mathcal C[0,T])$ to $(y,K)$. 
As previously, up to extracting a subsequence, we deduce that for almost every $(\omega,u,u') \in \Omega \times[0,1] \times[0,1]$, 
\begin{equation}
\label{cvg 1}
\sup_{t \leq T} |y_\sigma(u,t)-y(u,t)|(\omega) \underset{\sigma \to 0}{\longrightarrow} 0,
\end{equation}
and 
\begin{equation}
\label{cvg 2}
\sup_{t \leq T} |K_\sigma(u,u',t)-K(u,u',t)|(\omega) \underset{\sigma \to 0}{\longrightarrow} 0.
\end{equation}
Therefore,  there exists a (non-random) subset $\A$ of $[0,1]$, such that for every $u,u' \in \A$, \eqref{cvg 1} and~\eqref{cvg 2} holds almost surely.

Let $u,u' \in \A$. 
If $g(u)=g(u')$ then $\tau_{u,u'}=0$ almost surely, thus~\eqref{equ prop 5.12} is clear. Up to exchanging $u$ and $u'$, assume that $g(u)<g(u')$.
Let $ \delta <2 (g(u')-g(u))$. 
Almost surely, by~\eqref{cvg 1}, there exists $\sigma_0$ such that for all $\sigma \in (0,\sigma_0)\cap \Q_+$, 
\begin{align*}
\sup_{t \leq T} |y_\sigma(u,t)-y(u,t)| &\leq \frac{\delta}{4},\\
\sup_{t \leq T} |y_\sigma(u',t)-y(u',t)| &\leq \frac{\delta}{4}.
\end{align*}
Define $\tau_{u,u'}^\delta:= \inf\{ t \geq 0: |y(u,t)-y(u',t)|\leq \delta \}\wedge T$. Therefore, for all $t <\tau_{u,u'}^\delta$ and for all $\sigma < \sigma_0$, $|y_\sigma(u,t)-y_\sigma(u',t)| \geq \frac{\delta}{2}$. 
Let $\sigma< \min (\sigma_0,\frac{\delta}{2})$. For all $t < \tau_{u,u'}^\delta$, we have $|y_\sigma(u,t)-y_\sigma(u',t)| \geq \sigma$ and thus $m_\sigma(u,u',t)=0$, hence $K_\sigma(u,u',t)=\int_0^t \frac{m_\sigma(u,u',s)}{m_\sigma(u,s) m_\sigma(u',s)}\mathrm ds=0$ for $t \leq \tau_{u,u'}^\delta$. By~\eqref{cvg 2}, we obtain 
\[\sup_{t \leq \tau_{u,u'}^\delta} |K(u,u',t)|=0.\]
Thus for every $\delta>0$, for every $u, u' \in \A$ and $t \leq \tau_{u,u'}^\delta$, $\langle y(u),y(u')\rangle_t=0$. 
Since $\tau_{u,u'}^\delta \to \tau_{u,u'}$ when $\delta \to 0$, we have for each $u,u' \in \A$:
\begin{align}
\label{crochet}
\langle y(u),y(u') \rangle_{t \wedge \tau_{u,u'}}=0.
\end{align}

It remains to show that~\eqref{crochet} holds for every $(u,u') \in (0,1)^2$. 
Let $(u,u') \in (0,1)^2$. As previously, we may assume that $g(u)<g(u')$. 
 By continuity of the processes $(y(u,t))_{t \in [0,T]}$ and $(y(u',t))_{t \in [0,T]}$, the first time of coalescence $\tau_{u,u'}$ is almost surely positive. 
Fix $l \geq 1$, $0\leq s_1\leq s_2\leq \dots \leq s_l\leq s\leq t$ and a bounded and continuous function $f_l:(L_2(0,1))^l\to \R$. Suppose that $s>0$. 
We want to prove that:
\begin{align}
\E{(y(u,t \wedge \tau_{u,u'})y(u',t \wedge \tau_{u,u'})
-y(u,s \wedge \tau_{u,u'})y(u',s \wedge \tau_{u,u'}))
f_l(Y(s_1),\dots,Y(s_l))}=0. 
\label{martingale temps d'arret}
\end{align}

Let $\eps>0$. For each $v \in (u,u+\eps)\cap \mathcal A$ and $v' \in (u',u'+\eps) \cap\mathcal A$ (since $\mathcal A$ is of plain measure in $(0,1)$, both sets are non-empty), since we have equality~\eqref{crochet}, 
\begin{multline}
\label{eq mart 1}
0= \mathbb E \big[(y(v,t \wedge \tau_{v,v'})y(v',t \wedge \tau_{v,v'})
-y(v,s \wedge \tau_{v,v'})y(v',s \wedge \tau_{v,v'})
)
f_l(Y(s_1),\dots,Y(s_l))\big].
\end{multline}

Let $t_0 \in (0,s)$.
We define 
\begin{align*}
\eta:=\sup \{ h \geq 0: y(u+h,t_0)=y(u,t_0) \text{ and } y(u'+h,t_0)=y(u',t_0) \}.
\end{align*}
By the coalescence property given by Proposition~\ref{propriete B4}, under the event $\{ \tau_{u,u'} > t_0\}$, we know that for every $r \geq t_0$, for each $v \in (u,u+\eta)$ and $v' \in (u',u'+\eta)$, $y(v,r)=y(u,r)$ and $y(v',r)=y(u',r)$, whence $\tau_{v,v'}=\tau_{u,u'}$. 
Thus, by equality~\eqref{eq mart 1}, we deduce that for each $v \in (u,u+\eps)\cap \mathcal A$ and $v' \in (u',u'+\eps) \cap\mathcal A$, 
\begin{multline}
\label{eq mart 2}
0= \mathbb E \big[\mathds 1_{\{\eta > \eps\}} \mathds 1_{\{\tau_{u,u'} > t_0\}}
(y(u,t \wedge \tau_{u,u'})y(u',t \wedge \tau_{u,u'})
\\-y(u,s \wedge \tau_{u,u'})y(u',s \wedge \tau_{u,u'}))
f_l(Y(s_1),\dots,Y(s_l))\big]
\\+\mathbb E \big[\mathds 1_{\{\eta \leq \eps\}\cup \{\tau_{u,u'} \leq t_0\}}(y(v,t \wedge \tau_{v,v'})y(v',t \wedge \tau_{v,v'})
-y(v,s \wedge \tau_{v,v'})y(v',s \wedge \tau_{v,v'}))
\\
f_l(Y(s_1),\dots,Y(s_l))\big].
\end{multline}
Let $h>0$ be such that $(u,u+\eps)$ and $(u',u'+\eps)$ are contained in $(h,1-h)$. Thus for every $v \in (u,u+\eps) \cap \mathcal A$, for every $r \in [0,T]$, 
by inequality~\eqref{inegalite delta 1-delta} and by Doob's inequality, we deduce that:
\begin{align*}
\E{\sup_{r \leq T} y(v,r)^{2\beta}}\leq \frac{2}{h}\E{\int_0^1 \sup_{r \leq T} y(x,r)^{2\beta} \mathrm dx}\leq \frac{C_\beta}{h}\E{\int_0^1  y(x,T)^{2\beta} \mathrm dx}\leq \frac{\widetilde{C}_\beta}{h},
\end{align*}
for a  $\beta$ arbitrarily chosen in $(1,\frac{3}{2}-\frac{1}{p})$ (by Corollary~\ref{coro estimation y puissance 2p}). Thus, there exists $\beta>1$ such that
$\E{(y(v,t \wedge \tau_{v,v'})y(v',t \wedge \tau_{v,v'})^{\beta}}$
is uniformly bounded for $v \in (u,u+\eps)$ and $v' \in (u',u'+\eps)$. 
Let $\alpha=1-\frac{1}{\beta}$. 
Therefore, we deduce from~\eqref{eq mart 2} that there is a constant $C$ depending only on $u$, $u'$ and $\alpha$ such that:
\begin{multline}
\label{67}
\mathbb E \big[\mathds 1_{\{\eta > \eps\}} \mathds 1_{\{\tau_{u,u'} > t_0\}}
(y(u,t \wedge \tau_{u,u'})y(u',t \wedge \tau_{u,u'})
-y(u,s \wedge \tau_{u,u'})y(u',s \wedge \tau_{u,u'}))
\\f_l(Y(s_1),\dots,Y(s_l))\big]
\leq C\left(\P{\eta \leq \eps}^\alpha+\P{\tau_{u,u'} \leq t_0}^\alpha\right).
\end{multline}
We divide the left hand side of inequality~\eqref{67} into two parts by writing $\mathds 1_{\{\eta > \eps\}} \mathds 1_{\{\tau_{u,u'} > t_0\}}=1-\mathds 1_{\{\eta \leq \eps\}\cup \{\tau_{u,u'} \leq t_0\}}$ and we estimate the second term in the same way as above. We deduce that there is a constant $C'$ such that:
\begin{multline*}
\E{(y(u,t \wedge \tau_{u,u'})y(u',t \wedge \tau_{u,u'})
-y(u,s \wedge \tau_{u,u'})y(u',s \wedge \tau_{u,u'}))
f_l(Y(s_1),\dots,Y(s_l))}
\\ \leq C'\left(\P{\eta \leq \eps}^\alpha+\P{\tau_{u,u'} \leq t_0}^\alpha\right).
\end{multline*}
Let $\delta>0$. Since $\tau_{u,u'} >0$ almost surely, we choose $t_0\in (0,s)$ such that $\P{\tau_{u,u'} \leq t_0}^\alpha \leq \delta$. 
Since $t_0>0$, we know by Proposition~\ref{prop step function} that $y(\cdot,t_0)$ is almost surely a step function, so $\eta>0$ almost surely. Therefore, we can choose $\eps >0$ so that $\P{\eta \leq \eps}^\alpha \leq \delta$. 
This concludes the proof of equality~\eqref{martingale temps d'arret}. 

Recall that we suppose that $t\geq s>0$. By continuity of time of $y(u,\cdot)$ and $y(u',\cdot)$, equality~\eqref{martingale temps d'arret} also holds for $s=0$. Therefore, $y(u,t \wedge \tau_{u,u'})y(u',t \wedge \tau_{u,u'})$ is a $(\mathcal F_t)_{t \in [0,T]}$-martingale and $\langle y(u),y(u') \rangle_{t \wedge \tau_{u,u'}}=0$. This 
concludes the proof of Proposition~\ref{propriete B5 partie 2}.  
\end{proof}

\appendix
\section{Appendix: It\^o's formula for the Wasserstein diffusion}
\label{appendix}

Let $g \in \mathcal L^\uparrow_{2+}[0,1]$. We assume, to simplify the notations, that $g(1)$ is finite, but the proof can be easily adapted to functions $g$ with $g(u) \underset{u \to 1}{\longrightarrow} +\infty$. Let $y$ be a process in $\mathcal D([0,1], \mathcal C[0,T])$ satisfying $(i)-(iv)$ (see Introduction).

Recall that the process $y(\cdot,t)_{t \in [0,T]}$ can be considered as the quantile function of $(\mu_t)_{t \in [0,T]}$, by setting $\mu_t=\Leb|_{[0,1]}\circ y(\cdot,t)^{-1}$. 
The latter process has every feature of a Wasserstein diffusion. We describe in this Paragraph the dynamics of the process $(\mu_t)_{t \in [0,T]}$, after having introduced a differential calculus on $\mathcal P_2(\R)$ due to Lions (\cite{lions_college}, \cite{cardaliaguet13}).
%We generalize a formula given by Konarovskyi and von Renesse in~\cite{konarovskyi_modified_2015}. 
%We prove that for a smooth function $U: \mathcal P_2(\R) \rightarrow \R$, the process $U(\mu_t)$ is a semi-martingale with quadratic variation proportional to the square of the gradient of $U$ (see theorem~\ref{theo Ito}), which is the same type of result that we obtain thanks to the classical Itô's formula for a standard Brownian motion. 
We prove that, for a smooth function $U: \mathcal P_2(\R) \rightarrow \R$, the process $(U(\mu_t))_{t \in [0,T]}$ is a semi-martingale with quadratic variation proportional to the square of the gradient of $U$ (see Theorem~\ref{theo Ito}). This result is a generalization of the formula given by Konarovskyi and von Renesse in~\cite{konarovskyivonrenesse18}. We compare it to a similar result obtained by von Renesse and Sturm~\cite{vrenessesturm09} for the Wasserstein diffusion on $[0,1]$ (see Remark~\ref{remark vR S}). 

In order to describe the dynamics of $(\mu_t)_{t \in [0,T]}$, we begin by a discretization in space and by writing the classical Itô formula for that discretized process. 
Let introduce $\widetilde{\mu}_t^n:=\frac{1}{n}\sum_{k \in [n]} \delta_{y(\frac{k}{n},t)}$, where $[n]$ denotes the set $\{1,\dots,n\}$. Fix $U: \mathcal P_2(\R) \rightarrow \R$ a continuous function, with respect to the Wasserstein distance $W_2$ on $\mathcal P_2(\R)$. Let define $U^n(x_1,\dots,x_n):=U(\frac{1}{n}\sum_{j \in [n]} \delta_{x_j} )$. Remark that $U(\widetilde{\mu}_t^n)=U^n\left(y(\frac{1}{n},t),y(\frac{2}{n},t),\dots,y(1,t)\right)$. 
Assuming that $U^n$ belongs to $C^2(\R^n)$, and using that $y(\frac{k}{n},\cdot)$ is a square integrable continuous martingale on $[0,T]$, we have (recall that $g(1)$ is finite):
\begin{align}
U(\widetilde{\mu}_t^n)
=&U^n(\textstyle{g(\frac{1}{n}),\dots,g(1)}) + \displaystyle\sum_{k\in [n]} \int_0^t \partial_k U^n(\textstyle{y(\frac{1}{n},s),\dots,y(1,s)}) \, \mathrm dy(\frac{k}{n},s) \notag\\
&+ \frac{1}{2} \sum_{k,l \in [n]} \int_0^t \partial_{k,l}^2 U^n(\textstyle{y(\frac{1}{n},s),\dots,y(1,s)}) \,\mathrm d \langle y(\frac{k}{n},\cdot),y(\frac{l}{n},\cdot) \rangle_s.
\label{eq Ito}
\end{align}

In order to write the derivatives of $U^n$ in terms of derivatives of $U$, we should introduce a differential calculus on $\mathcal P_2(\R)$, well-adapted to the differentiation of empirical measures. P.L. Lions introduces in his lectures at Collège de France (see Section 6.1 of Cardaliaguet's notes~\cite{cardaliaguet13}) a differential calculus on $\mathcal P_2(\R)$ by using the Hilbertian structure of $L_2(\Omega)$. 
We set $\widetilde{U}(X):=U(\operatorname{Law}(X))$ for all $X \in L_2(\Omega)$. 

A function $U:\mathcal P_2(\R) \rightarrow \R$ is said to be \emph{$L$-differentiable} (or differentiable in the sense of Lions) at a point $\mu_0 \in \mathcal P_2(\R)$ if there is a random variable $X_0$ with law $\mu_0$ such that $\widetilde{U}$ is Fréchet-differentiable at $X_0$. The definition does not depend on the choice of the representative $X_0$ of the law $\mu_0$, and if $X_0$ and $X_1$ have the same law, then the laws of $D\widetilde{U}(X_0)$ and $D\widetilde{U}(X_1)$ are equal (see e.g.~\cite{cardaliaguet13}). 
Furthermore, if  $D\widetilde{U}:L_2(\Omega)\rightarrow L_2(\Omega)$ is a continuous function, then for all $\mu_0 \in \mathcal P_2(\R)$, there exists a measurable function $\R \rightarrow \R$, denoted by $\partial_\mu U(\mu_0)$, such that for each $X \in L_2(\Omega)$ with law $\mu_0$, we have $D\widetilde{U}(X)=\partial_\mu U(\mu_0)(X)$ almost surely (see~\cite{cardaliaguet13}). 

In~\cite{carmonadelarue18}, Carmona and Delarue prove that the L-differentiability of $U:\mathcal P_2(\R) \rightarrow \R$ implies the differentiability of $U^n$ on $\R^n$, and that we have for each $k \in [n]$:
\[
\partial_k U^n (x_1,\dots,x_n)= \frac{1}{n} \partial_\mu U(\textstyle{\frac{1}{n}\sum_{j\in[n]}\delta_{x_j}})(x_k).\]

Furthermore, assume that $U$ is L-differentiable and that $(\mu,v) \in \mathcal P_2(\R) \times \R \mapsto\partial_\mu U(\mu)(v) \in \R$ is continuous. Moreover, we assume that for every $\mu \in \mathcal P_2(\R)$, the map $v \in \R \mapsto \partial_\mu U(\mu)(v) \in \R$ is differentiable on $\R$ in the classical sense and that its derivative is given by a jointly continuous function $(\mu,v) \mapsto \partial_v\partial_\mu U(\mu)(v)$. We also assume that for every $v \in \R$, the map $\mu \mapsto \partial_\mu U(\mu)(v)$ is L-differentiable and its derivative is denoted by $(\mu,v,v') \mapsto \partial_\mu^2 U(\mu)(v,v')$. 
Then, $U^n$ is $\mathcal C^2$ on $\R^n$ and for all $k,l \in [n]$:
\[
\partial_{k,l}^2 U^n (x_1,\dots,x_n)= \frac{1}{n}\partial_v \partial_\mu U(\textstyle{\frac{1}{n}\sum_{j\in[n]}\delta_{x_j}})(x_k) \mathds 1_{\{k=l\}} 
+\displaystyle \frac{1}{n^2} \partial_\mu^2 U(\textstyle{\frac{1}{n}\sum_{j\in[n]}\delta_{x_j}}) (x_k,x_l).\]

Therefore, we obtain from equation~(\ref{eq Ito}):
\begin{align}
U(\widetilde{\mu}_t^n)
=& U(\widetilde{\mu}_0^n)+ \frac{1}{n} \sum_{k\in [n]} \int_0^t \partial_\mu U(\widetilde{\mu}_s^n) (\textstyle y(\frac{k}{n},s)) \mathrm dy(\frac{k}{n},s) \notag
+\displaystyle\frac{1}{2n} \sum_{k \in [n]} \int_0^t \partial_v\partial_\mu U(\widetilde{\mu}_s^n)(\textstyle y(\frac{k}{n},s)) \frac{\mathrm ds}{m(\frac{k}{n},s)} \notag\\
&+ \frac{1}{2n^2} \sum_{k,l \in [n]} \int_0^t \partial_\mu^2 U(\widetilde{\mu}_s^n)(\textstyle y(\frac{k}{n},s),y(\frac{l}{n},s)) \frac{\mathds 1_{\{\tau_{\frac{k}{n},\frac{l}{n}}\leq s\}}}{m(\frac{k}{n},s)} \mathrm ds.
\label{eq nouvelle}
\end{align}
By property of coalescence, if $\tau_{\frac{k}{n},\frac{l}{n}}\leq s$, we have $y(\frac{k}{n},s)=y(\frac{l}{n},s)$, so that the last term in the latter equation is equal to:
\begin{align*}
\frac{1}{2n} \sum_{k \in [n]} \int_0^t \partial_\mu^2 U(\widetilde{\mu}_s^n)(\textstyle y(\frac{k}{n},s),y(\frac{k}{n},s)) \frac{\frac{1}{n}\sum_{l \in [n]}\mathds 1_{\{\tau_{\frac{k}{n},\frac{l}{n}}\leq s\}}}{m(\frac{k}{n},s)} \mathrm ds.
\end{align*}
Observe that the difference between $\frac{1}{n}\sum_{l \in [n]}\mathds 1_{\{\tau_{\frac{k}{n},\frac{l}{n}}\leq s\}}$ and $m(\frac{k}{n},s)=\int_0^1 \mathds 1_{\{\tau_{\frac{k}{n},u} \leq s  \} } \mathrm du$ is bounded by $\frac{2}{n}$, since the set $\{ u: \tau_{\frac{k}{n},u} \leq s\}$ is an interval. 

We want to let $n$ tend to $+ \infty$ in order to obtain an Itô formula for the limit process. We start by proving the convergence of a subsequence of $((\widetilde{\mu}_t^n)_{t \in [0,T]})_{n\geq 1}$ to $(\mu_t)_{t \in [0,T]}$ with respect to the $L_2$-Wasserstein distance.

\begin{prop}
\label{prop conv mesures}
There exists a subsequence $((\widetilde{\mu}_t^{\varphi(n)})_{t \in [0,T]})_{n\geq 1}$ of $((\widetilde{\mu}_t^{n})_{t \in [0,T]})_{n\geq 1}$ such that, 
%Up to extracting a subsequence, 
for almost every $t \in [0,T]$, the sequence $(\widetilde{\mu}_t^{\varphi(n)})_{n\geq 1}$ converges almost surely to $\mu_t$ with respect to the Wasserstein distance $W_2$. 
\end{prop}

\begin{rem}
We point out that the extraction function $\varphi$ does not depend on $t \in [0,T]$. 
\end{rem}

\begin{proof}
To obtain the statement of the Proposition, it is sufficient to prove that:
\[
\E{\int_0^T W_2(\widetilde{\mu}_t^n,\mu_t)^2\mathrm dt}\rightarrow 0.
\]

Let $V$ be a uniform random variable on $[0,1]$, defined on a probability space $(\widetilde{\Omega},\widetilde{\mathcal F},\widetilde{\mathbb P})$. 
Therefore, $\mu_t$ is the law of $y(V,t)$ and $\widetilde{\mu}_t^n$ the law of $\sum_{k \in [n]} \mathds 1_{\{\frac{k-1}{n}<V\leq \frac{k}{n}\}} y(\frac{k}{n},t)$. 
Hence we have:
\begin{align*}
W_2(\widetilde{\mu}_t^n,\mu_t)^2 &\leq \widetilde{\mathbb E} \left[  \left( \sum_{k \in [n]}  \mathds 1_{\{\frac{k-1}{n}<V\leq \frac{k}{n}\}} y(\textstyle\frac{k}{n},t)-y(V,t)   \right)^2\right]\\
&= \int_0^1 \sum_{k \in [n]}  \mathds 1_{\{\frac{k-1}{n}<u\leq \frac{k}{n}\}} | y(\textstyle \frac{k}{n},t)-y(u,t) |^2 \mathrm du.
\end{align*}

Therefore, it is sufficient to show that:
\begin{equation}
\label{eq page 5}
\E{\int_0^T\!\!\int_0^1 \sum_{k \in [n]}\mathds 1_{\{\frac{k-1}{n}<u\leq \frac{k}{n}\}} |y(\textstyle \frac{k}{n},t)-y(u,t)|^2\mathrm du \mathrm dt} \underset{n\to+\infty}{\longrightarrow} 0.
\end{equation}
Fixing $u \in (0,1), t \in (0,T)$, $ \sum_{k \in [n]} \mathds 1_{\{\frac{k-1}{n}<u\leq \frac{k}{n}\}} |y(\frac{k}{n},t)-y(u,t)|^2$ converges almost surely to 0 by the right-continuity of $y(\cdot,t)$ at point $u$. To prove~\eqref{eq page 5}, we have to show a uniform integrability property, \textit{i.e.} that $\sup_{n \geq 1} \E{\left(\int_0^T\int_0^1 \sum_{k \in [n]}\mathds 1_{\{\frac{k-1}{n}<u\leq \frac{k}{n}\}} |y(\textstyle \frac{k}{n},t)-y(u,t)|^2\mathrm du \mathrm dt\right)^\beta}<+\infty$ for some $\beta >1$.

We compute:
\begin{multline*}
\E{\left(\int_0^T\!\!\int_0^1 \sum_{k \in [n]}\mathds 1_{\{\frac{k-1}{n}<u\leq \frac{k}{n}\}} |y(\textstyle \frac{k}{n},t)-y(u,t)|^2\mathrm du \mathrm dt\right)^\beta}^{1/(2\beta)}
\\
\leq T^{\frac{\beta-1}{2 \beta}}
\E{\int_0^T\!\!\int_0^1 \sum_{k \in [n]}\mathds 1_{\{\frac{k-1}{n}<u\leq \frac{k}{n}\}} |y(\textstyle \frac{k}{n},t)-y(u,t)|^{2 \beta}\mathrm du \mathrm dt}^{1/(2\beta)}
\\
\begin{aligned}
&\leq T^{\frac{\beta-1}{2 \beta}}
\E{\int_0^T\!\!\int_0^1 \sum_{k \in [n]}\mathds 1_{\{\frac{k-1}{n}<u\leq \frac{k}{n}\}} |M_0|^{2 \beta}\mathrm du \mathrm dt}^{1/(2\beta)}
\\
&\quad +T^{\frac{\beta-1}{2 \beta}}
\E{\int_0^T\!\!\int_0^1 \sum_{k \in [n]}\mathds 1_{\{\frac{k-1}{n}<u\leq \frac{k}{n}\}} |M_t-M_0|^{2 \beta}\mathrm du \mathrm dt}^{1/(2\beta)},
\end{aligned}
\end{multline*}
where $M_t=y(\frac{k}{n},t)-y(u,t)$. Recall that by property $(i)$ of the process $y$, $M_0=g(\frac{k}{n})-g(u)$. We deduce that:
\begin{align*}
\E{\int_0^T\!\!\int_0^1 \sum_{k \in [n]}\mathds 1_{\{\frac{k-1}{n}<u\leq \frac{k}{n}\}} \textstyle |g(\frac{k}{n})-g(u)|^{2 \beta}\mathrm du \mathrm dt}\leq T C_\beta \E{\int_0^1 g(u)^{2\beta}\mathrm du}.
\end{align*}
Since $g$ belongs to $\mathcal L^\uparrow_{2+}[0,1]$, there exists $p>2$ such that $g \in L_p(0,1)$. Therefore, we can choose $\beta >1$ such that $2 \beta \leq p$. 
By Burkholder-Davis-Gundy inequality and the martingale property of $M$, we have:
\begin{align*}
\E{(M_t-M_0)^{2\beta}}\leq C_\beta \E{\langle M,M \rangle_t^\beta}.
\end{align*}
By property $(iv)$, 
\begin{align*}
\langle M,M \rangle_t
&=\int_0^t \frac{\mathrm ds}{m(\textstyle \frac{k}{n},s)}
+\int_0^t \frac{\mathrm ds}{m(u,s)}
-2\int_0^t \frac{\mathds 1_{\{\tau_{\frac{k}{n},u}\leq s\}}}{m(\textstyle \frac{k}{n},s)^{1/2}m(u,s)^{1/2}}\mathrm ds\\
&\leq\int_0^t \frac{\mathrm ds}{m(\textstyle \frac{k}{n},s)}
+\int_0^t \frac{\mathrm ds}{m(u,s)},
\end{align*}
so that there is a constant $C_\beta$ satisfying:
\begin{align*}
\E{\langle M,M \rangle_t^\beta}
\leq C_\beta t^{\beta -1} \E{\int_0^t \frac{\mathrm ds}{m(\textstyle \frac{k}{n},s)^\beta}
+\int_0^t \frac{\mathrm ds}{m(u,s)^\beta}}.
\end{align*}
To conclude, we use the following statement: provided $\beta < \frac{3}{2}-\frac{1}{p}$, there is a constant $C_\beta$ such that for each $t$ and $u$:
%for each $\beta \in (1,\frac{3}{2})$, there is a constant $C_\beta$ such that for each $t$ and $u$:
\begin{align}
\label{estimation inv mass}
\E{\int_0^1\!\!\int_0^t \frac{\mathrm ds}{m(u,s)^\beta}\mathrm du}\leq C_\beta \sqrt{t}.
\end{align}
This statement is Proposition~\ref{prop inverse masse} for the limit process that we constructed in this paper, or in~\cite[Prop. 4.3]{konarovskyi17behavior} for the process constructed by Konarovskyi. 
This completes the proof.
\end{proof}

By similar arguments of convergence, equation~(\ref{eq nouvelle}) leads to the following Itô formula for $(\mu_t)_{t \in [0,T]}$, by letting $n$ tend to $\infty$.
The estimation~\eqref{estimation inv mass} is the key of the proof of those convergences.

\begin{théo}
Let $U:\mathcal P_2(\R) \rightarrow \R$ be smooth enough so that $U$ and its derivatives $\partial_\mu U$, $\partial_v \partial_\mu U$ and $\partial_\mu^2 U$ exist, are uniformly continuous and bounded. 
Almost surely, for each $t \in [0,T]$, we have:
\begin{align*}
\label{formule Ito}
U(\mu_t)=&U(\mu_0)+\int_0^1 \!\!\int_0^t \partial_\mu U(\mu_s)(y(u,s))\mathrm dy(u,s)\mathrm du
+\frac{1}{2} \int_0^1 \!\!\int_0^t \partial_v \partial_\mu U(\mu_s)(y(u,s))\frac{\mathrm ds}{m(u,s)}\mathrm du\\
&+ \frac{1}{2}\int_0^1\!\! \int_0^t \partial_\mu^2 U(\mu_s)(y(u,s),y(u,s))\mathrm ds\mathrm du,
\end{align*}
where $\displaystyle \int_0^1 \int_0^t \partial_\mu U(\mu_s)(y(u,s))\mathrm dy(u,s)\mathrm du$ is a square integrable continuous martingale with a quadratic variation process equal to $t\mapsto \displaystyle
\int_0^1 \!\!\int_0^t \left( \partial_\mu U(\mu_s)\right)^2 (y(u,s)) \mathrm ds \mathrm du$. 
\label{theo Ito}
\end{théo}
%Choosing in particular $U: \mu \mapsto \int_\R f\mathrm d\mu$ with $f\in \mathcal C^2_b (\R)$, in which case we have $\partial_\mu U(\mu)=f'$ for all $\mu \in \mathcal P_2(\R)$, we obtain again the result stated in~\cite{konarovskyi_modified_2015}, Corollary 1.1. 

\begin{rem}
\label{remark vR S}
Choose in particular $ U: \mu \mapsto V\left(\int_\R \alpha_1 \mathrm d\mu,\dots,\int_\R \alpha_m \mathrm d\mu  \right)=:V(\int \overrightarrow{\alpha}\mathrm d\mu)$, where $V \in \mathcal C^2(\R^m)$ and $\alpha_1,\dots,\alpha_m$ are bounded $\mathcal C^2(\R)$-functions, with bounded first and second-order derivatives. In this case, $\partial_\mu U(\mu)(v)=\sum_{i=1}^m \partial_i V \left(\int \overrightarrow{\alpha}  \mathrm d\mu \right) \alpha_i'(v)$ for all $\mu \in \mathcal P_2(\R)$ and $v\in \R$. 
Computing the second-order derivatives, we show that 
\[
U(\mu_t)-U(\mu_0) -\frac{1}{2}\int_0^t \mathcal L_1 U(\mu_s) \mathrm ds -\frac{1}{2}\int_0^t \mathcal L_2 U(\mu_s) \mathrm ds
\]
is a martingale with quadratic variation process $\int_0^t \int_0^1 \left(\sum_{i=1}^m \partial_i V \left(\int \overrightarrow{\alpha}\mathrm d\mu_s \right) \alpha_i'(y(u,s)) \right)^2 \mathrm du \mathrm ds$ and an operator $\mathcal L=\mathcal L_1+\mathcal L_2$ of the form $\mathcal L_1 U(\mu_s) := \sum_{i=1}^m \partial_i V\left( \int \overrightarrow{\alpha}\mathrm d\mu_s\right)\int_0^1 \frac{\alpha_i''(y(u,s))}{m(u,s)}\mathrm du $ and $\mathcal L_2 U(\mu_s) := \sum_{i,j=1}^m \partial_{i,j}^2 V\left( \int \overrightarrow{\alpha}\mathrm d\mu_s\right)\int_0^1 \alpha_i'(y(u,s))\alpha_j'(y(u,s))\mathrm du $. 

Remark that we have some restrictions on the domain of the generator $\mathcal L_1$. We know that for measures with finite support, $\int_0^1 \frac{\mathrm du}{m(u,s)}$ is finite and is equal to the cardinality of the support (see the Paragraph preceding Corollary~\ref{coro int N(t) finite}). The fact that the generator of the martingale problem is not defined on the whole Wasserstein space is related to the fact that the process $(\mu_t)_{t\in [0,T]}$ takes values, for every positive time $t$, on the space of measures with finite support.

We compare this result to Theorem~7.17 in~\cite{vrenessesturm09}. The generator of the martingale in the case of von Renesse and Sturm's Wasserstein diffusion is $\mathds L = \mathds L_1+\mathds L_2+\beta \mathds L_3$, with $\mathds L_1=\mathcal L_2$ and $\mathds L_3$ similar to $\mathcal L_1$ up to the lack of the mass function, whereas $\mathds L_2$, which is the part of the generator considering the gaps of the measure $\mu$, does not appear in our model.

\end{rem}

\bibliographystyle{amsplain}
\bibliography{refmathscinet}

\end{document}